\documentclass[11pt]{amsart}
\usepackage{graphicx, color}
\usepackage{amscd}
\usepackage{amsmath,empheq}
\usepackage{amsfonts}
\usepackage{amssymb}
\usepackage{mathrsfs}
\usepackage[utf8]{inputenc}
\usepackage[all]{xy}
\usepackage{multirow}

\usepackage{todonotes}
\usepackage{enumitem}

\newtheorem{theorem}{Theorem}[section]

\newtheorem{lemma}{Lemma}[section]

\theoremstyle{definition}

\newtheorem{remark}{Remark}[section]

\newcommand{\R}{\mathbb{R}}

\def\x#1{{\rm (\ref{#1})}}

\makeatletter
\@namedef{subjclassname@2020}{\textup{2020} Mathematics Subject Classification}
\makeatother

\title[Another look at quasilinear Schr\"{o}dinger equations]{Another look at quasilinear Schr\"{o}dinger equations with prescribed mass via dual method}

\author[J. Chen]{Jianhua Chen}
\address[Jianhua Chen]{School of Mathematics and Computer Sciences, Nanchang University,  Nanchang, Jiangxi 330031, P. R. China}
\email{\tt jianhuachen@ncu.edu.cn}

\author[V.D. R\u{a}dulescu]{Vicen\c tiu D. R\u{a}dulescu}
\address[Vicentiu D.  R\u{a}dulescu]{Faculty of Applied Mathematics, AGH University of Krak\'ow, 30-059 Krak\'{o}w, Poland  \&
	\& Brno University of Technology, Faculty of Electrical Engineering and Communication, Technick\'a 3058/10, Brno,
	61600, Czech Republic \&
	Department of Mathematics, University of Craiova, 200585 Craiova, Romania \& ``Simion Stoilow" Institute of Mathematics of the Romanian Academy, P.O. Box 1-764, 014700 Bucharest, Romania \& Department of Mathematics, Zhejiang Normal University,
321004	Jinhua, Zhejiang,  China}
\email{\tt radulescu@inf.ucv.ro}

\author[J. Sun]{Jijiang Sun}
\address[Jijiang Sun]{School of Mathematics and Computer Sciences, Nanchang University,  Nanchang, Jiangxi 330031, P. R. China}
\email{\tt sunjijiang2005@163.com}

\author[J. Zhang]{Jian Zhang}
\address[Jian Zhang]{College of Science, Hunan University of Technology and Business, Changsha, Hunan 410205, P. R. China}
\email{\tt zhangjian@hutb.edu.cn}


\begin{document}

\begin{abstract}
In this paper, we aim to study the existence of ground state  normalized solutions for the following quasilinear Schr\"{o}dinger equation
\begin{equation*}
\aligned
-\Delta u-\Delta(u^2)u=h(u)+\lambda u,\,\, x\in\R^N,
\endaligned
\end{equation*}
under the mass constraint
\begin{equation*}
\int_{\R^N}|u|^2\text{d}x=a,
\end{equation*}
where $N\geq2$, $a>0$ is a given mass, $\lambda$ is a Lagrange multiplier and $h$ is a nonlinear reaction term with some suitable conditions.
By employing a suitable transformation $u=f(v)$, we reformulate the original problem into the equivalent form
\begin{equation*}
-\Delta v =h(f(v))f'(v)+\lambda f(v)f'(v),\,\, x\in\R^N,
\end{equation*}
with prescribed mass
\begin{equation*}
\int_{\R^N}|f(v)|^2\text{d}x=a.
\end{equation*}
To address the challenge posed by the $L^2$-norm $\|f(v)\|^2_2$ not necessarily equaling
 $a$,  we introduce a novel stretching mapping:
\begin{equation*}
v_t(x):=f^{-1}(t^{N/2}f(v(tx))).
\end{equation*}
This construction, combined with a dual method  and some new analytical techniques, enables us to establish the following existence results:
  \begin{itemize}
\item[{\bf(P1)}]
 Existence of solutions via constrained minimization using dual method;
 \end{itemize}

  \begin{itemize}
\item[{\bf(P2)}]
Existence of ground state normalized solutions under general
 $L^2$-supercritical growth conditions, along with nonexistence results, analyzed via dual method;
 \end{itemize}

  \begin{itemize}
\item[{\bf(P3)}]
Existence of   normalized solutions under critical growth conditions, treated via dual method.
 \end{itemize}
Additionally,  we analyze the asymptotic behavior of the ground state energy obtained in {\bf(P2)}.
Our results extend and refine those of Colin-Jeanjean-Squassina [Nonlinearity 20: 1353-1385, 2010], of Jeanjean-Luo-Wang [J. Differ. Equ. 259: 3894-3928, 2015], of Li-Zou [Pacific J. Math. 322: 99-138, 2023], of Zhang-Li-Wang [Topol. Math. Nonl. Anal. 61: 465-489, 2023] and so on.

We believe that the methodology developed here can be adapted to study related problems concerning the existence of normalized solutions for quasilinear Schr\"{o}dinger equations via the dual method.
\end{abstract}

\maketitle

\section{Introduction}
This article is concerned with the following quasilinear Schr\"{o}dinger equation
\begin{equation}\label{p}
\aligned
-\Delta u-\Delta(u^2)u=h(u)+\lambda u,\,\, x\in\R^N,
\endaligned
\end{equation}
with the prescribed mass
\begin{equation*}
  \int_{\R^N}u^2\text{d}x=a,
\end{equation*}
where $\lambda$ is a Lagrange multiplier,  $a$ is a given mass, $h$ is a nonlinearity and $N\geq2$.
\par

It is well known that the existence of solitary wave solutions for the following quasilinear Schr\"{o}dinger equation has been considered
\begin{equation}\label{p2}
i\partial_t z=-\Delta z+W(x)z-k(x,|z|)-\Delta l(|z|^2)l'(|z|^2)z,
\end{equation}
where $z : \R\times\R^N\rightarrow\mathbb{C}$, $W : \R^N\rightarrow\R$ is a given potential, $l : \R\rightarrow\R$ and $k : \R^N\times\R\rightarrow\R$ are suitable functions. For various types of $l$, the quasilinear equation of the form \x{p} has been derived from models of several physical phenomenon. In particular, $l(s)=s$ was used for the superfluid film equation in fluid mechanics by Kurihara \cite{SK}. For more physical background, we can refer to \cite{BN,MF} and references therein.
\par

From the perspectives of mathematics and physics, the standing wave of equation \x{p2} has received increasing attention. Substituting $z(t,x)=\exp(-iE t)u(x)$ to \x{p2}, where $E\in\R$ and $u$ is a real function, \x{p2} can reduce to the
corresponding elliptic type equation (see \cite{SC1}):
\begin{equation}\label{p5}
-\Delta u-\lambda u-\Delta l(u^2)l'(u^2)u=h(u)\quad x\in\R^N.
\end{equation}
Let $l(s)=s$ in \x{p5},  the following form of quasilinear Schr\"{o}dinger equation can be obtained
\begin{equation}\label{p4}
-\Delta u-\lambda u-\Delta(u^2)u=h(u)\quad x\in\R^N.
\end{equation}
As for $\lambda$ being a fixed parameter, there are many papers focusing on the existence and concentration of  solutions for \x{p4} in the past few decades.
At present, there are mainly two methods for studying such problems. Please see \cite{CTC,JS,DMS,FS,HQZ,LW,LWW,LLW1,LLW2,AM2,JS1,SV,WZ,LLW3,LLW4,LLW5,LLW6,ZTAW} and their references therein.

\par

In this thesis, we investigate normalized solutions of the quasilinear Schr\"{o}dinger equation \x{p} with a prescribed
$L^2$-norm. The study of normalized solutions is strongly motivated by physical considerations, as the
$L^2$-norm plays a crucial role in determining the orbital stability or instability of solutions. This makes the existence of normalized solutions a problem of significant physical and mathematical interest.
From a variational perspective, we seek nontrivial solutions
$(u_a,\lambda_a)\in \mathcal{K}\times \R$ to \x{p} such that $(u_a,\lambda_a)$ satisfies $\|u_a\|_2^2=a$, where
 \begin{equation*}
\mathcal{K}:=\left\{u\in H^1(\R^N): \int_{\R^N}|\nabla u|^2u^2\text{d}x<+\infty\right\}.
 \end{equation*}
Notably, $\mathcal{K}$ is  not convex and  not even a vector space. So the usual min-max
techniques cannot be directly applied to this problem. In general, via variational method to study quaslinear Schr\"{o}dinger equation, we can deduce formally that the Euler-Lagrange functional is
\begin{equation*}
J_\lambda(u)=\frac{1}{2}\int_{\R^N}(1+2u^2)|\nabla u|^2\text{d}x-\frac{\lambda}{2}\int_{\R^N}u^2\text{d}x-\int_{\R^N}H(u)\text{d}x,
\end{equation*}
where $u\in\mathcal{K}$. It is easy to check that $u$ is a weak solution of \x{p} if and only if
\begin{equation*}
0=\langle J'_\lambda(u),\phi\rangle=\lim_{t\rightarrow0^+}\frac{J_\lambda(u+t\phi)-J_\lambda(u)}{t},
\end{equation*}
for any $\phi\in\mathcal{C}^\infty_0(\R^N)$.
\par

In contrast to  semilinear Schr\"{o}dinger equation, quasilinear Schr\"{o}dinger equation possesses an additional quasilinear term
\begin{equation*}
\mathcal{Z}(u)=\int_{\R^N}u^2|\nabla u|^2\text{d}x,
\end{equation*}
which  fails to be differentiable on the space $\mathcal{K}$ when $N\geq 2$. This lack of differentiability presents a significant challenge in the variational analysis of such problems. To overcome this difficulty, two primary approaches have been developed in the literature:
  \begin{itemize}
\item[]{\bf Dual Method (e.g., \cite{CL1,LWW}):}
This technique transforms the quasilinear problem into an equivalent semilinear problem through a suitable variable substitution. The resulting semilinear formulation allows the application of standard variational methods.
 \end{itemize}

 \begin{itemize}
\item[]{\bf Perturbation Method (e.g., \cite{LLW2}):}
This approach restores differentiability of $\mathcal{Z}$ by introducing a perturbation term to the energy functional
$J_\lambda$. The analysis is then carried out on a carefully chosen smaller function space where the perturbed functional becomes differentiable.
 \end{itemize}

\par
Although the existence and multiplicity of normalized solutions for the quasilinear Schr\"{o}di-inger equation \x{p} has been investigated using various methods, research in this area remains far from complete. Below, we summarize key results on normalized solutions for \x{p}.
The earliest results can be traced back to the works of \cite{CJS, JLW}, which primarily examined the minimization problem:
\begin{equation*}
m(a)=\inf_{u\in\mathcal{\hat{S}}_a}J_0(u)
\end{equation*}
where
\begin{equation*}
J_0(u)=\frac{1}{2}\int_{\R^N}(1+2u^2)|\nabla u|^2\text{d}x -\int_{\R^N}|u|^p\text{d}x
\end{equation*}
 with $2<p\leq4+\frac{4}{N}$ and
\begin{equation*}
\mathcal{\hat{S}}_a=\left\{u\in\mathcal{K}: \int_{\R^N}|u|^2\text{d}x=a\right\}.
\end{equation*}
Using the constrained minimization method, they uncovered a rich structure for $m(a)$ depending on $a>0$ and $p$. Notably, they proved that $m(a)>-\infty$
for $2<p<4+\frac{4}{N}$. Building on  \cite{CJS}, Jeanjean and Luo \cite{LJ} further demonstrated that $m(a)=-\infty$ for $p>4+\frac{4}{N}$. In the $L^2$-critical case
$p=4+\frac{4}{N}$, they established that:
\begin{itemize}
\item  No minimizer exists for any $a>0$

\item There exists a threshold $a_N \in(0,+\infty)$ such that $m(a)=0$ for $a\in(0,a_N)$ and $m(a)=-\infty$ for all $a>a_N$.
\end{itemize}
Consequently,  $p=4+\frac{4}{N}$ is commonly referred to as the $L^2$-mass critical exponent for quasilinear Schr\"{o}dinger equations. However, the precise expression for
$a_N $ and the exact value of $m(a)$ remain unknown. Subsequently, For the $L^2$-critical case  $p=4+\frac{4}{N}$, Ye and Yu \cite{YY} successfully derived a threshold value for $a_N $ using an alternative approach and established the existence of ground state normalized solutions for $N\leq3$. However, their method did not extend to dimensions $N\geq4$. Later, Li and Zou \cite{LZ} proved the existence of ground state normalized solutions and infinitely many normalized solutions for \x{p} when $N=2,3$ and $4+\frac{ 4}{N}<p<2^*$.  They also identified a critical threshold  $a_*>0$ such that
\begin{itemize}

\item  For $N\leq3$ and $p=4+\frac{4}{N }$, there exists a ground state solutions when $a>a^*$.

\item  For $N\geq4$ and $p=4+\frac{4}{N}$, a ground state exists when $a^*<a<\left(\frac{N-2}{N-2-(4/N)}\right)^{N/2}a^*$.
\end{itemize}
More recently, Jeanjean, Zhang, and Zhong \cite{JZZ} revisited the same constrained minimization problem, extending the analysis from $1\leq N\leq3$ to arbitrary
$N\geq1$.

\par

 To establish a unified framework for studying normalized solutions, Zhang, Li, and Wang \cite{ZLW} recently pioneered the application of the dual method to this problem. Their approach builds upon the variable transformation technique introduced by Liu-Wang-Wang \cite{LWW} and Colin-Jeanjean \cite{CL1}, which converts the quasilinear Schr\"{o}dinger equation \x{p} into a semilinear problem of the form:
\begin{equation}\label{fff2}
-\Delta v=h(f(v))f'(v)+\lambda f(v)f'(v),\ \ \ x\in\R^N,
\end{equation}
with the prescribed mass
\begin{equation*}
\int_{\R^2}|f(v)|^2\text{d}x=a,
\end{equation*}
where $v=f^{-1}(u)$,  and $f$ is defined by
\begin{equation*}
f'(t)=\frac{1}{(1+2f^2(t))^{\frac{1}{2}}}\,\,\text{on}\,\,[0,\infty)\,\,\text{and}\,\,f(t)=-f(-t)\,\,\text{on}(-\infty,0].
\end{equation*}
The corresponding energy functional of \x{fff2} is given by
\begin{equation*}
\aligned
\Psi(v)&=\frac{1}{2}\int_{\R^N}|\nabla v|^2\text{d}x-\int_{\R^N}H(f(v))\text{d}x.
\endaligned
\end{equation*}
Obviously, the functional $\Psi$ is well defined in $H^1(\R^N)$ and of class $\mathcal{C}^1(H^1(\R^N),\R)$. Under this transformation, the original constraint set $\mathcal{\hat{S}}_a$ reduce to
\begin{equation*}
\mathcal{S}_a=\left\{v\in H^1(\R^N): \int_{\R^N}|f(v)|^2\text{d}x=a\right\}.
\end{equation*}
{\bf Key Challenges of the Dual Method:}
 \begin{itemize}
\item  While this reformulation simplifies the problem structurally, it introduces a significant {\bf mathematical obstacle:} the constraint
$\mathcal{S}_a$ {\bf loses $L^2$-mass invariance}. In variational problems, scaling invariance plays a crucial role in establishing the existence of normalized solutions, and its absence here complicates the analysis. This lack of invariance is the primary difficulty when applying dual methods to normalized solution problems, necessitating alternative techniques to overcome the resulting analytical limitations.
\end{itemize}

\par
 Building upon the aforementioned studies, we employ the dual method to investigate the existence of ground state normalized solutions for the quasilinear Schr\"{o}dinger equation \x{p} under various parameter conditions. Our approach leverages the following Gagliardo-Nirenberg inequality \cite{MIW}, which will play a fundamental role in our subsequent analysis:
\begin{lemma}\label{lemma1.0}\cite{MIW}
 For any $N\geq2$ and $s\in(2,2^*)$, there is a constant $C_{N,s}>0$
 depending on $N$ and $s$ such that
 \begin{equation*}
   \int_{\R^N}|u|^s\text{d}x\leq C_{N,s}^s\left( \int_{\R^N}|u|^2\text{d}x\right)^{\frac{2s-N(s-2)}{4}}\left( \int_{\R^N}|\nabla u|^2\text{d}x\right)^{\frac{N(s-2)}{4}},\ \ \forall\ u\in H^1(\R^N).
 \end{equation*}
\end{lemma}

\par

Now we give another version of Gagliardo-Nirenberg inequality. At first, we give the following notions. Let
\begin{equation*}
\mathcal{E}:=\left\{u\in L^1(\R^N): \nabla u\in L^2(\R^N)\right\}
\end{equation*}
with the norm
\begin{equation*}
\|u\|_\mathcal{E}=\|\nabla u\|_2+\|u\|_1.
\end{equation*}
\begin{lemma}\label{lemma1.0a}\cite{MA}
 For any $N\geq2$ and $t\in(2,22^*)$, there is a constant $C_{N,t}>0$
 depending on $N$ and $t$ such that
 \begin{equation*}
   \int_{\R^N}|u|^{\frac{t}{2}}\text{d}x\leq C_{N,t}^t\left( \int_{\R^N}|u|\text{d}x\right)^{\frac{4N-t(N-2)}{2(N+2)}}\left( \int_{\R^N}|\nabla u|^2\text{d}x\right)^{\frac{N(t-2)}{2(N+2)}},\ \ \forall\ u\in\mathcal{E}.
 \end{equation*}
\end{lemma}

 \begin{remark}
$(i)$ In the Gagliardo-Nirenberg inequalities stated above, the constants  $C_{N,t}$ and $C_{N,s}$ are distinct. However, since they serve the same role as multiplicative constants in our analysis, we will not distinguish between them in subsequent arguments.

$(ii)$ Our analysis employs both versions of the Gagliardo-Nirenberg inequality:
 \begin{itemize}
\item  For dimensions  $2\leq N\leq3$, we only apply Lemma  \ref{lemma1.0}.

\item  For dimensions   $N\geq2$, we utilize Lemma \ref{lemma1.0a}.
\end{itemize}
 \end{remark}

\par

As observed in \cite{ZLW}, the nonlinear nature of the transformation $f$ prevents the application of the conventional scaling technique $v^t:=t^{N/2}v(tx)$, since this approach fails to preserve the mass constraint. Prior to their work, these results could only be established indirectly through constraint equivalence relations (see Lemma 2.5 and Remark 2.6 of \cite{ZLW}). In contrast, our approach employs a novel direct method based on the modified scaling transformation:
\begin{equation}\label{vt}
v_t(x):=f^{-1}(t^{N/2}f(v(tx))),
\end{equation}
which enables us to not only recover their results but also establish additional, more comprehensive conclusions.

\par

To state the first result, let $\Phi_p(v):=\Psi(v)$ with $h(f(v))=|f(v)|^{p-2}f(v)$. Let
\begin{equation*}
\digamma(a):=\inf\{\Phi_p(v): v\in \mathcal{S}_a\}.
\end{equation*}
Now, we show the first results of the paper.

\begin{theorem}\label{theoremnew}
Assume that $N\geq2$ and $p\in(2,22^*)$.

(I) Let
\begin{equation*}
a_*:=\inf\left\{a>0\big|\digamma(a)<0\right\}
\end{equation*}
and $ 2<p<4+\frac{4}{N} $.
Then

$(i)$ if $2<p<4+\frac{4}{N}$, then $-\infty<\digamma(a)\leq0$ for all $a>0$;

$(ii)$ if $2<p<2+\frac{4}{N}$, then $a_*=0$ and  $\digamma(a)<0$ is achieved for all $a>0$;

$(iii)$ if $2+\frac{4}{N}\leq p<4+\frac{4}{N}$, then  $a_*>0$ and the following conclusions hold:

 \begin{itemize}
 \item $\digamma(a)=0$ is not attained for all $a\in(0,a_*)$;

\item $\digamma(a_*)=0$ is attained for all $2+\frac{4}{N}< p<4+\frac{4}{N}$;

\item $\digamma(a_*)=0$ is not attained for  $p=2+\frac{4}{N}$;

\item $\digamma(a)<0$ is attained for all $a\in(a_*,+\infty)$.
  \end{itemize}

(II) Let $p=4+\frac{4}{N}$ and
\begin{equation*}
a^*:=\left[\frac{2(N+1)Q_{4+\frac{4}{N}}}{N}\right]^{\frac{N}{4}},
\end{equation*}
where
\begin{equation*}
Q_{4+\frac{4}{N}}:=\inf_{v\in H^1(\R^N)\backslash\{0\}}\frac{\left(\int_{\R^N}|f(v)|^2\text{d}x\right)^{\frac{2}{N}}\int_{\R^N}\frac{2f^2(v)}{1+2f^2(v)}|\nabla v |^2\text{d}x}{\int_{\R^N}|f(v)|^{4+\frac{4}{N}}\text{d}x}.
\end{equation*}
Then $Q_{4+\frac{4}{N}}>0$ and
 \begin{itemize}
 \item   $\digamma(a)=0$ for all $a\in(0,a^*]$;

\item  $\digamma(a)=-\infty$ for all $a>a^*$;

\item $\digamma(a)$ is not attained for all $a>0$.
  \end{itemize}

(III) Let  $4+\frac{4}{N}<p<22^*$.  Then $\digamma(a)=-\infty$ for all $a>0$.
\end{theorem}

 \begin{remark}
 (i)  In order to facilitate the reader's understanding, the results of Theorem \ref{theoremnew}  are summarized as  the following table:
 \begin{center}
\begin{tabular}{|c|c|c|c|}
\hline
$p$ & $a$ & $\digamma(a)$ &  Existence of the minimizer \\ 
\hline
$(2,2+\frac{4}{N})$ & $a>0$ & $-\infty<\digamma(a)<0$ & YES \\
\hline
\multirow{2}{*}{$[2+\frac{4}{N},4+\frac{4}{N})$} & $a\in(0,a_*)$ & $\digamma(a)=0$ & NO \\
\cline{2-4}
 & $a\in(a_*,+\infty)$ & $-\infty<\digamma(a)<0$ & YES \\
\hline
$(2+\frac{4}{N},4+\frac{4}{N})$ & $a=a_*$ & $\digamma(a)=0$ & YES \\
\hline
$ 2+\frac{4}{N} $ & $a=a_*$ & $\digamma(a)=0$ & NO \\
\hline
\multirow{2}{*}{$ 4+\frac{4}{N} $} & $a\in(0,a^*]$ & $\digamma(a)=0$ & NO \\
\cline{2-4}
 & $a\in(a^*,+\infty)$ & $\digamma(a)=-\infty$ & NO \\
\hline
$(4+\frac{4}{N},22^*)$ & $a>0$ & $\digamma(a)=-\infty$ & NO \\
\hline
\end{tabular}
\end{center}

(ii) While the results stated in Theorem \ref{theoremnew} were previously established in \cite{CJS,LJ,YY}, our proof employs the dual method with a novel stretching transformation. This alternative approach not only provides independent verification of these results, but more importantly demonstrates the efficacy of our modified scaling technique for handling the nonlinear constraint.

(iii) Compared with  Theorem 1.2 obtained by \cite{ZLW}, Theorem \ref{theoremnew} of this paper is much richer. Moreover, since $f$ is nonlinear, the scaling method
(see   \cite{CJS} of Lemma 4.5)
\begin{equation*}
v^t:=t^{N/2}v(tx)
\end{equation*}
 does not come into play here, which keeps the constraint
unchanged. So the proof of Theorem 1.2 in \cite{ZLW} need to use Lemma 4.5 in \cite{CJS}. But  the proof of Theorem \ref{theoremnew} in this article does not require the use of Lemma 4.5 in \cite{CJS}, and our  method is more direct.

 \end{remark}

 \vskip1mm
  \par
Now back to the above papers \cite{JLW,YY,ZLW},  the authors have primarily focused on the power-type nonlinearity   $h(u)=|u|^{p-2}u$ in the $L^2$-subcritical($2<p<4+\frac{4}{N}$) and $L^2$-critical($p=4+\frac{4}{N}$) for $N\geq2$. However, the
 $L^2$-supercritical regime remains less explored. In \cite{LZ}, the authors made initial progress by employing perturbation methods from \cite{LLW2} to establish the existence of radial ground state normalized solutions for the $L^2$-supercritical case $4+\frac{4}{N}<p\leq2^*$
  in dimensions $2\leq N\leq3$.
This leads us to naturally formulate the following important open questions regarding the
$L^2$-supercritical and critical exponential growth cases:

\begin{itemize}
\item[{\bf(Q1)}]Can we establish the existence of ground state normalized solutions for \eqref{p} with more general mass-supercritical nonlinearities (beyond pure power-type) in dimensions $2\leq N\leq3$ using the dual method?
\end{itemize}

\begin{itemize}
\item[{\bf(Q2)}] Does \eqref{p} admit the existence of normalized solutions when
$h$ exhibits critical exponential growth in dimensions $2\leq N\leq3$
 and can this be proved via the dual method?
\end{itemize}

 Building upon the aforementioned open questions, one of the primary objectives of this work is to provide a complete resolution to problem {\bf(Q1)}. Prior to presenting our main results, we first introduce the following fundamental assumptions on the nonlinearity $h$:
\begin{itemize}
\item[
$(h_1)$] $h\in \mathcal{C}(\R,\R)$, $h(t)=o(t)$ as $t\rightarrow0$.
\end{itemize}

\begin{itemize}
\item[
$(h_2)$]  there exist $\mu_1,\mu_{2}\in\R$ satisfying
\begin{displaymath}
\bigg\{\begin{array}{ll}
4+\frac{4}{N}<\mu_1\leq\mu_2<2^*,& \textrm{if $N= 3$},\\
4+\frac{4}{N}<\mu_1\leq\mu_2<+\infty,& \textrm{if $N=2$},
\end{array}
\end{displaymath}
such that
\begin{equation*}
0<\mu_1H(t)\leq h(t)t\leq\mu_2H(t),\ \text{for all}\   t\in\R\backslash\{0\}.
\end{equation*}
 \end{itemize}

\begin{itemize}
\item[
$(h_3)$] $\frac{h(t)t-2H(t)}{|t|^{3+\frac{4}{N}}t}$  is nondecreasing on $(-\infty,0)$ and $(0,+\infty)$.
\end{itemize}

\begin{remark}\label{remark1.1}
It follows from $(h_2)$ that
\begin{displaymath}
\bigg\{\begin{array}{ll}
|t|^{\mu_2}H(s)\leq H(st)\leq|t|^{\mu_1}H(s),& \textrm{if $t\in[0,1]$},\\
|t|^{\mu_1}H(s)\leq H(st)\leq|t|^{\mu_2}H(s),& \textrm{if $t\in[1,+\infty)$},
\end{array}
\end{displaymath}
for all $s\in\R$. Furthermore, it yields that for any $\varepsilon>0$, there exists $C_\varepsilon>0$ such that
\begin{equation}\label{o1}
\max\{h(t)t,H(t)\}\leq\varepsilon|t|^{4+\frac{4}{N}}+C_\varepsilon|t|^{\mu_2}.
\end{equation}
\end{remark}

Now, we show two results of this paper as follows:

\begin{theorem}\label{theorem1}
Assume that $(h_1)$-$(h_3)$ and $2\leq N\leq3$ hold.
Then for any $a>0$, problem
\x{p} has a radial normalized solution $(u=f(v),\lambda_a)\in\mathcal{ \hat{S }}_a\times(-\infty,0)$ such that
\begin{equation*}
\Psi(v)=\sigma(a):=\inf_{v\in\mathcal{N}_a}\Psi(v),
\end{equation*}
where \begin{equation*}
\mathcal{N}_a:=\left\{v\in \mathcal{S}_a : \frac{\text{d}\Psi(v_t)}{\text{d}t}\bigg|_{t=1}=0 \right\}.
\end{equation*}
Furthermore,
$
\lim\limits_{a\rightarrow0^+}\sigma(a)=+\infty
$
and $
\lim\limits_{a\rightarrow+\infty}\sigma(a)=0
.$
\end{theorem}

\begin{theorem}\label{theoremA}
Assume that $(h_1)$-$(h_3)$ and $\mu_1\geq 22^*$ in $(h_2)$ and $N\geq2$.
Then for any $a>0$, problem
\x{p} has no solution in $ \mathcal{\hat{S}}_a\times(-\infty,0)$.
\end{theorem}

\begin{remark}\label{remark1.2}
(i)  While the studies in  \cite{JLW,YY,ZLW} primarily addressed the $L^2$-subcritical and critical cases, our work makes significant progress by establishing complete results for the $L^2$-supercritical quasilinear Schr\"{o}dinger equation, thereby complementing and extending their findings.

(ii) To the best of our knowledge, the work \cite{LZ} represents the only existing treatment of the $L^2$-supercritical, focusing specifically on power-type nonlinearities $h(u)=|u|^{p-2}u$ with $4+\frac{4}{N}<p<2^*$ in dimensions $2\leq N\leq3$. Their approach, based on perturbation methods, examined the minimization problem:
\begin{equation*}
d(a):=\inf\left\{J_0(u): \|u\|^2_2=a,\  J'_0|_{\mathcal{\hat{S}}_a}(u)=0 \right\}
\end{equation*}
and established the attainability of $d(a)$. However, this perturbation framework cannot capture Nehari-type or Pohozaev-type ground state solutions.
In contrast, our dual method approach provides several key advantages:
 \begin{itemize}
\item  direct construction of Pohozaev-type ground state normalized solutions;

\item   elimination of the need to study the asymptotic behavior of  the perturbation functional;

\item more streamlined proof technique;

\item broader applicability to general nonlinearities beyond power-type.
\end{itemize}
\end{remark}

 \vskip1mm
  \par
We now present our third main result, which establishes the existence of a local minimizer for the critical exponential growth case. This somewhat answer question {\bf (Q2)}.

\par

At first, we examine the case involving combined power-type nonlinearities with Sobolev critical exponent:  $h(u)=|u|^{p-2}u+|u|^{22^*-2}u$ for $2<p<2+\frac{4}{N}$ and $N=3$. The result is established as follows:
\begin{theorem}\label{theoremeX}
Assume that $N=3$ and
\ $2<p<2+\frac{4}{N}$.
Then there exist $\tilde{a}_0>0$ and $t^*_0>0$ such that for any $a\in(0, \tilde{a}_0)$, problem
\x{p} has a normalized solution $(\hat{u}=f(\hat{\upsilon}),\lambda_a)\in \mathcal{\hat{S}}_a\times(-\infty,0)$ such that
\begin{equation*}
\Psi(\hat{\upsilon})=m_0(a):=\inf_{v\in \mathcal{S}_a\cap\Lambda_{t^*_0}}\Psi(v),
\end{equation*}
where \begin{equation*}
\Lambda_{\ell}:=\left\{ v\in H^1(\R^3): \|\nabla v\|^2_2<\ell\right\}.
\end{equation*}
\end{theorem}

\par

Finally, we  turn to the two-dimensional case, where the nonlinearity exhibits critical exponential growth in the sense of Trudinger-Moser. This needs some unique analytical challenges due to the critical nature of the exponential nonlinearity.
A crucial tool for handling such nonlinearities is the Trudinger-Moser inequality, which plays a fundamental role in our compactness analysis. Before proceeding, we recall the following version of this inequality:
\begin{lemma}\label{lemma1.1}\cite{CAO,JS0}
 (i) If $\beta>0$ and $u\in H^1(\R^2)$, then
\begin{equation*}
\int_{\R^2}\left(e^{\beta u^2}-1\right)\text{d}x<\infty.
\end{equation*}

(ii)  If $u\in H^1(\R^2)$, $\|\nabla u\|^2_2\leq1$, $\|u\|_2\leq M<\infty$ and $\beta<4\pi$, then there exists a constant $C(M,\beta)>0$, which only depends on $M$ and $\beta$, such that
\begin{equation*}
\int_{\R^2}\left(e^{\beta u^2}-1\right)\text{d}x\leq C(M,\beta).
\end{equation*}
\end{lemma}

\par

Prior to presenting our last result, we introduce the following novel characterization of critical exponential growth within the Trudinger-Moser framework, specifically adapted for quasilinear Schr\"{o}dinger equations:
\begin{itemize}
\item[
$(\mathcal{H}_1)$] $h\in \mathcal{C}(\R^2,\R)$ and there exists $\zeta_0>0$ such that
$$
\lim_{|t|\rightarrow+\infty}\frac{|h(t)|}{\exp{(\zeta t^4)}}=0,\ \ \text{for all}\ \ \ \zeta>\zeta_0
$$
and
$$
\lim_{|t|\rightarrow+\infty}\frac{|h(t)|}{\exp{(\zeta t^4)}}=+\infty,\ \ \text{for all}\ \ \ \zeta<\zeta_0.
$$
\end{itemize}

\par

Furthermore, we also give some another assumptions on $h$:
\begin{itemize}
\item[
$(H_1)$] $h(t)=o(t^5)$ as $t\rightarrow0$;
\end{itemize}

\begin{itemize}
\item[$(H_2)$] there exist $p>6$ and
\begin{equation*}
 \aligned
\xi>\xi^*&=\max\left\{ \frac{1}{p}\left[\frac{2\xi_0(p-4)m^*(a)}{\pi(p-6)}\right]^{\frac{p-4}{2}},\frac{\int_{\R^2}|\nabla \hat{v}_a|^2\text{d}x}{2\int_{\R^2}|f(\hat{v}_a)|^p\text{d}x}\right\}
 \endaligned
\end{equation*}
such that $$
H(t)\geq\xi|t|^p,\ \ \ \forall\, t\in\R,
$$
where $\hat{v}_a$  and $m^*(a)$ are given in Section 4;
 \end{itemize}

\begin{itemize}
\item[
$(H_3)$] for all $t\in\R\backslash\{0\}$, $0<8H(t)\leq th(t)$;
\end{itemize}

\begin{itemize}
\item[$(H_4)$]
there exist $L_0>0$ and $t_0>0$ such that
$$
H(t)\leq L_0|h(t)|,\ \ \ |t|\geq t_0.
$$
\end{itemize}

\begin{theorem}\label{theorem2}
Assume that $(\mathcal{H}_1)$, $(H_1)$-$(H_4)$, $(h_3)$ and $N=2$  hold.
Then for all $a>0$, problem
\x{p}   has a radial normalized solution $(u_0=f(v_0),\lambda_a)\in  \mathcal{\hat{S}}_a\times(-\infty,0)$ such that
\begin{equation*}
\Psi(v_0)=\sigma^*(a):=\inf_{v\in\mathcal{N}_a}\Psi(v),
\end{equation*}
where \begin{equation*}
\mathcal{N}_a:=\left\{v\in \mathcal{S}_a : \frac{\text{d}\Psi(v_t)}{\text{d}t}\bigg|_{t=1}=0 \right\}.
\end{equation*}
\end{theorem}

\begin{remark}
$(i)$ The technical conditions $(H_2)$ and $(H_4)$ play a crucial role in establishing an upper bound for the minimax energy level. This bound is essential for proving the compactness of Palais-Smale sequences in our variational framework.
\par

$(ii)$ While the nonlinearity $h$ exhibits critical exponential growth in $\R^2$ from the Trudinger-Moser perspective, it simultaneously demonstrates $L^2$-supercritical behavior in terms of its mass dependence. This dual characteristic presents unique analytical challenges.
\par

$(iii)$ Existing works \cite{DMS1,DMS2,JS,JS1,AMI1} on planar quasilinear Schrödinger equations with critical exponential growth have exclusively focused on unconstrained problems. To the best of our knowledge, our work represents the first systematic treatment of the prescribed mass problem for quasilinear Schr\"{o}dinger equations with critical exponential nonlinearities.
\end{remark}

\par

The study of quasilinear Schr\"{o}dinger equations with prescribed mass presents several profound mathematical challenges, which explains the limited existing literature on this subject. The principal difficulties we must address include:
  \begin{itemize}
\item[{\bf(A)}]{\bf Functional space limitations:}

The quasilinear term $$\int_{\R^2}\left(1+u^2\right)|\nabla u|^2\text{d}x$$ is no meaningful in $H^1(\R^N)$, creating fundamental obstacles for direct variational approaches in $H^1(\R^N)$.
 \end{itemize}

 \begin{itemize}
\item[{\bf(B) }] {\bf Mass constraint preservation:}

The $L^2$-norm $\|f(v)\|^2_2$ may be equal to a given number and  requiring novel analytical techniques;
 \end{itemize}

 \begin{itemize}
\item[{\bf(C)}] {\bf Variational framework complications:}

The presence of the undetermined Lagrange multiplier $\lambda$ invalidates conventional approaches using the Nehari manifold method.
 \end{itemize}

 \begin{itemize}
\item[{\bf(D)}] {\bf Energy functional behavior:}

 The combination of $L^2$-supercritical growth and planar critical exponential growth in
$h$ may lead to energy functionals that are unbounded below on the constraint set $\mathcal{S}_a$.

 \end{itemize}

 \begin{itemize}
\item[{\bf(E)}] {\bf Compactness issues:}

The critical exponential growth in two dimensions possesses significant challenges for establishing compactness properties on the constrained set $\mathcal{S}_a$.
 \end{itemize}

\par

To overcome the difficulty {\bf(A)}, we mainly use the change of variables and transform quasilinear problem into a semilinear problem.  By this method,  the corresponding energy functional for \x{p}
 is well defined in $H^1(\R^N)$, which enables us to use minimax methods to study the existence of standing wave solutions.

\par

As for  the difficulty {\bf(B)}, by changing of variables,
it is worth mentioning that the dual function $f$ being nonlinear function, the usual stretching:
\begin{equation*}
\ddot{v}_t(x):=t^{N/2}v(tx)
\end{equation*}
 for the  functional \x{f2}, which does not ensure $L^2$-norm unchanged, is not applied.
 To surmount  this obstacle, we used {\bf a new stretching mapping} \x{vt}
to  guarantee the unchanged of $L^2$-norm. Note that the similar stretching transformation
\begin{equation*}
\tilde{v}_t:=f^{-1}(tf(v(t^{-1}x)))
\end{equation*}
is first introduced by \cite{WW}, and it was used to deal with  unconstraint problems  in \cite{CTC1,CTC2}.

\par

As pointed by the difficulties {\bf(C)} and {\bf(D)}, the constrained energy functional is unbounded from below on $\mathcal{\hat{S}}_a$ for $L^2$-supercritical conditions on $h$ and $\lambda$
is a unknown  Lagrange multiplier. Hence the  constraint set $\mathcal{\hat{S}}_a$ and Nehari manifold method are no longer applicable for handling such problems. To overcome these obstacles,  by the stretching $v_t$,  we define a new  constraint set
\begin{equation*}
\mathcal{S}_a=\left\{v\in H^1(\R^N): \int_{\R^N}|f(v)|^2\text{d}x=a\right\}
\end{equation*}
and a new
 $L^2$-Pohazaev manifold
\begin{equation*}
\mathcal{N}_a:=\left\{v\in \mathcal{S}_a : \frac{\text{d}\Psi(v_t)}{\text{d}t}\bigg|_{t=1}=0 \right\},
\end{equation*}
which  can be used to prove that $\Psi$ is bounded from below on $\mathcal{N}_a$.
Some ideas of the proof mainly come from \cite{CTC2,Chen1,JSL,Jean}.

\par

As for the difficulty {\bf(E)},
by using a critical point obtained for $\Phi_p$ with $p>6$ as test function, we can get the mini-max level under a threshold value
and further achieve the compactness of the minimizing sequence. Note that our method for estimating mini-max level is different from \cite{Chen2}.

\begin{remark}
$(i)$ Compared with existing results in  \cite{JLW,YY,LZ,ZLW} our work establishes the first existence results for critical quasilinear Schr\"{o}dinger equations using variational methods. This represents a significant extension of previous work limited to subcritical cases.
\par
$(ii)$
In view of \cite{JLW,YY,ZLW,LZ}, a natural question arises regarding the use of change of variables for mountain pass type solutions. Our analysis reveals that:
  \begin{itemize}
\item  The complex nonlinear stretching $v_t(x):=f^{-1}(t^{N/2}f(v(tx)))$
 breaks key homogeneity properties;

 \item The exponential terms prevent derivation of crucial relations between original and Pohozaev functionals;

  \item These technical obstacles currently limit our ability to construct bounded Palais-Smale-Pohozaev sequences through mountain pass geometry.
 \end{itemize}
 Nonetheless, we believe that  our method  and our strategies may be helpful in studying normalized solutions of quasilinear Schr\"{o}dinger equations.
 Our work suggests several important directions for future research:
\par

{\bf(OP1)} Existence of mountain pass type normalized solutions via dual methods:
  \begin{itemize}
\item Can the current technical limitations be overcome?

\item What new tools would be required?
 \end{itemize}

 \par

{\bf(OP2)} Extension to higher dimensions ($N\geq4$) with $L^2$-supercritical nonlinearities ($4+4/N<p<22^*$) via dual method:
  \begin{itemize}
\item How does dimension affect the variational structure?

\item What modifications to the dual method are needed?
 \end{itemize}

\par

{\bf(OP3)} Treatment of Sobolev critical cases  $22^*$  in dimensions $N\geq3$ via dual method:
  \begin{itemize}
\item  Can the critical exponent be handled while preserving mass constraints?

\item How to address the additional compactness challenges?
 \end{itemize}
While our current work resolves several fundamental questions, the study of normalized solutions for quasilinear Schr\"{o}dinger equations remains rich with open problems.
 \end{remark}

 \vskip2mm
  \par
The remainder part of the paper is organized as follows. In section 2, we give some useul preliminary results. In section 3, we will prove Theorem  \ref{theoremnew}.
In section 4, we will prove Theorems \ref{theorem1} and \ref{theoremA}.
In section 5, we prove the existence of normalized  solutions obtained by Theorem \ref{theoremeX}.
In section 6, we shall prove Theorem \ref{theorem2}.
  \vskip2mm
  \par\noindent

  \vskip2mm
  \par\noindent
{\bf Notation:} Throughout this paper, the notations we need to use are as follows:

\quad $\bullet$ $\mathcal{X}$ is a Banach Space and $\mathcal{X}^*$ denotes the dual space of $\mathcal{X}$.

\quad $\bullet$ For any $z\in\R^N$ and $R>0$, $B_R(z):=\left\{x\in\R^N:|x-z|<R\right\}$.

\quad  $\bullet$
$\mathbb{N}$  denotes the positive
  integer set.

\quad $\bullet$ $L^q(\R^N)$ denotes the Lebesgue space with the norm
\begin{equation*}
\|u\|_q=\left(\int_{\R^N}|u|^q \text{d}x\right)^\frac{1}{q},
\end{equation*}
 for $q\in[1,+\infty)$.

\quad $\bullet$
\begin{displaymath}
2^*=\left\{
\begin{array}{ll}
\frac{2N}{N-2}, & \textrm{if $N\geq3$},\\
+\infty,& \textrm{if $N=1,2$}.
\end{array} \right.\\
\end{displaymath}

\quad $\bullet$ $C$ denote various positive constants which may vary from line to line.

\quad $\bullet$
 Let $$
H^{1}(\R^N)=\left\{u\in L^{2}(\R^N) : \nabla u\in L^2(\R^N)\right\}$$
with the norm $$\|u\|=\left(\int_{\R^N}(|\nabla u|^2+u^2)\text{d}x\right)^{\frac{1}{2}}.$$

\quad $\bullet$
Define
$$
H^{1}_r(\R^N)=\left\{u\in H^{1}(\R^N) : u(x)=u(|x|)\ \ \text{a.e. in}\ \ \R^N\right\}.$$

\section{Preliminary results}
In this section, we want to apply an argument developed by Liu-Wang-Wang \cite{LWW} and Colin-Jeanjean \cite{CL1}
to study the existence of ground state normalized solutions. Now, we need to recall some definitions again and collect these as follows. At first, we make the change of variables by $v=f^{-1}(u)$, where $f$ is defined by
\begin{equation}\label{h}
f'(t)=\frac{1}{(1+2f^2(t))^{\frac{1}{2}}}\,\,\text{on}\,\,[0,\infty)\,\,\text{and}\,\,f(t)=-f(-t)\,\,\text{on}(-\infty,0],
\end{equation}
and then equation \x{p} in form can be transformed into
\begin{equation}\label{f1}
-\Delta v=h(f(v))f'(v)+\lambda f(v)f'(v),\ \ \ x\in\R^N,
\end{equation}
with the prescribed mass
\begin{equation*}
\int_{\R^N}|f(v)|^2=a.
\end{equation*}
From \x{h} and \x{f1}, the normalized solutions of \x{p} can be obtained by looking for critical points for the energy functional  given by
\begin{equation}\label{f2}
\aligned
\Psi(v)&=\frac{1}{2}\int_{\R^N}|\nabla v|^2\text{d}x- \int_{\R^N}H(f(v))\text{d}x,
\endaligned
\end{equation}
on the constraint
\begin{equation*}
\mathcal{S}_a:=\left\{v\in H^1(\R^N): \int_{\R^N}|f(v)|^2\text{d}x=a\right\}.
\end{equation*}

It is easy to check  that $\Psi\in \mathcal{C}^1(H^1(\R^N),\R)$. We also know that if $v$ is a critical point of the functional $\Psi$ on $\mathcal{S}_a$, then $u=f(v)$ is a critical point of the
functional $J$ on $\mathcal{\hat{S}}_a$. By this fact and Lagrange's Multiplier rule, there exists $\lambda_a>0$ such that $(v,\lambda_a)$ is a solution of \x{f1},
 that is,  $(u,\lambda_a):=(f(v),\lambda_a)$ is a solution of problem \x{p}. To look for such  solutions,
 by this scaling \x{vt} and the definition of \x{f2}, we have
\begin{equation}\label{s}
\Psi(v_t)=\frac{t^2}{2}\int_{\R^N}\frac{1+2t^Nf^2(v)}{1+2f^2(v)}|\nabla v|^2\text{d}x- t^{-N}\int_{\R^N}H(t^{N/2}f(v))\text{d}x.
\end{equation}
By $\frac{\text{d}\Psi(v_t)}{\text{d}t}|_{t=1}=0$, it follows that
\begin{equation}\label{0s}
\aligned
\mathcal{G}(v):&=\int_{\R^N}|\nabla v|^2\text{d}x+\frac{N}{2}\int_{\R^N}\frac{2f^2(v)}{1+2f^2(v)}|\nabla v|^2\text{d}x\\
&\quad\quad\quad\quad\quad\quad -\frac{N}{2}\int_{\R^N}[h(f(v))f(v)-2H(f(v)]\text{d}x=0.
\endaligned
\end{equation}
Define
\begin{equation*}
  \mathcal{N}_a:=\left\{v\in\mathcal{S}_a: \mathcal{G}(v)=0\right\}.
\end{equation*}

In the following,
let us recall some properties of the change of variables $f : \R\rightarrow\R$  which are proved in \cite{CL1,LWW}.
\begin{lemma}\label{lemma2.1}\cite{CL1,LWW}
 The function  $f(t)$ and its derivative satisfy the following properties: \\
$(1)$ $f$  is uniquely defined, $\mathcal{C}^\infty$  and invertible;\\
$(2)$ $|f'(t)|\leq1$   for all  $t\in\R$;\\
$(3)$ $|f(t)|\leq|t|$   for all  $t\in\R$;\\
$(4)$ $f(t)/t\rightarrow1$  as  $t\rightarrow0$;\\
$(5)$ $f(t)/\sqrt{t}\rightarrow2^{\frac{1}{4}}$   as  $t\rightarrow+\infty$;\\
$(6)$ $f(t)/2\leq tf'(t)\leq f(t)$ for all  $t>0$;\\
$(7)$ $f^2(t)/2\leq tf(t)f'(t)\leq f^2(t)$   for all  $t\geq0$;\\
$(8)$ $|f(t)|\leq2^{1/4}|t|^{1/2}$  for all $t\in\R$;\\
$(9)$ {\it there exists a positive constant} $C$ {\it such that}
\begin{displaymath}
|f(t)|\geq\left\{
\begin{array}{ll}
C|t|, & \textrm{if $|t|\leq 1$},\\
C|t|^{\frac{1}{2}},& \textrm{if $|t|\geq1$};
\end{array} \right.\\
\end{displaymath}
$(10)$ $|f(t)f'(t)|\leq1/\sqrt{2}$  for all $t\in\R$.
 \end{lemma}

 \begin{lemma}\label{lemma2.3}
 (i) For any $v\in H^1(\R^N)$,  there holds $f(v)\in H^1(\R^N)$.

(ii) The mapping $v\mapsto f(v)$ from $H^1(\R^N)$  into $L^q(\R^N)$ is continuous for $q\in[2,22^*]$.

(iii) The mapping $v\mapsto f(v)$ from $H^1(\R^N)$  into $H^1(\R^N)$  is continuous.

(iv) There exists $C>0$ such that
\begin{equation}\label{henzhong}
\int_{\R^N}(|\nabla v|^2+|f(v)|^2)\text{d}x\geq C\|v\|^2,\ \ \forall\,v\in H^1(\R^N).
\end{equation}
\end{lemma}
\begin{proof}
$(i)$ For any $v\in H^1(\R^N)$, by Lemma \ref{lemma2.1}-$(3)$, we have that
\begin{equation*}
\aligned
\|f(v)\|^2\leq \int_{\R^N}\left(|\nabla v|^2+f^2(v)\right)\text{d}x\leq \|v\|^2<+\infty.
\endaligned
\end{equation*}

$(ii)$
Now, we only need to prove that if $v_n\rightarrow v$ in $H^1(\R^N)$, then $f(v_n)\rightarrow f(v)$ in $L^q(\R^N)$ for $q\in[2,2^*]$. In fact, since $v_n\rightarrow v$ in $H^1(\R^N)$, we conclude that
$
\frac{\partial v_n}{\partial x_i}\rightarrow\frac{\partial v}{\partial x_i}\ \text{in}\  L^2(\R^N)
$
for $i=1,2,\ldots,N$. From Lemma A.1 in \cite{MW}, up to a subsequences, there exists $L_i\in L^2(\R^N)$ such that
$
\left|\frac{\partial v_n}{\partial x_i}\right|\leq L_i(x)\, \text{a.e. on }\, \R^N
$
for $i=1,2,\ldots,N$. Thus

\begin{equation*}
\left|f'(v_n)\frac{\partial v_n}{\partial x_i}\right|\leq\left|\frac{\partial v_n}{\partial x_i}\right|\leq L_i(x)\, \text{a.e. on }\, \R^N.
\end{equation*}
Moreover,
\begin{equation*}
f'(v_n)\frac{\partial v_n}{\partial x_i}\rightarrow f'(v)\frac{\partial v}{\partial x_i}\ \ \ \text{a.e. on }\,   \R^N,
\end{equation*}
for $i=1,2,\ldots,N$. Therefore,
\begin{equation*}
\frac{\partial f(v_n)}{\partial x_i}\rightarrow \frac{\partial f(v)}{\partial x_i}\ \ \ \text{a.e. on }\ \R^N,
\end{equation*}
for $i=1,2,\ldots,N$. By the Lebesgue's dominated convergence theorem, we have
\begin{equation}\label{sssj12}
\|\nabla f(v_n)-\nabla f(v)\|_2\rightarrow0.
\end{equation}
In addition, by Lagrange Mean Value Theorem, there exists $\theta\in(0,1)$ such that
\begin{equation}\label{sss12}
\int_{\R^N}|f(v_n)-f(v)|^2\text{d}x=\int_{\R^N}|f'(v+\theta(v_n-v))|^2|v_n-v|^2\text{d}x\leq\|v_n-v\|^2_2\rightarrow0.
\end{equation}
It follows from (i), \x{sssj12}  and \x{sss12} that
\begin{equation*}
\|f(v_n)-f(v)\|^2_q\leq C\||\nabla f(v_n) -\nabla f(v)\|^2_2+C\|f(v_n)-f(v)\|^2_2\rightarrow0, \ \text{for all $q\in[2,2^*]$},
\end{equation*}
which, together with Young's inequality, we can get the results.

$(iii)$ Suppose that  $v_n\rightarrow v$ in $H^1(\R^2)$, similar to \x{sssj12} and \x{sss12}, then we also know that
\begin{equation*}
\|\nabla f(v_n)-\nabla f(v) \|_2\rightarrow0
\end{equation*}
and
\begin{equation*}
\int_{\R^N}|f(v_n)-f(v)|^2\text{d}x\rightarrow0,
\end{equation*}
as $n\rightarrow\infty$.

$(iv)$
In fact, this result has been proved by Step 2 of Theorem 2.1 in \cite{XW}. So we omit it.

The proof is completed.
\end{proof}

 \begin{lemma}\label{lemmas2.1}
Any critical point $v$ of $\Psi|'_{\mathcal{S}_a}$ satisfies $\mathcal{G}(v)=0$.
\end{lemma}
 \begin{proof}
By Lagrange's Multiplier rule, there exists $\lambda\in\R$ such that
\begin{equation}\label{m}
\Psi'(v)-\lambda f(v)f'(v)=0 \ \ \text{in}\ \ (H^1(\R^N))^*.
\end{equation}
In \x{m} testing with $\frac{f(v)}{f'(v)}$, we have
\begin{equation}\label{m1}
\int_{\R^N}|\nabla v|^2\text{d}x+\int_{\R^N}\frac{2f^2(v)}{1+2f^2(v)}|\nabla v|^2\text{d}x-\lambda\int_{\R^N}|f(v)|^2\text{d}x-\int_{\R^N}h(f(v))f(v)\text{d}x=0.
\end{equation}
Since $v$ is a solution of
\begin{equation*}
-\Delta v-\lambda f(v)f'(v)=h(f(v))f'(v),
\end{equation*}
 it follows that $v$ satisfies the following equality:
\begin{equation}\label{mm2}
\frac{N-2}{2}\int_{\R^N}|\nabla v|^2\text{d}x-\frac{N}{2}\lambda\int_{\R^N}|f(v)|^2\text{d}x-N\int_{\R^N}H(f(v))\text{d}x=0.
\end{equation}
Combining \x{m1} with \x{mm2}, we  have $\mathcal{G}(v)=0$.
 \end{proof}

\section{Minimizer problem for $N\geq 2$ with prescribed mass}
In this part, we give a dual method, which is a more direct method due to $L^2$-norm unchanged, to study minimization  problem on $\mathcal{S}_a$  with  $N\geq 2$.

\begin{lemma}\label{lemmas22.2}
Suppose that $2<p<4+\frac{4}{N}$. Then

$(i)$ $-\infty<\digamma(a)\leq0$ for all $a>0$ and $\Phi_p$ is coercive on $\mathcal{S}_a$; in addition, if $2<p<2+\frac{4}{N}$, then $\digamma(a)<0$ for all $a>0$; moreover, if $2+\frac{4}{N}\leq p<4+\frac{4}{N}$, then $\digamma(a)=0$ for $a>0$ small enough;

$(ii)$ the following sub-additivity inequality:
\begin{equation*}
\digamma(a_1)\leq \digamma(a_2)+ \digamma(a_1-a_2)\ \ \text{for all}\ \ a_2\in(0,a_1)
\end{equation*}
holds and the mapping $a\mapsto \digamma(a)$ is nonincreasing on $(0,+\infty)$; furthermore, if  $\digamma(a_2)$ or $\digamma(a_1-a_2)$ can be attained, then
\begin{equation}\label{a3}
\digamma(a_1)<\digamma(a_2)+ \digamma(a_1-a_2)\ \ \text{for all}\ \ a_2\in(0,a_1);
\end{equation}

 $(iii)$ the mapping $a\mapsto \digamma(a)$ is continuous  on $(0,+\infty)$.
\end{lemma}
\begin{proof}
$(i)$ For any $v\in H^1(\R^N)$, from Lemma \ref{lemma2.1}-(10),  we can infer that
\begin{equation}\label{section1}
\|\nabla f^2(v)\|^2_2\leq 2\|\nabla v\|^2_2\  \text{and}\  \|\nabla f(v)\|^2_2\leq \|\nabla v\|^2_2.
\end{equation}

For any $v\in \mathcal{S}_a$, it follows from Lemma \ref{lemma1.0a}  and \x{section1}
that
\begin{equation}\label{a5}
\aligned
\Phi_p(v)&\geq\frac{1}{2}\|\nabla v\|^2_2-\frac{C^p_{N,p}}{p}\|\nabla f^2(v)\|^{\frac{(p-2)N}{N+2}}_2\|f(v)\|^{\frac{4N-(N-2)p}{N+2}}_2\\
&\geq\frac{1}{2}\|\nabla v\|^2_2-\frac{C^p_{N,p}2^{\frac{(p-2)N}{2(N+2)}}}{p}a^{\frac{4N-(N-2)p}{2(N+2)}}\|\nabla v\|^{\frac{(p-2)N}{N+2}}_2,
\endaligned
\end{equation}
which implies that $\digamma(a)>-\infty$ and $\Phi_p$ is coercive on $\mathcal{S}_a$, since $2<p<4+\frac{4}{N}$.
From \x{s}, we get
\begin{equation*}
\Phi_p(v_t)=\frac{t^2}{2}\int_{\R^N}\frac{1+2t^Nf^2(v)}{1+2f^2(v)}|\nabla v|^2\text{d}x-\frac{t^{\frac{N(p-2)}{2}}}{p}\int_{\R^N}|f(v)|^p\text{d}x.
\end{equation*}
Thus $\digamma(a)\leq\lim\limits_{t\rightarrow 0^+}\Phi_p(v_t)=0$ for all $a>0$.

If $2<p<2+\frac{4}{N}$,
then there exists $t_0>0$ small enough such that $\Phi_p(v_{t_0})<0$.
Hence $\digamma(a)\leq \Phi_p(v_{t_0})<0$ for all $a>0$ and $2<p<2+\frac{4}{N}$.

 If $2+\frac{4}{N}\leq p<4+\frac{4}{N}$,  by \x{a5}, $\Phi_p(v)>0$ for all $v\in \mathcal{S}_a$ and $a$ sufficiently small.
For any fixed $v\in \mathcal{S}_a$, since $\Phi_p(v_t)\rightarrow0$ as $t\rightarrow0^+$, this leads to $\digamma(a)=0$ for $a>0$ small enough.

$(ii)$ For any fixed $a>0$, we can choose $\{v_n\}\subset \mathcal{S}_a$ such that $\Phi_p(v_n)\rightarrow \digamma(a)<0$. Set
\begin{equation*}
\bar{v}_n(x):=v_n\left(\alpha^{-\frac{1}{N}}x\right),\ \ \text{for any $\alpha>1$}.
\end{equation*}
By a simple calculation, we get that
\begin{equation}\label{section7}
\aligned
&\|f(\bar{v}_n)\|^2_2=\alpha\|f(v_n)\|^2_2=\alpha a, \ \ \|f(\bar{v}_n)\|^p_p=\alpha\|f(v_n)\|^p_p,\ \ \|\nabla \bar{v}_n\|^2_2=\alpha^{1-\frac{2}{N}}\|\nabla v_n\|^2_2.
\endaligned
\end{equation}
Hence it follows from \x{section7} and $\alpha>1$ that
\begin{equation}\label{section8}
\aligned
\digamma(\alpha a)&\leq\Phi_p(\bar{v}_n)\\
&=\alpha\left[\frac{1}{2}\alpha^{-\frac{2}{N}}\|\nabla v_n\|^2_2 -\frac{1}{p}\|f(v_n)\|^p_p\right]\\
&<\alpha\Phi_p(v_n)=\alpha \digamma(a)+o_n(1).
\endaligned
\end{equation}
Thus $\digamma(\alpha a)\leq\alpha \digamma(a)\leq\digamma(a)$ for any $\alpha>1$ and any $a>0$, which  implies that the first part of $(ii)$  holds. Now, we only need to prove the second part of this item.
 If  $\digamma(a_2)$ or $\digamma(a_1-a_2)$ can be attained at $v$, then we choose $v_n\equiv v$ in \x{section8}, and so
\begin{equation}\label{section9}
\text{$\digamma(\alpha a_2)<\alpha \digamma(a_2)$  or $\digamma(\alpha(a_1-a_2))<\alpha \digamma(a_1-a_2)$ for any $\alpha>1$.}
\end{equation}
It follows from \x{section9} and  the
monotonicity of $\digamma(a)$ that
\begin{equation*}
\aligned
\digamma(a_1)&=\frac{a_1-a_2}{a_1} \digamma\left(\frac{a_1}{a_1-a_2}(a_1-a_2)\right)+\frac{a_2}{a_1} \digamma\left(\frac{a_1}{a_2}a_2\right)\\
&<\frac{a_1-a_2}{a_1}\digamma\left(\frac{a_1}{a_1-a_2}(a_1-a_2)\right)+\digamma(a_2)\\
&<\digamma\left(a_1-a_2\right)+\digamma(a_2)
\ \ \text{for all}\ \ a_2\in(0,a_1),
\endaligned
\end{equation*}
or
\begin{equation*}
\aligned
\digamma(a_1)&=\frac{a_1-a_2}{a_1} \digamma\left(\frac{a_1}{a_1-a_2}(a_1-a_2)\right)+\frac{a_2}{a_1} \digamma\left(\frac{a_1}{a_2}a_2\right)\\
&<\digamma\left(a_1-a_2\right)+\frac{a_2}{a_1} \digamma\left(\frac{a_1}{a_2}a_2\right)\\
&<\digamma\left(a_1-a_2\right)+\digamma(a_2)
\ \ \text{for all}\ \ a_2\in(0,a_1).
\endaligned
\end{equation*}

$(iii)$ Let $\{a_n\}\subset(0,+\infty)$ such that $a_n\rightarrow a$ as $n\rightarrow\infty$. Since $\digamma(a)\leq0$ for all $a>0$, if $a_n<a$, then it follows from $(ii)$ that for any $\varepsilon>0$,
 \begin{equation}\label{section3}
\digamma(a)\leq\digamma(a-a_n)+\digamma(a_n)\leq\digamma(a_n)+\varepsilon.
 \end{equation}
If $a_n\geq a$, we can choose a sequence $\{v_n\}\subset \mathcal{S}_{a_n}$ such that $\Phi_p(v_n)\leq \digamma(a_n)+\frac{1}{n}\leq\frac{1}{n}$. By \x{a5}, we have that $\{\|\nabla v_n\|^2_2\}$ is bounded.
By \x{section1} and Lemma \ref{lemma1.0a}, we know that
$\{\|f(v_n)\|^p_p\}$ is bounded.
 Moreover, let $\varpi_n(x):=f^{-1}\left(\sqrt{\frac{a}{a_n}}f(v_n(x))\right)$.
Then $\varpi_n\in \mathcal{S}_a$. From \x{f2}, we get
\begin{equation}\label{section4}
\aligned
\digamma(a)&\leq\Phi_p(\varpi_n)=\Phi_p(v_n)+[\Phi_p(\varpi_n)-\Phi_p(v_n)]\\
&=\Phi_p(v_n)+\frac{\frac{a}{a_n}-1}{2}\int_{\R^N}\frac{\left[1+2\left(\frac{a}{a_n}+1\right)f^2(v_n)\right]}{1+2f^2(v_n)}|\nabla v_n|^2\text{d}x\\
&\quad\quad\quad\quad\quad\quad-\frac{\left(\sqrt{\frac{a}{a_n}}\right)^{p}-1}{p}\int_{\R^N}|f(v_n)|^p\text{d}x\\
&\leq\Phi_p(v_n)+o_n(1)\\
&\leq \digamma(a_n)+o_n(1).
\endaligned
\end{equation}
Combining \x{section3} and \x{section4}, for any $\varepsilon>0$,  we have
\begin{equation}\label{section4a}
\digamma(a)\leq \digamma(a_n)+o_n(1)+\varepsilon.
\end{equation}
Moreover, for the above $\varepsilon>0$, there exists $v\in \mathcal{S}_a$ such that
\begin{equation}\label{section5}
\Phi_p(v)<\digamma(a)+\varepsilon.
\end{equation}
Set $\omega_n(x)=f^{-1}\left(\sqrt{\frac{a_n}{a}}f(v)\right)$. Then $\omega_n\in \mathcal{S}_{a_n}$. Since $\Phi_p(\omega_n)\rightarrow\Phi_p(v)$ as $n\rightarrow\infty$, it follows from \x{section5} that
\begin{equation}\label{section6}
\digamma(a_n)\leq\Phi_p(\omega_n)=\Phi_p(v)+[\Phi_p(\omega_n)-\Phi_p(v)]=\Phi_p(v)+o_n(1)<\digamma(a)+\varepsilon+o_n(1).
\end{equation}
By \x{section4a}, \x{section6} and the arbitrariness of $\varepsilon>0$, we deduce that $\digamma(a_n)\rightarrow \digamma(a)$ as $n\rightarrow\infty$. The proof is completed.
\end{proof}

\begin{lemma}\label{lemma22.3}
Suppose that $2<p<2+\frac{4}{N}$.
Then
$\digamma(a)<0$ is achieved for all $a>0$.
\end{lemma}
\begin{proof}
It follows from $2<p<2+\frac{4}{N}$ and Lemma \ref{lemmas22.2}(i) that $\digamma(a)<0$  for all $a>0$. Next, we show that $\digamma(a)$ is achieved. Indeed, let $\{v_n\}\subset \mathcal{S}_a$ such that $\Phi_p(v_n)\rightarrow \digamma(a)$. Since $\Phi_p$ is coercive on $\mathcal{S}_a$, this shows that $\{\|\nabla v_n\|^2_2\}$ is bounded.   Thus it follows from Lemma \ref{lemma2.3}(iv) that $\{v_n\}$ is bounded in $H^1(\R^N)$.

{\bf Claim: }
\begin{equation*}
\delta:=\limsup_{n\rightarrow\infty}\sup_{y\in\R^N}\int_{B_1(y)}|v_n|^2\text{d}x>0.
\end{equation*}

If $\delta=0$, by Lions' concentration compactness principle as \cite{MW}, we have that $v_n\rightarrow 0$ in $L^r(\R^N)$ with $r\in(2,2^*)$ and thus $f(v_n)\rightarrow 0\  \text{in}\ \ L^{\bar{r}}(\R^N)$ for all $\bar{r}\in(2,22^*)$.  Then
\begin{equation*}
0\leq\lim_{n\rightarrow\infty}\frac{1}{2}\left(\int_{\R^N}|\nabla v_n|^2\text{d}x\right)=\digamma(a)<0.
\end{equation*}
This is a contradiction. This shows that $\delta>0$.

Going if necessary to a following point, there exists a sequence $\{y_n\}\subset\R^N$ such that
\begin{equation*}
\int_{B_1(y_n)}|v_n|^2\text{d}x>\frac{\delta}{2}.
\end{equation*}
Let $\tilde{v}_n(x):=v_n(x+y_n)$. Then
\begin{equation*}
\int_{B_1(0)}|\tilde{v}_n|^2\text{d}x>\frac{\delta}{2}.
\end{equation*}
Up to a subsequence, there exists $\tilde{v}_0\in H^1(\R^N)\backslash\{0\}$ such that
\begin{equation}\label{section10}
\tilde{v}_n\rightharpoonup \tilde{v}_0\ \text{in}\ H^1(\R^N),\  \tilde{v}_n\rightarrow \tilde{v}_0\  \text{in}\ \ L^r_{loc}(\R^N) \  \text{for all $r\in(2,2^*)$}, \  \tilde{v}_n\rightarrow \tilde{v}_0\ \text{a.e. on $\R^N$}.
\end{equation}
By the weaker semi-continuous of the norm, we get $\|f(\tilde{v}_0)\|^2_2\leq\liminf\limits_{n\rightarrow\infty}\|f(\tilde{v}_n)\|^2_2=a$. Next, we claim that $\|f(\tilde{v}_0)\|^2_2=a$. In fact, by a contradiction, we assume that $b:=\|f(\tilde{v}_0)\|^2_2<a$, where $b>0$.
Let $w_n:=\tilde{v}_n-\tilde{v}_0$. Now, we first show that
 \begin{equation}\label{section11}
\|f(\tilde{v}_n)\|^2_2-\|f(w_n)\|^2_2=\|f(\tilde{v}_0)\|^2_2+o_n(1)
 \end{equation}
and
\begin{equation}\label{section12}
\Phi_p(\tilde{v}_n)-\Phi_p(w_n) =\Phi_p(\tilde{v}_0)+o_n(1).
\end{equation}
By Lemma \ref{lemma2.1}-(2), (3) and Young's inequality, for any $\varepsilon>0$, there exists a constant $M_\varepsilon>0$ such that
\begin{equation}\label{section13}
 \aligned
\left|f^2(\tilde{v}_n)-f^2(\tilde{v}_n-\tilde{v}_0)-f^2(\tilde{v}_0)\right|&=\left|-\int_0^1\frac{\text{d}}{\text{d}t}f^2(\tilde{v}_n-t\tilde{v}_0)\text{d}t-f^2(\tilde{v}_0)\right|\\
&=\left|2\int_0^1f(\tilde{v}_n-t\tilde{v}_0)f'(\tilde{v}_n-t\tilde{v}_0)\tilde{v}_0\text{d}t-f^2(\tilde{v}_0)\right|\\
&\leq\varepsilon|\tilde{v}_n|^2+M_\varepsilon|\tilde{v}_0|^2.
\endaligned
\end{equation}
Set
\begin{equation*}
l^\varepsilon_n(x):=\max\left\{\left|f^2(\tilde{v}_n)-f^2(\tilde{v}_n-\tilde{v}_0)-f^2(\tilde{v}_0)\right|-\varepsilon|\tilde{v}_n|^2,0\right\}.
\end{equation*}
Then $l^\varepsilon_n(x)\rightarrow0$ a.e. on $\R^N$ and $0\leq l^\varepsilon_n(x)\leq M_\varepsilon|\tilde{v}_0|^2\in L^1(\R^N)$.
Thus by Lebesgue's dominated  convergence theorem, \x{section10} and \x{section13}, we have that  $\int_{\R^N}l^\varepsilon_n(x)\text{d}x\rightarrow0$ as $n\rightarrow\infty$.
 Furthermore, by the definition of $l^\varepsilon_n(x)$, we get \x{section11}. Similar to the proof of \x{section11}, we can also get \x{section12}.

From \x{section11}, we have $c_n:=\|f(w_n)\|^2_2\rightarrow a-b$ as $n\rightarrow\infty$. By \x{section12}, one has
\begin{equation*}
\digamma(a)=\Phi_p(w_n)+\Phi_p(\tilde{v}_0)+o_n(1)\geq \digamma(c_n)+\Phi_p(\tilde{v}_0)+o_n(1).
\end{equation*}
Since $\Phi_p(\tilde{v}_0)\geq \digamma(b)$, suppose that $\Phi_p(\tilde{v}_0)>\digamma(b)$,
by the fact that $\digamma(a)$ is  continuous on $(0,+\infty)$ and Lemma \ref{lemmas22.2}(ii), we get
\begin{equation}\label{section14}
\digamma(a)\geq \digamma(a-b)+\Phi_p(\tilde{v}_0)>\digamma(a-b)+\digamma(b)\geq \digamma(a).
\end{equation}
  This is a contradiction. Thus $\Phi_p(\tilde{v}_0)=\digamma(b)$ and so $\digamma(b)$ is achieved at $\tilde{v}_0$.  Arguing as \x{section14}, we also get a contradiction and so $\|f(\tilde{v}_0)\|^2_2=a^2$. This shows that $f(\tilde{v}_n-\tilde{v}_0)\rightarrow0$ in $L^2(\R^N)$. By H\"{o}lder inequality, for any $\bar{r}\in(2,22_*)$, we get that $f(\tilde{v}_n-\tilde{v}_0)\rightarrow0$ in $L^{\bar{r}}(\R^N)$ and
 \begin{equation*}
\int_{\R^N}|f(\tilde{v}_n)|^{\bar{r}}\text{d}x\rightarrow\int_{\R^N}|f(\tilde{v}_0)|^{\bar{r}}\text{d}x
\end{equation*}
as $n\rightarrow\infty$, where the second limit needs to use
\begin{equation*}
\|f(\tilde{v}_n)\|^{\bar{r}}_{\bar{r}}-\|f(\tilde{v}_n-\tilde{v}_0)\|^{\bar{r}}_{\bar{r}}=\|f(\tilde{v}_0)\|^{\bar{r}}_{\bar{r}}+o_n(1).
\end{equation*}
Hence $\Phi_p(\tilde{v}_0)\leq\liminf\limits_{n\rightarrow\infty}\Phi_p(\tilde{v}_n)= \digamma(a)$, by the weak semi-continuity of the norm. Since $\tilde{v}_0\in \mathcal{S}_a$, it follows that $\Phi_p(\tilde{v}_0)=\digamma(a)$.
\end{proof}

Let
 \begin{equation*}
a_*=\inf\left\{a>0\big|\digamma(a)<0\right\}.
 \end{equation*}

\begin{lemma}\label{lemma22.4}
Suppose that $2+\frac{4}{N}\leq p<4+\frac{4}{N}$. Then

 $(i)$ $a_{*}>0$,
$
\digamma(a)=0
$ for any $a\in(0,a_*]$ and $\digamma(a)<0$ for all $a\in(a_{*},+\infty)$;

$(ii)$ $\digamma(a)=0$ is not attained for any $a\in(0,a_*)$;

$(iii)$ $\digamma(a_*)=0$ is attained for $2+\frac{4}{N}< p<4+\frac{4}{N}$;

$(iv)$ $\digamma(a_*)=0$ is not attained for  $p=2+\frac{4}{N}$.

\end{lemma}
\begin{proof}
$(i)$ Let $v\in \mathcal{S}_1$ be fixed. Set
\begin{equation*}
v^a(x):=v(a^{-\frac{1}{N}}x).
\end{equation*}
Then $\|f(v^a)\|^2_2=a\|f(v)\|^2_2=a$. From this, we infer that
\begin{equation*}
\Phi_p(v^a)=\frac{a^{1-\frac{2}{N}}}{2}\int_{\R^N}|\nabla v|^2\text{d}x-\frac{a}{p}\int_{\R^N}|f(v)|^p\text{d}x.
\end{equation*}
This show that $\Phi_p(v^a)\rightarrow-\infty$ as $a\rightarrow+\infty$. Thus there exists $a_{0}>0$ large enough such that $\digamma(a_0)<0$. Thus, in view of Lemma \ref{lemmas22.2}(i), we get $a_{*}>0$. Moreover, using Lemma \ref{lemmas22.2}(ii)-(iii), we conclude that $\digamma(a)=0$ for any $0<a\leq a_*$ and $\digamma(a)<0$ for all $a\in(a_{*},+\infty)$.

$(ii)$ Suppose by contradiction, there exists $a_0\in(0, a_*)$ such that $\digamma(a_0)=0$ is achieved at $v\in \mathcal{S}_{a_0}$. Then  from this  and \x{section8}, we have
\begin{equation*}
0=\digamma(a_*)=\digamma\left(\frac{a_*}{a_0}a_0\right)<\frac{a_*}{a_0}\digamma\left(a_0\right)=0,
\end{equation*}
which is contradiction.

$(iii)$ Let $a_n=a_*+\frac{1}{n}$. It follows from (i) that $\digamma(a_n)<0$. Since $a_n>a^*$, we can suppose that there exists $\{v_n\}\subset\mathcal{S}_{a_n}$ such that $\digamma(a_n)=\Phi_p(v_n)<0$ for all  $n$. By Lemma \ref{lemmas22.2}(iii), we  have that $\digamma(a_n)\rightarrow \digamma(a_*)=0$ as $n\rightarrow\infty$. From Lemma \ref{lemmas22.2}(i), we know that $\{\|\nabla v_n\|^2_2\}$ is bounded and thus it follows from Lemma \ref{lemma2.3}(iv) that  $\{ v_n\}$ is bounded in $H^1(\R^N)$. Set
\begin{equation*}
\bar{\delta}:=\limsup_{n\rightarrow\infty}\sup_{y\in\R^N}\int_{B_1(y)}|v_n|^2\text{d}x.
\end{equation*}
If  $\bar{\delta}=0$,   then it follows from  Lions' Lemma in \cite{MW} that $v_n\rightarrow0$ in $L^s(\R^N)$ for all $s\in(2,2^*)$. Thus
\begin{equation*}
\int_{\R^N}|f(v_n)|^s\text{d}x\rightarrow0,\ \ \text{as}\,\, n\rightarrow\infty.
\end{equation*}
Hence
\begin{equation*}
  0= \digamma(a_*)=\lim_{n\rightarrow\infty}\digamma(a_n)=\frac{1}{2}\lim_{n\rightarrow\infty}\|\nabla v_n\|^2_2.
\end{equation*}

{\bf Case 1: $2+\frac{4}{N}< p\leq2^*$($2\leq N\leq3$). }
For $2+\frac{4}{N}< p\leq2^*$,
 we deduce from the Sobolev inequality that
\begin{equation*}
\aligned
\Phi_p(v_n)&\geq\frac{1}{2}\|\nabla v_n\|^2_2-\|f(v_n)\|^{\frac{2N-(N-2)p}{2}}_2\|f(v_n)\|^{\frac{(N-2)(p-2)}{2}}_{2^*}\\
&\geq\frac{1}{2}\|\nabla v_n\|^2_2-Ca^{\frac{2N-(N-2)p}{4}}\|\nabla v_n\|^{\frac{N(p-2)}{2}}_{2}\\
&=\|\nabla v_n\|^2_2\left(\frac{1}{2}-Ca^{\frac{2N-(N-2)p}{4}}\|\nabla v_n\|^{\frac{N(p-2)-4}{2}}_{2}\right)\\
&\geq0\ \ \text{for $n$ large enough.}
\endaligned
\end{equation*}
This is a contradiction.

{\bf Case 2: $2^*< p<4+\frac{4}{N}$($N\geq4$). } Motivated by \cite{ZLW},  set
\begin{displaymath}
  v_n^T=\left\{
\begin{array}{ll}
v_n, & \textrm{if $|v_n|\leq T$},\\
\pm T,& \textrm{if not},
\end{array} \right.\\
\end{displaymath}
where $T>0$ is a parameter. Obviously, $|v_n^T|\leq   |v_n|$   and $|\nabla v_n^T|\leq   |\nabla v_n|$.   Moreover, by Lagrange multipliers' multipliers rule, there exists $\lambda_n\in\R$ such that
\begin{equation}\label{ewq1}
-\Delta v_n-\lambda_n f(v_n)f'(v_n)=|f(v_n)|^{p-2}f(v_n)f'(v_n).
\end{equation}
From the fact that $\digamma(a_n)=\Phi_p(v_n)<0$ and \x{mm2}, we know that
\begin{equation*}
\aligned
\digamma(a_n)=\Phi_p(v_n)=\left(\frac{1}{2}-\frac{1}{2^*}\right)\|\nabla v_n\|^2_2+\frac{\lambda_n}{2}\|f(v_n)\|^{2}_{2},
\endaligned
\end{equation*}
which shows that $\lambda_n<0$. In \x{ewq1} testing with
$|v_n^T|^{r_0-r}v_n^T$ where  $r_0=2^*,\,r=\frac{p}{2}$,  for fixed $n$, letting $T\rightarrow+\infty$, it follows from Lemma \ref{lemma2.1} that
\begin{equation*}
\aligned
\int_{\R^N}\nabla v_n\nabla(v_n^{r_0-r+1})\text{d}x&\leq\int_{\R^N}|f(v_n)|^{p-2}f(v_n)f'(v_n)v_n^{r_0-r+1}\text{d}x\\
&\leq C\int_{\R^N}| v_n |^{r}|v_n|^{r_0-r}\text{d}x\\
&\leq C\int_{\R^N}|v_n|^{2^*}\text{d}x\leq C.
\endaligned
\end{equation*}
This implies that
\begin{equation*}
\aligned
C&\geq\int_{\R^N}\nabla v_n\nabla(v_n^{r_0-r+1})\text{d}x=(r_0-r+1)\int_{\R^N}|\nabla v_n|^2v_n^{r_0-r}\text{d}x\\
&=\frac{4(r_0-r+1)}{(r_0-r+2)^2}\int_{\R^N}\left|\nabla v_n^{\frac{r_0-r+2}{2}}\right|^2\text{d}x
\\
&\geq C\frac{4(r_0-r+1)}{(r_0-r+2)^2}\left[\int_{\R^N}\left| v_n^{\frac{r_0-r+2}{2}}\right|^{r_0}\text{d}x\right]^{2/r_0}.
\endaligned
\end{equation*}
Let $r_1:=\frac{(r_0-r+2)r_0}{2}$. Then $r_1\geq r_0$ and $\{v_n\}$ is bounded in $L^{r_1}(\R^N)$. By a similar argument as the above, setting  $r_i=\frac{r_0(r_{i-1}-r+2)}{2}$ and testing with   $|v_n^T|^{r_i-r}v_n^T$ in \x{ewq1}, it follows from the iterating that   $\{v_n\}$ is bounded in $L^{q}(\R^N)$ for all $q\geq  2^*$. By Sobolev inequality and H\"{o}lder inequality, we have
\begin{equation*}
\aligned
\|f(v_n)\|^p_p&\leq C\|\nabla f(v_n)\|^{\tau p}\|f(v_n)\|^{(1-\tau)p}_{4N}\\
&\leq\|\nabla  v_n \|^{\tau p}\|f(v_n)\|^{(1-\tau)p}_{4N}\\
&\leq C\|\nabla  v_n \|^{\tau p},\ \ \tau=\frac{4N-p}{p(2N-5)}.
\endaligned
\end{equation*}
Thus $\tau p>2$ and
\begin{equation*}
\aligned
\Phi_p(v_n)&\geq\frac{1}{2}\|\nabla v_n\|^2_2-C\|\nabla  v_n \|^{\tau p}\\
&=\|\nabla v_n\|^2_2\left(\frac{1}{2}-C\|\nabla  v_n \|^{\tau p-2}\right)\geq0\ \ \text{for $n$ large enough.}
\endaligned
\end{equation*}
This is also a contradiction.

Hence $\bar{\delta}>0$.  Up to subsequence, there exists $\{y_n\}\subset\R^N$ and $v\in H^1(\R^N)\backslash\{0\}$ such that $\bar{v}_n:=v_n(x+y_n)\rightharpoonup v$ in $H^1(\R^N)$, $v_n\rightarrow v$ in $L^s_{loc}$ for $s\in[2,2^*)$, $v_n\rightarrow v$ a. e. on $\R^N$.
Thus $\|f(\bar{v}_n)\|^2_2=\|f(v_n)\|^2_2\rightarrow a_*$ and $\Phi_{p}(\bar{v}_n)\rightarrow\digamma(a_*)$ as $n\rightarrow\infty$, and so $b:=\|f(v)\|^2_2\leq a_*$. Now, let $b<a_*$. By \x{section11}, we have that  $\|f(\bar{v}_n-v)\|^2_2\rightarrow a_*-b$. From this and \x{section12}, we get
\begin{equation*}
\digamma(a_*)=\lim_{n\rightarrow\infty}\Phi_p(\bar{v}_n)=\Phi_p(v)+ \Phi_p(\bar{v}_n-v)+o_n(1)\geq\Phi_p(v)+\digamma(\|f(\bar{v}_n-v)\|^2_2)+o_n(1).
\end{equation*}
Since $\Phi_p(v)\geq \digamma(b)$, assume that $\Phi_p(v)> \digamma(b)$, one has
\begin{equation*}
\digamma(a_*)\geq\Phi_p(v)+\digamma(a_*-b)\digamma(b)+\digamma(a_*-b)\geq\digamma(a_*).
\end{equation*}
This is a contradiction. Thus $\Phi_p(v)=\digamma(b)$ and $\digamma(b)$ is achieved at $v$. The rest of the proof is similar to the last part on the proof of Lemma \ref{lemma22.3}.

$(iv)$ Suppose by a contradiction,  there exists $v_*\in \mathcal{S}_{a_*}$ such that $\Phi_p(v_*)=\digamma(a_*)=0$. By Lemma \ref{lemmas2.1}, we have
\begin{equation*}
\Lambda(v_*):=\int_{\R^N}|\nabla v_*|^2\text{d}x+\frac{N}{2}\int_{\R^N}\frac{2f^2(v_*)}{1+2f^2(v_*)}|\nabla v_*|^2\text{d}x-\frac{N}{N+2}\int_{\R^N}|f(v_*)|^{2+\frac{4}{N}}\text{d}x=0.
\end{equation*}
Based on the above facts, it follows that
\begin{equation*}
0=\Phi_p(v_*)-\frac{1}{2}\Lambda(v_*)=-\frac{N}{4}\int_{\R^N}\frac{2f^2(v_*)}{1+2f^2(v_*)}|\nabla v_*|^2\text{d}x<0,
\end{equation*}
which is contradiction.
\end{proof}

\begin{remark}\label{remark111}
For $2+\frac{4}{N}\leq p<4+\frac{4}{N}$, from Lemma \ref{lemma22.4}, as the proof of Lemma \ref{lemma22.3}, we know that $F(a)<0$ can be achieved for all $a\in(a_{*},+\infty)$.
\end{remark}

\begin{lemma}\label{lemma22.5a}
Suppose that $p=4+\frac{4}{N}$ and
\begin{equation*}
a^*:=\left[\frac{2(N+1)Q_{4+\frac{4}{N}}}{N}\right]^{\frac{N}{2}},
\end{equation*}
where
\begin{equation*}
Q_{4+\frac{4}{N}}:=\inf_{v\in H^1(\R^N)\backslash\{0\}}\frac{\left(\int_{\R^N}|f(v)|^2\text{d}x\right)^{\frac{2}{N}}\int_{\R^N}\frac{2f^2(v)}{1+2f^2(v)}|\nabla v |^2\text{d}x}{\int_{\R^N}|f(v)|^{4+\frac{4}{N}}\text{d}x}.
\end{equation*}
  Then $Q_{4+\frac{4}{N}}>0$ and

(i)   $\digamma(a)=0$ for all $a\in(0,a^*]$;

(ii)  $\digamma(a)=-\infty$ for all $a>a^*$;

(iii) $\digamma(a)$ is not attained for all $a>0$.
\end{lemma}
\begin{proof}
By Lemma \ref{lemma1.0a} and Lemma \ref{lemma2.1}-(10),  for any $v\in H^1(\R^N)\backslash\{0\}$,   we know that
\begin{equation*}
   \int_{\R^N}|f(v)|^{4+\frac{4}{N}}\text{d}x\leq 2C_{N,4+\frac{4}{N}}^{4+\frac{4}{N}}\left( \int_{\R^N}|f(v)|^2\text{d}x\right)^{\frac{2}{N}}\int_{\R^N}\frac{2f^2(v)}{1+2f^2(v)}|\nabla  v|^2\text{d}x,
 \end{equation*}
which shows that $Q_{4+\frac{4}{N}}>0$. By the expression of $Q_{4+\frac{4}{N}}$, we have
\begin{equation*}
Q_{4+\frac{4}{N}}\int_{\R^N}|f(v)|^{4+\frac{4}{N}}\text{d}x\leq \left(\int_{\R^N}|f(v)|^2\text{d}x\right)^{\frac{2}{N}}\int_{\R^N}\frac{2f^2(v)}{1+2f^2(v)}|\nabla v |^2\text{d}x.
\end{equation*}
This implies that
\begin{equation*}
\frac{1}{4+\frac{4}{N}}\int_{\R^N}|f(v)|^{4+\frac{4}{N}}\text{d}x\leq  \frac{a^{\frac{2}{N}}}{\left(4+\frac{4}{N}\right)Q_{4+\frac{4}{N}}}\int_{\R^N}\frac{2f^2(v)}{1+2f^2(v)}|\nabla v |^2\text{d}x.
\end{equation*}
Since \begin{equation*}
a^*:=\left[\frac{2(N+1)Q_{4+\frac{4}{N}}}{N}\right]^{\frac{N}{2}},
\end{equation*}
it follows that $0<\frac{a^{\frac{2}{N}}}{\left(4+\frac{4}{N}\right)Q_{4+\frac{4}{N}}}\leq\frac{1}{2}$ for all $0<a\leq a^*$ and  $\frac{a^{\frac{2}{N}}}{\left(4+\frac{4}{N}\right)Q_{4+\frac{4}{N}}}>\frac{1}{2}$ for all $a> a^*$.

$(i)$ On the one hand, for all  $0<a\leq a^*$, we have
\begin{equation*}
\Phi_{4+\frac{4}{N}}(v)\geq\left[\frac{1}{2}- \frac{a^{\frac{2}{N}}}{\left(4+\frac{4}{N}\right)Q_{4+\frac{4}{N}}}\right]\int_{\R^N}|\nabla v |^2\text{d}x \geq0, \  \ \text{for all $v\in\mathcal{S}_a$}.
\end{equation*}
This shows that $\digamma(a)\geq0$.  On the other hand, for any $t>0$, $v\in \mathcal{S}_a$ and the expression of $v_t$, we have
\begin{equation*}
\Phi_{4+\frac{4}{N}}(v_t)=\frac{t^2}{2}\int_{\R^N}\frac{1+2t^Nf^2(v)}{1+2f^2(v)}|\nabla v|^2\text{d}x- t^{N+2}\int_{\R^N}|f(v)|^{4+\frac{4}{N}}\text{d}x\rightarrow0,
\end{equation*}
as $t\rightarrow0^+$. Thus  $\digamma(a)\leq0$ and so $\digamma(a)=0$.

$(ii)$ For any $a>a^*$, we have that $Q_{4+\frac{4}{N}}<\frac{a^{\frac{2}{N}}}{2+\frac{2}{N}}$. Thus there exists $v\in H^1(\R^N)\backslash\{0\}$ such that
\begin{equation*}
\frac{\left(\int_{\R^N}|f(v)|^2\text{d}x\right)^{\frac{2}{N}}\int_{\R^N}\frac{2f^2(v)}{1+2f^2(v)}|\nabla v |^2\text{d}x}{\int_{\R^N}|f(v)|^{4+\frac{4}{N}}\text{d}x}<\frac{a^{\frac{2}{N}}}{2+\frac{2}{N}}.
\end{equation*}
Let $\bar{v}:=f^{-1}\left(\frac{\sqrt{a}}{\|f(v)\|_2}f(v)\right).$ Then $\bar{v}\in\mathcal{S}_a$. Thus
\begin{equation*}
\aligned
\Phi_{4+\frac{4}{N}}(\bar{v}_t)&=\frac{t^2}{2}\int_{\R^N}|\nabla f(\bar{v})|^2\text{d}x+\frac{t^{N+2}}{2}\int_{\R^N}\frac{2f^2(\bar{v})}{1+2f^2(\bar{v})}|\nabla \bar{v}|^2\text{d}x- t^{N+2}\int_{\R^N}|f(\bar{v})|^{4+\frac{4}{N}}\text{d}x\\
&=\frac{at^2}{2\|f(v)\|^2_2}\int_{\R^N}|\nabla f(v)|^2\text{d}x+\left(\frac{a}{\|f(v)\|^2_2}\right)^2t^{N+2}\\
&\quad \times\left[\frac{1}{2}\int_{\R^N}\frac{2f^2(v)}{1+2f^2(v)}|\nabla  v|^2\text{d}x-\frac{1}{4+\frac{4}{N}}\left(\frac{a}{\|f(v)\|^2_2}\right)^{2/N}\int_{\R^N}|f(v)|^{4+\frac{4}{N}}\text{d}x\right]\\
&\rightarrow-\infty, \ \ \text{as $t\rightarrow+\infty$.}
\endaligned
\end{equation*}
Thus $\digamma(a)\leq \lim\limits_{t\rightarrow+\infty}\Phi_{4+\frac{4}{N}}(\bar{v}_t)=-\infty$ and so $\digamma(a)=-\infty$ for all $a>a^*$.

$(iii)$ To our aim, we only need to prove the case $0<a\leq a^*$. Suppose by a contradiction, there exist $0<a\leq a^*$ and $v_a\in \mathcal{S}_a$ such that $\digamma(a)=\Phi_{4+\frac{4}{N}}(v_a)$. By   Lemma \ref{lemmas2.1}, we have
\begin{equation*}
\tilde{\Lambda}(v_a):=\frac{N+2}{2}\int_{\R^N}|\nabla v_a|^2\text{d}x-\frac{N}{2}\int_{\R^N}\frac{1}{1+2f^2(v_a)}|\nabla v_a|^2\text{d}x-\frac{N+2}{4+\frac{4}{N}}\int_{\R^N}|f(v_a)|^{4+\frac{4}{N}}\text{d}x=0.
\end{equation*}
Based on the above facts, it follows that
\begin{equation*}
0=\Phi_{4+\frac{4}{N}}(v_a)-\frac{1}{N+2}\tilde{\Lambda}(v_a)=\frac{N}{2(N+2)}\int_{\R^N}\frac{1}{1+2f^2(v_a)}|\nabla v_a|^2\text{d}x>0,
\end{equation*}
which is contradiction.
\end{proof}

\begin{lemma}\label{lemma22.5}
Suppose that $4+\frac{4}{N}<p<22^*$. Then $\digamma(a)=-\infty$ for all $a>0$.
\end{lemma}
\begin{proof}
For $4+\frac{4}{N}<p<22^*$, it follows  that $\Phi_p(v_t)\rightarrow-\infty$ as $t\rightarrow+\infty$.
Thus $\digamma(a)=-\infty$. The proof is completed.
\end{proof}
\par\noindent
{\bf Proof of Theorem \ref{theoremnew}.}\,\, By Lemmas \ref{lemmas22.2}-\ref{lemma22.5}, we can get the proof. $\hfill\Box$

\section{Ground state normalized solutions with general $L^2$-supercritical growth and $2\leq N\leq3$}
In this part, we will show the existence of ground state normalized solutions with general $L^2$-supercritical growth. Since the ground state energy of the perturbed functional may not coincide with that of the original functional, perturbation methods cannot ensure the existence of ground state solutions. Based on this, we propose to employ a dual method to prove the existence of  ground state normalized solutions. Now, we study $\Psi(v)$ on radial space $H^1_r(\R^N)$.  From this and Palais' symmetric principle in \cite{MW}, the critical point in $H^1_r(\R^2)$ is also the critical point in $H^1(\R^2)$.  Next, we give the following lemmas.
\begin{lemma}\label{lemma3.1}
Let
\begin{equation*}
\aligned
A(t,v)&=\frac{1}{2}\bigg[\left(1-\frac{t^2(1+2t^Nf^2(v))}{1+2f^2(v)}\right)|\nabla v|^2-\frac{1-t^{N+2}}{N+2}\left( 2
+\frac{2Nf^2(v)}{1+2f^2(v)}\right)|\nabla v|^2\bigg].
\endaligned
\end{equation*}
Then for any  $t>0$ and $v\in H^1_r(\R^N)\backslash\{0\}$, $A(t,v)\geq 0$.
\end{lemma}

\begin{proof}
Since
\begin{equation*}
\aligned
&\frac{\text{d}A(t,v)}{\text{d}t}\\
&~~~=\frac{t}{2}\bigg\{-2\left(\frac{1+2t^Nf^2(v)}{1+2f^2(v)}\right)|\nabla v|^2-\frac{2Nt^Nf^2(v)}{1+2f^2(v)}|\nabla v|^2+t^{N}\left(2+\frac{2Nf^2(v)}{1+2f^2(v)} \right)|\nabla v|^2\bigg\}\\
&~~~=t(t^N-1)|\nabla f(v)|^2 ,
\endaligned
\end{equation*}
it follows that $\frac{\text{d}A(t,v)}{\text{d}t}\leq0$ for all $0<t\leq1$, and $\frac{\text{d}A(t,v)}{\text{d}t}\geq0$ for all $t\geq1$.
Thus $A(t,v)\geq0$, for any $t>0$ and $v\in H^1_r(\R^N)\backslash\{0\}$.
\end{proof}

\begin{lemma}\label{lemma3.2}
Assume that  $(h_3)$ holds and
\begin{equation*}
\aligned
B(t,v)&:=-H(f(v))+t^{-N}H(t^{N/2}f(v))+\frac{N(1-t^{N+2})}{2(N+2)}[h(f(v))f(v)-2H(f(v))].
\endaligned
\end{equation*}
Then $B(t,v)\geq0$, for any  $t>0$ and $v\in H^1_r(\R^N)\backslash\{0\}$.
\end{lemma}
\begin{proof}
Since
\begin{equation*}
\aligned
\frac{\text{d}B(t,v)}{\text{d}t}
&=\frac{N}{2t^{N+1}}\left[h(t^{N/2}f(v))t^{N/2}f(v)-2H(t^{N/2}f(v))\right]\\
&\quad\quad\quad\quad\quad\quad\quad\quad\quad\quad\quad\quad\quad\quad\quad\quad\quad-\frac{Nt^{N+1}}{2}\left[h(f(v))f(v)-2H(f(v))\right]\\
&=\frac{Nt^{N+1}}{2}|f(v)|^{3+\frac{4}{N}}f(v)\bigg[\frac{h(t^{N/2}f(v))t^{N/2}f(v)-2H(t^{N/2}f(v))}{|t^{N/2}f(v)|^{3+\frac{4}{N}}t^{N/2}f(v)}\\
&\quad\quad\quad\quad\quad\quad\quad\quad\quad\quad\quad\quad\quad\quad\quad\quad\quad\quad\quad\quad-\frac{h(f(v))f(v)-2H(f(v))}{|f(v)|^{3+\frac{4}{N}}f(v)}\bigg],
\endaligned
\end{equation*}
it follows that $\frac{\text{d}B(t,v)}{\text{d}t}\leq0$ for all $0<t\leq1$, and $\frac{\text{d}B(t,v)}{\text{d}t}\geq0$ for all $t\geq1$, which dues to $(h_3)$. Thus
\begin{equation*}
B(t,v)\geq B(1,v)=0
\end{equation*}
 for any  $t>0$ and $v\in H^1_r(\R^N)\backslash\{0\}$.
\end{proof}

\begin{lemma}\label{lemma3.3}
For any  $t>0$ and $v\in H^1_r(\R^N)$, there holds
\begin{equation}\label{s5}
\Psi(v)=\Psi(v_t)+\frac{1-t^{N+2}}{N+2}\mathcal{G}(v)+\int_{\R^N}A(t,v)\text{d}x+\int_{\R^N}B(t,v)\text{d}x.
\end{equation}

\end{lemma}
\begin{proof}
By a direct computation, we have
\begin{equation*}
\aligned
&\Psi(v)-\Psi(v_t)\\
&~~=\frac{1}{2}\int_{\R^N}|\nabla v|^2\text{d}x-\int_{\R^N}H(f(v))\text{d}x-\frac{t^2}{2}\int_{\R^N}\frac{1+2t^Nf^2(v)}{1+2f^2(v)}|\nabla v|^2\text{d}x\\
&\quad\quad\quad\quad\quad\quad\quad+t^{-N}\int_{\R^N} H(t^{N/2}f(v))\text{d}x\\
&~~=\frac{1-t^{N+2}}{N+2}\left\{\int_{\R^N}|\nabla v|^2\text{d}x+\frac{N}{2}\int_{\R^N}\frac{2f^2(v)}{1+2f^2(v)}|\nabla v|^2\text{d}x\right\}\\
&\quad\quad\quad\quad\quad\quad\quad-\frac{N(1-t^{N+2})}{2(N+2)}\int_{\R^N}[h(f(v))f(v)-2H(f(v))]\text{d}x\\
&\quad\quad+\frac{1}{2}\bigg[ \int_{\R^N}\left(1-\frac{t^2(1+2t^Nf^2(v))}{1+2f^2(v)}\right)|\nabla v|^2\text{d}x-\frac{1-t^{N+2}}{N+2}\int_{\R^N}\left( 2+\frac{2Nf^2(v)}{1+2f^2(v)}\right)|\nabla v|^2\text{d}x\bigg]\\
&\quad\quad -\int_{\R^N}H(f(v))\text{d}x
+ t^{-N}\int_{\R^N}H(t^{N/2}f(v))\text{d}x\\
&\quad\quad\quad\quad\quad\quad\quad+\frac{N(1-t^{N+2})}{2(N+2)}\int_{\R^N}[h(f(v))f(v)-2H(f(v))]\text{d}x\\
&~~=\frac{1-t^{N+2}}{N+2}\mathcal{G}(v)+\int_{\R^N}A(t,v)\text{d}x+\int_{\R^N}B(t,v)\text{d}x,
\endaligned
\end{equation*}
which implies that \x{s5} holds for any  $t>0$ and $v\in H^1_r(\R^N)$.
\end{proof}

\begin{lemma}\label{lemma3.4}
Assume that $(h_3)$ holds. Then
\begin{equation}\label{s4}
\Psi(v)=\max_{t>0}\Psi(v_t), \ \ \ \forall\,v\in \mathcal{N}_a.
\end{equation}
\end{lemma}
 \begin{proof}
 The proof is obvious. So we omit it.
 \end{proof}

\begin{lemma}\label{lemma3.6} Assume that  $(h_1)$-$(h_3)$ hold. Then

(i) there exists a constant $\varrho>0$ such that $\|\nabla v\|_2^2\geq\varrho$  for any $v\in \mathcal{N}_a$;

(ii) $\sigma(a):=\inf\limits_{v\in\mathcal{N}_a}\Psi(v)>0$.
 \end{lemma}
\begin{proof}
$(i)$ By Lemma \ref{lemma2.3}, Lemma \ref{lemma1.0a} and Lemma \ref{lemma2.1}-(10),  for any $v\in H^1_r(\R^N)$ and $4+\frac{4}{N}< p<22^*$, we have
 \begin{equation}\label{s7}
 \aligned
\|f(v)\|^p_p&\leq C^p_{p,N}\|f(v)\|_2^{\frac{4N-p(N-2)}{N+2}}\|\nabla f^2(v)\|_2^{\frac{N(p-2)}{N+2}}\\
&=C^p_{p,N}2^{\frac{N(p-2)}{2(N+2)}}a^{\frac{4N-p(N-2)}{2(N+2)}}\|\nabla v\|_2^{\frac{N(p-2)}{N+2}}.
\endaligned
 \end{equation}
 Since $v\in \mathcal{N}_a$, it follows from $\mathcal{G}(v)=0$, $(h_2)$,  \x{o1}, \x{h}, \x{0s} and \x{s7} that
 \begin{equation}\label{s8}
  \aligned
\|\nabla v\|_2^2
&\leq\frac{N}{2}\int_{\R^N}[h(f(v))f(v)-2H(f(v))]\text{d}x\\
&\leq\frac{N(\mu_2-2)}{2}\int_{\R^N}H(f(v))\text{d}x\\
&\leq\frac{N(\mu_2-2)}{2}\int_{\R^N}\left(\varepsilon|f(v)|^{4+\frac{4}{N}}+C_\varepsilon|f(v)|^{\mu_2}\right)\text{d}x\\
&\leq\frac{N(\mu_2-2)}{2}\left[C_\varepsilon C^{\mu_2}_{\mu_2,N}2^{\frac{N(\mu_2-2)}{2(N+2)}}a^{\frac{4N-\mu_2(N-2)}{2(N+2)}}\|\nabla v\|_2^{\frac{N(\mu_2-2)}{N+2}}
+2 \varepsilon C^{4+\frac{4}{N}}_{4+\frac{4}{N},N}a^{\frac{2}{N}}\|\nabla v\|_2^{2}\right].
\endaligned
 \end{equation}
Since $\frac{N(\mu_2-2)}{N+2}>2$, it follows that  there exists a constant $\varrho>0$ such that
\begin{equation*}
\|\nabla v\|_2^2\geq\varrho,\ \ \forall\,v\in\mathcal{N}_a.
\end{equation*}

$(ii)$ It follows from $(h_2)$ that
\begin{equation*}
 \aligned
\Psi(v)&=\Psi(v)-\frac{2}{N(\mu_1-2)}\mathcal{G}(v)\\
 &=\frac{N\mu_1-4-4N}{2N(\mu_1-2)}\|\nabla  v\|_2^2+\frac{1}{\mu_1-2}\|\nabla  f(v)\|_2^2+\frac{1}{\mu_1-2}\int_{\R^N}\left[h(f(v))f(v)-\mu_1H(f(v))\right]\text{d}x\\
 &\geq\frac{N\mu_1-4-4N}{2N(\mu_1-2)} \varrho>0,\ \ \forall\,v\in\mathcal{N}_a.
 \endaligned
\end{equation*}
From this fact, we know that $\sigma(a)>0$.
\end{proof}

\begin{lemma}\label{lemma3.7}
 Assume that $(h_1)$-$(h_3)$ hold.  Then

(i) for any $v\in H^1_r(\R^N)\backslash\{0\}$, there exists a unique $t_v>0$ such that $\mathcal{G}(v_{t_v})=0$ and $\Psi(v_{t_v})>0$;
in addition, if $v\in \mathcal{S}_a$, then $v_{t_v}\in\mathcal{N}_a$;

(ii) the mapping $v\mapsto t_v$ is continuous in $v\in H^1_r(\R^N)\setminus\{0\}$.

 \end{lemma}
 \begin{proof}
$(i)$ Let $v\in H^1(\R^N)\backslash\{0\}$ be fixed. Set
 \begin{equation*}
 \aligned
\text{g}(t):&=\Psi(v_t)
=\frac{t^2}{2}\int_{\R^N}\frac{1+2t^Nf^2(v)}{1+2f^2(v)}|\nabla v|^2\text{d}x-t^{-N}\int_{\R^N}H(t^{N/2}f(v))\text{d}x.
 \endaligned
 \end{equation*}
By a simple calculation, we have
\begin{equation}\label{z}
\aligned
0&=\frac{\text{d}\text{g}(t)}{\text{d}t}\\
&=t\int_{\R^N}\frac{1+2t^Nf^2(v)}{1+2f^2(v)}|\nabla v|^2\text{d}x+\frac{Nt^{N+1}}{2}\int_{\R^N}\frac{2f^2(v)}{1+2f^2(v)}|\nabla v|^2\text{d}x\\
&\quad\quad\quad\quad\quad-\frac{N}{2}t^{-(N+1)}\int_{\R^N}[h(t^{N/2}f(v))t^{N/2}f(v)-2H(t^{N/2}f(v))]\text{d}x\\
&=\frac{1}{t}\mathcal{G}(v_t),
 \endaligned
\end{equation}
which implies that for $v\in \mathcal{S}_a$,
\begin{equation*}
\mathcal{G}(v_t)=0\Leftrightarrow v_t\in\mathcal{N}_a.
\end{equation*}
It follows from Remark \ref{remark1.1} that $\text{g}(t)<0$ for $t>0$ enough large and $\text{g}(t)>0$ for $t>0$ enough small.
Thus $\max\limits_{t>0}\text{g}(t)$ is achieved at some $t_v>0$ such that $\frac{\text{d}\text{g}(t)}{\text{d}t}|_{t=t_v}=0$  and $\Psi(v_{t_v})>0$. Furthermore, if $v\in\mathcal{S}_a$, then $v_{t_v}\in\mathcal{N}_a$.

Next, we prove the uniqueness.
For any given $v\in H^1_r(\R^N)\backslash\{0\}$, there exist $t_1,t_2>0$ and $t_1\neq t_2$ such that $v_{t_1},v_{t_2}\in\mathcal{N}_a$. Then $\mathcal{G}(v_{t_1})=\mathcal{G}(v_{t_2})=0$. Therefore, it follows from Lemmas \ref{lemma3.1}-\ref{lemma3.3} that
\begin{equation*}
\aligned
\Psi(v_{t_1})>\Psi(v_{t_2})+\frac{t^{N+2}_1-t^{N+2}_2}{(N+2)t^{N+2}_1}\mathcal{G}(v_{t_1})=\Psi(v_{t_2})
\endaligned
\end{equation*}
and
\begin{equation*}
\aligned
\Psi(v_{t_2})&>\Psi(v_{t_1})+\frac{t^{N+2}_2-t^{N+2}_1}{(N+2)t^{N+2}_2}\mathcal{G}(v_{t_2})=\Psi(v_{t_1}).
\endaligned
\end{equation*}
This contradiction shows that $t_1=t_2$. Thus $t_v>0$ is unique for $v\in H^1_r(\R^N)\backslash\{0\}$ and thus $\Psi(v_{t_v})=\max\limits_{t>0}\Psi(v_t)$.

$(ii)$ To this end, we suppose that there exists $\{v_n\}\subset H^1_r(\R^N)$ and $v_0\in H^1(\R^N)\backslash\{0\}$ such that $v_n\rightarrow v_0$ in $H^1_r(\R^N)$.
By Lemma \ref{lemma2.1}-(3) and (8), there exists $C>0$ such that $$\int_{\R^N}|f(v_n)|^{2}\text{d}x\leq C \ \ \text{and}\ \ \int_{\R^N}|f(v_n)|^{\mu_2}\text{d}x\leq C.$$   By $(i)$, for any fixed $n\in \mathbb{N}$,
 there exists a unique $t_{v_n}>0$ such that $\mathcal{G}((v_n)_{t_{v_n}})=0$ and $\Psi((v_n)_{t_{v_n}})>0$.

 Let $t_n:=t_{v_n}$ for any $n\geq 1$.
Next, we need to show that up to a subsequence $t_n\rightarrow t_{v_0}$ as $n\rightarrow\infty$.
Now, we first prove that $\{t_n\}$ is bounded. We only need to show that the sequence  $\{t_n\}$ is bounded  from above. Suppose by a contradiction, if up to a subsequence $t_n\rightarrow+\infty$, by $(i)$,
 Fatou's lemma and Remark \ref{remark1.1},  then
\begin{equation}\label{zzzz3}
\aligned
0&<\Psi((v_n)_{t_n})\\
&=\frac{t^2_n}{2}\int_{\R^N}\frac{1+2t^N_nf^2(v_n)}{1+2f^2(v_n)}|\nabla v_n|^2\text{d}x-\int_{\R^N}H(t_n^{N/2}f(v_n))\text{d}x\\
&\rightarrow-\infty,\ \ \text{as}\ \ n\rightarrow\infty,
\endaligned
\end{equation}
which is a contradiction. Thus the sequence  $\{t_n\}$ is bounded.

Without loss of generality, we can assume that there exists $t^*\in\R$ such that $t_n\rightarrow t^*$ as $n\rightarrow\infty$. Recalling that $v_n\rightarrow v_0$ in $H^1(\R^N)$,
we can infer that $(v_n)_{t_n}\rightarrow(v_0)_{t^*}$ in $H^1_r(\R^N)$. Since $\mathcal{G}((v_n)_{t_n})=0$ for any $n\geq1$, we deduce that $\mathcal{G}((v_0)_{t^*})=0$.
 By $(i)$, we have $t^*=t_{v_0}$ and thus item $(ii)$ holds.
\end{proof}

\begin{lemma}\label{lemma3.10}

Assume that  $(h_1)$-$(h_3)$ hold. Then

(i) the function $a\mapsto \sigma(a)$ is continuous at each $a>0$;

 (ii) the function $a\mapsto \sigma(a)$ is nonincreasing on $(0,+\infty)$;

(iii) $\sigma(a)>\sigma(\tilde{a})$ for  all $\tilde{a}>a$, if $\sigma(a)$ is achieved;

(iv) $\sigma(a)$ is achieved.
\end{lemma}
\begin{proof}
$(i)$ It is equivalent to  prove that for a given $a>0$ and any positive sequence $\{a_n\}$ such that $a_n\rightarrow a$ as $n\rightarrow\infty$,
  one has $\sigma(a_n)\rightarrow\sigma(a)$ as $n\rightarrow\infty$. Next, we first show that
 \begin{equation}\label{zz}
\limsup_{n\rightarrow\infty}\sigma(a_n)\leq\sigma(a).
 \end{equation}
 In fact, for any $v\in \mathcal{N}_a$, we define $v_n:=f^{-1}(\sqrt{\frac{a_n}{a}}f(v))\in S_{a_n}$. Thus
 \begin{equation}\label{zzzm3}
\int_{\R^N}(|\nabla v_n|^2+|f(v_n)|^2)\text{d}x\rightarrow\int_{\R^N}(|\nabla v|^2+|f(v)|^2)\text{d}x\  \text{as}\  n\rightarrow\infty.
 \end{equation}
From  this, \x{section11} and  \x{zzzm3}, we have that $v_n\rightarrow v$ in $H^1_r(\R^N)$.
By Lemma \ref{lemma3.7}, we can conclude that there exists $t_{v_n}>0$ such that $(v_n)_{t_{v_n}}\in\mathcal{N}_a$ and $t_{v_n}\rightarrow t_v=1$ as $n\rightarrow\infty$. Thus
\begin{equation*}
(v_n)_{t_{v_n}}\rightarrow(v)_{t_v}=v\ \ \text{in}\ \ H^1_r(\R^N),
\end{equation*}
as $n\rightarrow\infty$. As a consequence, it follows from Lemma \ref{lemma3.7}(i) that
$$\limsup\limits_{n\rightarrow\infty}\sigma(a_n)\leq\limsup\limits_{n\rightarrow\infty}\Psi((v_n)_{t_{v_n}})=\Psi(v),$$
for any $v\in\mathcal{N}_a$. Thus we can get \x{zz}.

Now, we show the following
\begin{equation}\label{zz1}
\liminf_{n\rightarrow\infty}\sigma(a_n)\geq\sigma(a).
\end{equation}
For each $n\geq1$, there exists $\bar{v}_n\in \mathcal{N}_{a_n}$, such that
\begin{equation}\label{zz2}
\Psi(\bar{v}_n)\leq\sigma(a_n)+\frac{1}{n}.
\end{equation}
Since $a_n\rightarrow a$ as $n\rightarrow\infty$, there exists a constant $M>0$ such that $|a_n|\leq M$ for any $n\geq1$.
Similar to the proof of Lemma \ref{lemma3.6}(ii), we can show that
$\{
\|\nabla \bar{v}_n\|_2^2
\}$
is bounded, since $\mu_1>4+\frac{4}{N}$.
By Lemma \ref{lemma2.3}(iv),  we have  that $\{\bar{v}_n\}$ is also bounded in $H^1_r(\R^N)$.

Let
\begin{equation*}
\hat{v}_n:=\bar{v}_n\left(\left(\frac{a}{a_n}\right)^{-1/N}x\right).
\end{equation*}
Then $\{\hat{v}_n\}$ is bounded in $H^1_r(\R^N)$ and
\begin{equation*}
\int_{\R^N}|f(\hat{v}_n)|^2\text{d}x=a,
\end{equation*}
which shows that $\hat{v}_n\in \mathcal{S}_a$. Thus by Lemma \ref{lemma3.7}, Lemmas \ref{lemma3.1}-\ref{lemma3.3} and \x{zz2}, we get
\begin{equation*}
\aligned
\sigma(a)&\leq\Psi((\hat{v}_n)_{t_{\hat{v}_n}})\\
&\leq\Psi((\bar{v}_n)_{t_{\hat{v}_n}})+\left|\Psi((\hat{v}_n)_{t_{\hat{v}_n}})-\Psi((\bar{v}_n)_{t_{\hat{v}_n}})\right|\\
&\leq\Psi(\bar{v}_n)+\left|\Psi((\hat{v}_n)_{t_{\hat{v}_n}})-\Psi((\bar{v}_n)_{t_{\hat{v}_n}})\right|\\
&\leq\sigma(a_n)+\frac{1}{n}+\left|\Psi((\hat{v}_n)_{t_{\hat{v}_n}})-\Psi((\bar{v}_n)_{t_{\hat{v}_n}})\right|\\
&\leq\sigma(a_n)+\frac{1}{n}+\frac{1}{2}\left|\left(\frac{a}{a_n}\right)^{\frac{N-2}{N}}-1\right|\int_{\R^N}|\nabla(\bar{v}_n)_{t_{\hat{v}_n}}|^2\text{d}x\\
&\quad\quad\quad\quad\quad\quad+\left|\frac{a}{a_n}-1\right|\int_{\R^N}H(f((\bar{v}_n)_{t_{\hat{v}_n}}))\text{d}x.
\endaligned
\end{equation*}
Next, we only need to show that
\begin{equation*}
\frac{1}{2}\left|\left(\frac{a}{a_n}\right)^{\frac{N-2}{N}}-1\right|\int_{\R^N}|\nabla(\bar{v}_n)_{t_{\hat{v}_n}}|^2\text{d}x
+\left|\frac{a}{a_n}-1\right|\int_{\R^N}H(f((\bar{v}_n)_{t_{\hat{v}_n}}))\text{d}x\rightarrow0\  \text{as}\ n\rightarrow\infty.
\end{equation*}
It is equivalent to prove that
\begin{equation*}
\limsup_{n\rightarrow\infty}\|(\bar{v}_n)_{t_{\hat{v}_n}}\|<+\infty.
\end{equation*}
Recalling that  $\{\bar{v}_n\}$ is bounded in $H^1_r(\R^N)$ and  $a_n\rightarrow a$ as $n\rightarrow\infty$, we only to prove that $\limsup\limits_{n\rightarrow\infty}t_{\hat{v}_n}<+\infty.$
In fact, by contradiction, up to a subsequence, we assume that  up to a subsequence, $t_{\hat{v}_n}\rightarrow+\infty$ as $n\rightarrow\infty$. Recalling that $\{\hat{v}_n\}$ is bounded in $H^1_r(\R^N)$, there exists
 $v\in H^1_r(\R^N)\backslash\{0\}$ such that
\begin{equation}\label{zzzz}
\hat{v}_n\rightarrow v, \ \text{a.e. in}\ \R^N.
\end{equation}
Since
\begin{equation*}
\Psi((\hat{v}_n)_{t_{\hat{v}_n}})\geq\sigma(a)>0,
\end{equation*}
by \x{zzzz}, we get a contradiction as the same discussion in \x{zzzz3}.

$(ii)$ Now, we only need to show that for any $0<a_1<a_2$, there holds
\begin{equation*}
\sigma(a_2)\leq\sigma(a_1).
\end{equation*}
In fact, by the definition of $\sigma(a_1)$, there exists a sequence $\{v_n\}\subset\mathcal{N}_{a_1}$ such that
 \begin{equation}\label{lemma3.10-1}
\sigma(a_1)\leq\Psi(v_n)\leq\sigma(a_1)+\frac{1}{n}.
 \end{equation}
Let $\hat{\delta}:=\sqrt{\frac{a_2}{a_1}}\in(1,+\infty)$ and
\begin{equation*}
\omega_n(x):=f^{-1}\left(\hat{\delta}^{\frac{2-N}{2}}f\left(v_n\left(\hat{\delta}^{-1}x\right)\right)\right).
\end{equation*}
Thus
\begin{equation*}
\int_{\R^N}|f(\omega_n)|^2\text{d}x=a_2,\ \  \int_{\R^N}|\nabla f(\omega_n)|^2\text{d}x=\int_{\R^N}|\nabla f(v_n)|^2\text{d}x.
\end{equation*}
By  Lemma \ref{lemma3.7}-$(i)$, there exists $t_{n}>0$ such that $(\omega_n)_{t_{n}}\in \mathcal{N}_{a_2}$. Thus it follows from $\hat{\delta}>1$ that
\begin{equation*}
 \aligned
\int_{\R^N}\frac{1+2t_{n}^Nf^2(\omega_n)}{1+2f^2(\omega_n)}|\nabla \omega_n|^2\text{d}x&=\int_{\R^N}|\nabla f(\omega_n)|^2\text{d}x
+t_{n}^N\int_{\R^N}\frac{2f^2(\omega_n)}{1+2f^2(\omega_n)}|\nabla \omega_n|^2\text{d}x\\
&\leq\int_{\R^N}|\nabla f(v_n)|^2\text{d}x+\int_{\R^N}\frac{2 t_{n}^N f^2(v_n)}{1+2f^2(v_n)}|\nabla v_n|^2\text{d}x\\
&=\int_{\R^N}\frac{1+2t_{n}^Nf^2(v_n)}{1+2f^2(v_n)}|\nabla v_n|^2\text{d}x.
 \endaligned
\end{equation*}
Thus it follows from the above facts, Remark \ref{remark1.1} and   \x{lemma3.10-1} that
\begin{equation*}
 \aligned
\sigma(a_2)&\leq \Psi((\omega_n)_{t_{n}})\\
&=\frac{t_{n}^2}{2}\int_{\R^N}\frac{1+2t_{n}^Nf^2(\omega_n)}{1+2f^2(\omega_n)}|\nabla \omega_n|^2\text{d}x- t^{-N}_n\int_{\R^N}H(t_{n}^{N/2}f(\omega_n))\text{d}x\\
&\leq\frac{t_{n}^2}{2}\int_{\R^N}\frac{1+2t_{n}^Nf^2(v_n)}{1+2f^2(v_n)}|\nabla v_n|^2\text{d}x- t^{-N}_n\hat{\delta}^{N}\int_{\R^N}H\left(\hat{\delta}^{\frac{2-N}{2}}t_{n}^{N/2}f\left(v_n\right)\right)\text{d}x\\
&\leq\Psi((v_n)_{t_{n}})+\left(1-\hat{\delta}^{\frac{(2-N)\mu_2+2N}{2}}\right)t_{n}^{-N}\int_{\R^N}H(t_{n}^{N/2}f(v_n))\text{d}x\\
&<\Psi((v_n)_{t_{n}})\leq\Psi(v_n)\leq\sigma(a_1)+\frac{1}{n},
 \endaligned
\end{equation*}
which implies that $\sigma(a_2)\leq\sigma(a_1)$, for any $0<a_1<a_2$.

$(iii)$ Since $\sigma(a)$ is achieved, there exists $v\in\mathcal{N}_a$ such that $\Psi(v)=\sigma(a)$. For any $\tilde{a}>a$, set $\hat{\xi}:=\sqrt{\frac{\tilde{a}}{a}}>1$
and $\hat{v}:=f^{-1}(\hat{\xi}^{\frac{2-N}{2}}f(v(\hat{\xi}^{-1}x)))$.
In view of Lemma \ref{lemma3.7}, there exists $t_{\hat{v}}>0$ such that $(\hat{v})_{t_{\hat{v}}}\in \mathcal{N}_{\tilde{a}}$.  By a simply  calculation, we have
\begin{equation*}
\int_{\R^N}|f(\hat{v})|^2\text{d}x=\tilde{a},\ \  \int_{\R^N}|\nabla f(\hat{v})|^2\text{d}x=\int_{\R^N}|\nabla f(v)|^2\text{d}x
\end{equation*}
and
\begin{equation*}
 \aligned
\int_{\R^N}\frac{1+2t_{\hat{v}}^Nf^2(\hat{v})}{1+2f^2(\hat{v})}|\nabla \hat{v}|^2\text{d}x&=\int_{\R^N}|\nabla f(\hat{v})|^2\text{d}x
+t_{\hat{v}}^N\int_{\R^N}\frac{2f^2(\hat{v})}{1+2f^2(\hat{v})}|\nabla \hat{v}|^2\text{d}x\\
&\leq\int_{\R^N}|\nabla f(v)|^2\text{d}x+\int_{\R^N}\frac{2 t_{\hat{v}}^N f^2(v)}{1+2f^2(v)}|\nabla v|^2\text{d}x\\
&=\int_{\R^N}\frac{1+2t_{\hat{v}}^Nf^2(v)}{1+2f^2(v)}|\nabla v|^2\text{d}x.
 \endaligned
\end{equation*}
Thus it follows from the above facts, Remark \ref{remark1.1} and Lemmas \ref{lemma3.1}-\ref{lemma3.3} that
\begin{equation*}
 \aligned
\sigma(\tilde{a})&\leq\Psi((\hat{v})_{t_{\hat{v}}})\\
&=\frac{t_{\hat{v}}^2}{2}\int_{\R^N}\frac{1+2t_{\hat{v}}^Nf^2(\hat{v})}{1+2f^2(\hat{v})}|\nabla \hat{v}|^2\text{d}x-t_{\hat{v}}^{-N}\int_{\R^N}H(t_{\hat{v}}^{N/2}f(\hat{v}))\text{d}x\\
&\leq\frac{t_{\hat{v}}^2}{2}\int_{\R^N}\frac{1+2t_{\hat{v}}^Nf^2(v)}{1+2f^2(v)}|\nabla v|^2\text{d}x-t_{\hat{v}}^{-N}\hat{\xi}^N\int_{\R^N}H(t_{\hat{v}}^{N/2}\hat{\xi}^{\frac{2-N}{2}}f(v))\text{d}x\\
&\leq\Psi((v)_{t_{\hat{v}}})+t_{\hat{v}}^{-N}\left(1-\hat{\xi}^{\frac{(2-N)\mu_2+2N}{2}}\right)\int_{\R^N}H(t_{\hat{v}}^{N/2}f(v))\text{d}x\\
&<\Psi((v)_{t_{\hat{v}}})\leq\Psi(v)=\sigma(a),
 \endaligned
\end{equation*}
which implies that $\sigma(\tilde{a})<\sigma(a)$, for all $\tilde{a}>a$.

$(iv)$ From Lemma \ref{lemma3.6} and Lemma \ref{lemma3.7}, we have $\mathcal{N}_a\neq\emptyset$ and $\sigma(a)>0$.
 Assume that $\{v_n\}\subset \mathcal{N}_a$ such that $\Psi(v_n)\rightarrow\sigma(a)$. Similar to the proof of Lemma \ref{lemma3.6},  it follows from $\mathcal{G}(v_n)=0$ and $\Psi(v_n)\rightarrow\sigma(a)$ that
 $\{\|\nabla v_n\|_2^2\}$ is bounded. By Lemma \ref{lemma2.3}(iv), up to a subsequence, there exists $v_0\in H^1_r(\R^N)$ such that $v_n\rightharpoonup v_0$ in $H^1_r(\R^N)$, $v_n\rightarrow v_0$ in $L^{\bar{r}}(\R^N)$ for all $\bar{r}\in(2,2^*)$, and then $v_n\rightarrow v_0$ a.e. on $\R^N$ and so $f(v_n)\rightarrow f(v_0)$ in $L^{\tilde{r}}(\R^N)$ for all $\tilde{r}\in(2,22^*)$.  Now,  we show that $v_0\neq0$. Suppose that $v_0=0$. Then by Remark \ref{remark1.1}, one has
 \begin{equation*}
\int_{\R^N}[h(f(v_n))f(v_n)-2H(f(v_n))]\text{d}x\rightarrow0\ \ \text{as}\ \ n\rightarrow\infty,
 \end{equation*}
which, together with $\mathcal{G}(v_n)=0$, imply that
\begin{equation*}
\aligned
\int_{\R^N}|\nabla v_n|^2\text{d}x+\frac{N}{2}\int_{\R^N}\frac{2f^2(v_n)}{1+2f^2(v_n)}|\nabla v_n|^2\text{d}x\rightarrow0\ \ \text{as}\ \ n\rightarrow\infty.
\endaligned
\end{equation*}
Based on these facts, we infer that $\Psi(v_n)\rightarrow0$. This contradicts $\sigma(a)>0$. So $v_0\neq0$.
Now, we  claim
$\mathcal{G}(v_0)\leq0$. In  fact, since $\mathcal{G}(v_n)=0$, by Fatou's lemma, we get this claim.

Since $v_0\neq0$ and $a_0:=\|f(v_0)\|_2^2\leq\liminf\limits_{n\rightarrow\infty}\|f(v_n)\|_2^2=a$, by Lemma \ref{lemma3.7}, there exists $\bar{t}_0>0$ such that $(v_0)_{\bar{t}_0}\in\mathcal{N}_{a_0}$. In view of  Lemmas \ref{lemma3.1}-\ref{lemma3.3},  we have
\begin{equation*}
 \aligned
\sigma(a)&=\liminf_{n\rightarrow\infty}\left[\Psi(v_n) -\frac{1}{N+2}\mathcal{G}(v_n)\right]\\
&=\liminf_{n\rightarrow\infty}\bigg\{\frac{N}{2(N+2)}\|\nabla f(v_n)\|_2^2\\
&\quad\quad\quad\quad +\frac{N}{2(N+2)}\left[\int_{\R^N}\left(h(f(v_n))f(v_n)\text{d}x-\frac{4N+4}{N}H(f(v_n))\right)\text{d}x\right]\bigg\}\\
&\geq\frac{N}{2(N+2)}\|\nabla f(v_0)\|_2^2\\
&\quad\quad\quad\quad +\frac{N}{2(N+2)}\left[\int_{\R^N}\left(h(f(v_0))f(v_0)\text{d}x-\frac{4N+4}{N}H(f(v_0))\right)\text{d}x\right]\\
&=\Psi(v_0)-\frac{1}{N+2}\mathcal{G}(v_0)\\
&\geq\Psi((v_0)_{\bar{t}_0})-\frac{\bar{t}_0^{N+2}}{N+2}\mathcal{G}(v_0)\\
&\geq\sigma(a_0)\geq\sigma(a),
  \endaligned
\end{equation*}
since $a_0\leq a$. Thus $\Psi(v_0)=\sigma(a)=\sigma(a_0)$, $\mathcal{G}(v_0)=0$, and $a_0=\|f(v_0)\|_2^2$. In addition, if $a_0<a$, then it follows from $(iii)$ that $\sigma(a)<\sigma(a_0)$, which is a contradiction.
This shows that $a_0=a$. Thus $\sigma(a)$ is achieved.
\end{proof}

\begin{lemma}\label{lemma3.11}
Assume that  $(h_1)$-$(h_3)$ hold  and $\Psi(\tilde{v})=\sigma(a)$ with $\tilde{v}\in\mathcal{N}_a$. Then $\tilde{v}$ is a critical point of $\Psi|'_{\mathcal{S}_a}$.
\end{lemma}
\begin{proof}
Suppose by contradiction, if $\Psi|'_{\mathcal{S}_a}(\tilde{v})\neq0$, then there exist two constants $\delta>0$ and $\tilde{\varrho}>0$  such that
\begin{equation*}
v\in\mathcal{S}_a, \ \ \ \|v-\tilde{v}\|\leq 3\delta\Rightarrow\|\Psi'(v)\|\geq\tilde{\varrho}.
\end{equation*}
By a similar argument as \cite{Chen4}, we can prove that
\begin{equation}\label{z2}
\lim_{t\rightarrow1}\|\tilde{v}_t-\tilde{v}\|=0.
\end{equation}
 Thus for $\delta>0$, there exists $\hat{\xi}\in\left(0,\frac{1}{4}\right)$ such that
\begin{equation}\label{z7}
|t-1|<\hat{\xi}\Rightarrow\|\tilde{v}_{t}-\tilde{v}\|<\delta.
\end{equation}
By Lemma \ref{lemma3.3}, one has
\begin{equation}\label{z9}
\aligned
\sigma(a)=\Psi(\tilde{v})=\Psi(\tilde{v}_{t})+\int_{\R^N}A(t,\tilde{v})\text{d}x+\int_{\R^N}B(t,\tilde{v})\text{d}x,\ \ \text{for all}\ \ t>0.
\endaligned
\end{equation}
Thus by Lemma \ref{lemma3.7}, there exist $T_1\in(0,1)$ and $T_2\in(1,\infty)$ such that
\begin{equation}\label{z11}
  \mathcal{G}(\tilde{v}_{T_1})>0,\ \ \   \mathcal{G}(\tilde{v}_{T_2})<0.
\end{equation}
\par

Let $$\epsilon:=\min\left\{\frac{\int_{\R^N}A(T_1,\tilde{v})\text{d}x+\int_{\R^N}B(T_1,\tilde{v})\text{d}x}{4},\frac{\int_{\R^N}A(T_2,\tilde{v})\text{d}x+\int_{\R^N}B(T_2,\tilde{v})\text{d}x}{4},1,\frac{\delta\tilde{\varrho}}{8}\right\}$$
and  $$\mathcal{L}:=B(\tilde{v},\delta)\cap\mathcal{S}_a.$$
Then it follows from Lemma 2.3 in \cite{MW} that there exists a deformation $\eta\in\mathcal{C}([0,1]\times \mathcal{S}_a,\mathcal{S}_a)$ such that

(i) $\eta(1,v)=v$ if $\Psi(v)<\sigma(a)-2\epsilon$ or $\Psi(v)>\sigma(a)+2\epsilon$;

(ii) $\eta(1,\Psi^{\sigma(a)+\epsilon}\cap \mathcal{L})\subset \Psi^{\sigma(a)-\epsilon}$;

(iii) $\Psi(\eta(1,v))\leq \Psi(v)$, $\forall v\in\mathcal{S}_a$.\\
By Lemma \ref{lemma3.4} and $\Psi(\tilde{v}_{t})\leq \Psi(\tilde{v})=\sigma(a)$ for $t>0$, then it follows from \x{z7} and (ii) that
\begin{equation}\label{z8}
\Psi(\eta(1,\tilde{v}_{t}))\leq \sigma(a)-\epsilon,\ \ \forall\,t>0,\ |t-1|<\hat{\xi}.
\end{equation}
Moreover, by (iii) and \x{z9}, we have
\begin{equation}\label{z10}
\aligned
\Psi(\eta(1,\tilde{v}_{t}))&\leq\Psi(\tilde{v}_t)=\sigma(a)-\int_{\R^N}A(t,\tilde{v})\text{d}x-\int_{\R^N}B(t,\tilde{v})\text{d}x\\
&\leq \sigma(a)-\int_{\R^N}A(t,\tilde{v})\text{d}x-\int_{\R^N}B(t,\tilde{v})\text{d}x\\
&\leq\sigma(a)-\delta_1,\ \ \ \forall\, t>0,\,|t-1|\geq\hat{\xi},
\endaligned
\end{equation}
where $$\delta_1:=\min\left\{\int_{\R^N}\left[A(1-\hat{\xi},\tilde{v})+B(1-\hat{\xi},\tilde{v})\right]\text{d}x,\int_{\R^N}\left[A(1+\hat{\xi},\tilde{v})+B(1+\hat{\xi},\tilde{v})\right]\text{d}x\right\}>0.$$ Thus it follows from \x{z8} and \x{z10} that
\begin{equation*}
\max_{t\in[T_1,T_2]}\Psi(\eta(1,\tilde{v}_{t}))<\sigma(a).
\end{equation*}
Let $\Gamma_0(t):=\mathcal{G}(\eta(1,\tilde{v}_{t}))$ for $t>0$. By \x{z9} and (i), we infer that $\eta(1,\tilde{v}_{t})=\tilde{v}_{t}$ for $t=T_1$ and $t=T_2$, which, together with \x{z11}, shows that
\begin{equation*}
\Gamma_0(T_1)=\mathcal{G}(\eta(1,\tilde{v}_{T_1}))>0,\ \ \ \Gamma_0(T_2)=\mathcal{G}(\eta(1,\tilde{v}_{T_2}))<0.
\end{equation*}
Since $\Gamma_0(t)$ is continuous on $(0,+\infty)$, we know that $\eta(1,\tilde{v}_{t_0})\cap\mathcal{N}_a\neq\emptyset$ for some $t_0\in[T_1,T_2]$, which contradicts to the definition of $\sigma(a)$. This completes the proof.
\end{proof}
\par\noindent
{\bf Proof of Theorem \ref{theorem1}.}\,\,For any $a>0$, it follows from Lemma \ref{lemma3.10} and Lemma \ref{lemma3.11} that there exists $v\in\mathcal{N}_a$ such that
\begin{equation*}
\Psi(v)=\sigma(a)\ \ \text{and}\ \ \Psi|'_{\mathcal{S}_a}(v)=0.
\end{equation*}
By Lagrange's multiplier rule, there exists $\lambda_a\in\R$ such that $\Psi'(v)-\lambda_a f(v)f'(v)=0$.
In addition, from Lemma \ref{lemmas2.1}, we have $\mathcal{G}(v)=0$.
It follows from $$\mathcal{G}(v)-\left\langle\Psi'(v)-\lambda_a f(v)f'(v),\frac{f(v)}{f'(v)}\right\rangle=0$$ and $(h_2)$ that
\begin{equation*}
\lambda_a\int_{\R^2}|f(v)|^2\text{d}x\leq\left(\frac{N-2}{2}\mu_2-N\right)\int_{\R^2}H(f(v))\text{d}x,
\end{equation*}
which implies that $\lambda_a<0$, since $H(t)>0$ for all $t\neq0$ and $4+\frac{4}{N}<\mu_2<2^*$.

On the one hand, we  show that $\sigma(a)\rightarrow+\infty$ as $a\rightarrow0^+$. Suppose by a contradiction, there exist  a constant $C>0$, two subsequence $\{a_n\}$ and $\{v_n\}\subset H^1(\R^N)\backslash\{0\}$ such that
\begin{equation}\label{xy1}
\int_{\R^N}|f(v_n)|^2\text{d}x=a_n\rightarrow0\  \text{as}\  n\rightarrow\infty,\ \ \mathcal{G}(v_n)=0,\ \ \ \Psi(v_n)\leq C.
\end{equation}
Let
\begin{equation}\label{xy2}
t_n:=\frac{2\sqrt{C}}{\|\nabla f(v_n)\|^2_2}>0, \ \text{and}\ w_n=f^{-1}(t^{N/2}_nf(v_n(t_nx))).
\end{equation}
Thus $\|f(w_n)\|^2_2=\|f(v_n)\|^2_2\rightarrow0$ as $n\rightarrow\infty$. By interpolation inequality, we have that $f(w_n)\rightarrow0$ in $L^r(\R^N)$ for all $r\in(2,22^*)$.  From this and Remark \ref{remark1.1}, it follows that
\begin{equation}\label{xy3}
\int_{\R^N}H(f(w_n))\text{d}x\rightarrow0\ \ \text{as}\ \ n\rightarrow\infty.
\end{equation}
 In view of \x{xy1}-\x{xy3} and Lemma \ref{lemma3.3}, we derive
 \begin{equation*}
  \aligned
C\geq\Psi(v_n)&\geq\Psi(w_n)\\
&=\frac{1}{2}\|\nabla w_n\|^2_2-\int_{\R^N}H(f(w_n))\text{d}x\\
&=\frac{t^2_n}{2}\int_{\R^N}\frac{1+2t^N_nf^2(v_n)}{1+2f^2(v_n)}|\nabla v_n|^2\text{d}x+o_n(1)\\
&\geq\frac{t^2_n}{2}\|\nabla f(v_n)\|^2_2+o_n(1)\\
&\geq 2C+o_n(1),
 \endaligned
 \end{equation*}
which is a contraction. Thus $\lim\limits_{a\rightarrow0^+}\sigma(a)=+\infty$.

On the  other hand, we choose a  positive function $0<v\in H^1(\R^N)\cap L^\infty(\R^N)$ such that $$\int_{\R^N}|f(v)|^2\text{d}x=1.$$
For $a>0$, let $u_a=f^{-1}(\sqrt{a}f(v))$. Then $u_a\in\mathcal{S}_a$. By Lemma \ref{lemma3.7}-$(i)$, one has that there exists $t_a>0$ such that $(u_a)_{t_a}\in\mathcal{N}_a$.

Now, we claim that $t^{N/2}_{a}\sqrt{a}\rightarrow0$ as $a\rightarrow+\infty$. If not, there exists a sequence $\{a_n\}$ with $a_n\rightarrow+\infty$ as $n \rightarrow+\infty$ and $t^{N/2}_{a_n}\sqrt{a_n}\geq\delta>0$, for any $n\in\mathbb{N}$. Since $\mathcal{G}((u_{a_n})_{t_{a_n}})=0$, one has
\begin{equation*}
\aligned
0&=t_{a_n}^2\int_{\R^N}\frac{1+2t^N_{a_n}f^2(u_{a_n})}{1+2f^2(u_{a_n})}|\nabla u_{a_n}|^2\text{d}x+\frac{Nt^{N+2}_{a_n}}{2}\int_{\R^N}\frac{2f^2(u_{a_n})}{1+2f^2(u_{a_n})}|\nabla u_{a_n}|^2\text{d}x\\
&\quad\quad\quad\quad\quad\quad\quad-\frac{N}{2}t^{-N}_{a_n}\int_{\R^N}[h(t^{N/2}_{a_n}f(u_{a_n}))t^{N/2}_{a_n}f(u_{a_n})-2H(t^{N/2}_{a_n}f(u_{a_n}))]\text{d}x\\
&\leq t_{a_n}^{N+2}a^2_n\bigg[\left(\frac{1}{t_{a_n}^{N/2}\sqrt{a_n}}\right)^2\|\nabla f(v)\|_2^2+\frac{N+2}{2}\int_{\R^N}\frac{2f^2(v)}{1+2f^2(v)}|\nabla v|^2\text{d}x\\
&\quad\quad\quad\quad\quad\quad\quad\quad\quad+\frac{N(2-\mu_1)}{2}t_{a_n}^{-(2N+2)}a^{-2}_n\int_{\R^N}H(t^{N/2}_{a_n}\sqrt{a_n}f(v))\text{d}x\bigg].
 \endaligned
\end{equation*}
This, together with  Remark \ref{remark1.1}, imply that if $t^{N/2}_{a_n}\sqrt{a_n}\in[0,1]$, then
\begin{equation*}
\aligned
0&\leq t_{a_n}^{N+2}a^2_n\bigg[\left(\frac{1}{t_{a_n}^{N/2}\sqrt{a_n}}\right)^2\|\nabla f(v)\|_2^2+\frac{N+2}{2}\int_{\R^N}\frac{2f^2(v)}{1+2f^2(v)}|\nabla v|^2\text{d}x\\
&\quad\quad\quad\quad\quad\quad\quad\quad\quad+\frac{N(2-\mu_1)}{2}\left(t_{a_n}^{N/2}\sqrt{a_n}\right)^{\mu_2-\left(4+\frac{4}{N}\right)}a_n^{2/N}\int_{\R^N}H(f(v))\text{d}x\bigg],
 \endaligned
\end{equation*}
which is a contradiction.  Moreover, if $t^{N/2}_{a_n}\sqrt{a_n}\in[1,+\infty)$, then
\begin{equation*}
\aligned
0&\leq t_{a_n}^{N+2}a^2_n\bigg[\left(\frac{1}{t_{a_n}^{N/2}\sqrt{a_n}}\right)^2\|\nabla f(v)\|_2^2+\frac{N+2}{2}\int_{\R^N}\frac{2f^2(v)}{1+2f^2(v)}|\nabla v|^2\text{d}x\\
&\quad\quad\quad\quad\quad\quad\quad\quad\quad+\frac{N(2-\mu_1)}{2}\left(t_{a_n}^{N/2}\sqrt{a_n}\right)^{\mu_1-\left(4+\frac{4}{N}\right)}a_n^{2/N}\int_{\R^N}H(f(v))\text{d}x\bigg],
 \endaligned
\end{equation*}
which is a contradiction. Thus the claim holds.  From this, we know that $t_a\rightarrow0$ as $a\rightarrow+\infty$.

Next, we show that $t_{a}\sqrt{a}\rightarrow0$ as $a\rightarrow+\infty$. Suppose by contradiction, up to a subsequence,  $t_{a}\sqrt{a}$ as $a\rightarrow+\infty$ has a positive lower bound.

{\bf Case 1: $N=2$. } Since $t^{N/2}_{a}\sqrt{a}\rightarrow0$ as $a\rightarrow+\infty$, we can get the result.

{\bf Case 2: $N=3$.}\ \  By Remark \ref{remark1.1}, we have that
\begin{equation*}
\lim_{|t|\rightarrow 0}\frac{H(t)}{|t|^6}=+\infty,
\end{equation*}
which shows that, for any $\epsilon>0$ small enough, there exists $\delta_\epsilon<1$ such that
\begin{equation*}
H(t)\geq\epsilon^{-1} |t|^6\ \ \forall\,|t|\leq\delta_\epsilon.
\end{equation*}
It follows from $\mathcal{G}((u_{a_n})_{t_{a_n}})=0$ that
\begin{equation*}
\aligned
0&=t_{a_n}^2\int_{\R^N}\frac{1+2t^3_{a_n}f^2(u_{a_n})}{1+2f^2(u_{a_n})}|\nabla u_{a_n}|^2\text{d}x+\frac{3t^{5}_{a_n}}{2}\int_{\R^N}\frac{2f^2(u_{a_n})}{1+2f^2(u_{a_n})}|\nabla u_{a_n}|^2\text{d}x\\
&\quad\quad\quad\quad\quad\quad\quad\quad\quad-\frac{3}{2}t^{-3}_{a_n}\int_{\R^N}[h(t^{3/2}_{a_n}f(u_{a_n}))t^{3/2}_{a_n}f(u_{a_n})-2H(t^{3/2}_{a_n}f(u_{a_n}))]\text{d}x\\
&\leq t_{a_n}^2\int_{\R^N}\frac{1+2t^3_{a_n}f^2(u_{a_n})}{1+2f^2(u_{a_n})}|\nabla u_{a_n}|^2\text{d}x+\frac{3t^{5}_{a_n}}{2}\int_{\R^N}\frac{2f^2(u_{a_n})}{1+2f^2(u_{a_n})}|\nabla u_{a_n}|^2\text{d}x\\
&\quad\quad\quad\quad\quad-\frac{3}{2}t^{-3}_{a_n}(\mu_1-2)\int_{\R^N}H(t^{3/2}_{a_n}f(u_{a_n}))\text{d}x,
 \endaligned
\end{equation*}
which implies that
\begin{equation*}
\aligned
t^2_{a_n}{a_n}\|\nabla f(v)\|^2+\frac{5}{2}(t^2_{a_n}{a_n})^2\int_{\R^N}\frac{2f^2(v)}{1+2f^2(v)}|\nabla v|^2\text{d}x&\geq\frac{3}{2}t^{-3}_{a_n}(\mu_1-2)\int_{\R^N}H(t^{3/2}_{a_n}\sqrt{a_n}f(v))\text{d}x\\
&\geq\epsilon^{-1}\frac{3}{2}(\mu_1-2)(t^{2}_{a_n}a_n)^3\int_{\R^N}|f(v)|^6\text{d}x.
 \endaligned
\end{equation*}
This contradicts the arbitrariness of $\epsilon$.

Recalling that $(u_a)_{t_a}\in\mathcal{N}_a$ and from the above discussion, we infer that
\begin{equation}\label{aaahhh}
\aligned
0&<\sigma(a)\\
&\leq\Psi((u_a)_{t_a})=\frac{t^2_a a}{2}\int_{\R^N}\frac{1+2t^N_aaf^2(v)}{1+2f^2(v)}|\nabla v|^2\text{d}x- t^{-N}_a\int_{\R^N}H(t^{N/2}_a\sqrt{a}f(v))\text{d}x\\
&\leq\frac{t^2_a a}{2}\int_{\R^N}\frac{1+2t^N_aaf^2(v)}{1+2f^2(v)}|\nabla v|^2\text{d}x\\
&=\frac{t^2_a a}{2}\left\{\left[1-\left(t^{N/2}_a \sqrt{a}\right)^2\right]\|\nabla f(v)\|^2_2+(t^{N/2}_a \sqrt{a})^2\|\nabla v\|^2_2\right\}.
 \endaligned
\end{equation}
Since $t_a \sqrt{a}\rightarrow0$ as $a\rightarrow+\infty$, \x{aaahhh} shows that $
\lim\limits_{a\rightarrow+\infty}\sigma(a)=0
$. $\hfill\Box$

\par\noindent
{\bf Proof of Theorem \ref{theoremA}.}\,\,  Suppose by a contradiction, $(u=f(v),\lambda)\in \mathcal{\hat{S}}_a\times(-\infty,0)$ is a solution of \x{p}. Thus
\begin{equation}\label{mx}
\Psi'(v)-\lambda f(v)f'(v)=0 \ \ \text{in}\ \ (H^1(\R^N))^*.
\end{equation}
In \x{mx} testing with $\frac{f(v)}{f'(v)}$, we have
\begin{equation}\label{mx1}
\int_{\R^N}|\nabla v|^2\text{d}x+\int_{\R^N}\frac{2f^2(v)}{1+2f^2(v)}|\nabla v|^2\text{d}x-\lambda\int_{\R^N}|f(v)|^2\text{d}x-\int_{\R^N}h(f(v))f(v)\text{d}x=0.
\end{equation}
Since $v$ is a solution of $-\Delta v-\lambda f(v)f'(v)=h(f(v))f'(v)$, it follows that $v$ the following equality:
\begin{equation}\label{mmx2}
\frac{N-2}{2}\int_{\R^N}|\nabla v|^2\text{d}x-\frac{N}{2}\lambda\int_{\R^N}|f(v)|^2\text{d}x-N\int_{\R^N}H(f(v))\text{d}x=0.
\end{equation}
Combining \x{mx1} with \x{mmx2}, we  have
\begin{equation*}
\aligned
0&\leq\left(\frac{N-2}{2N}-\frac{2}{\mu_1}\right)\|\nabla v\|^2_2+\frac{1}{\mu_1}\|\nabla f(v)\|^2_2\\
&=\left(\frac{1}{2}-\frac{1}{\mu_1}\right)\lambda a+\frac{1}{\mu_1}\int_{\R^N}\left[\mu_1H(f(v))-h(f(v))f(v)\right]\text{d}x\\
&\leq\left(\frac{1}{2}-\frac{1}{\mu_1}\right)\lambda a<0,
 \endaligned
\end{equation*}
which is a contradiction. $\hfill\Box$

\section{A local minimizer for Sobolev critical exponent $22^*$ and $N=3$}
In this part, we will show the existence of ground state normalized solutions with  $L^2$-subcritical growth and Sobolev critical exponent. Since the ground state energy of the perturbed functional may not coincide with that of the original functional, perturbation methods cannot ensure the existence of ground state solutions. Based on this, we propose to employ a dual method to prove the existence of  ground state normalized solutions.

For simplicity, we consider $h(u)=|u|^{p-2}u+|u|^{22^*-2}u$ and $N=3$. Now, similar to \cite{JSLa,Chen3}, let the function $\rho_a: (0,+\infty)\rightarrow\R$ be
\begin{equation}\label{9-1}
\aligned
\rho_a(t)&=\frac{1}{2}-\frac{C^p_{p,3}a^{\frac{6-p}{4}}}{p}t^{\frac{3p-10}{4}}-\frac{2}{3S^{3}}t^{2}\\
&:=\frac{1}{2}-Aa^{\frac{6-p}{4}} t^{\frac{3p-10}{4}}-Bt^{2},
\endaligned
\end{equation}
where
\begin{equation*}
A=\frac{C^p_{p,3}}{p}\ \ \text{and}\ \ \ B=\frac{2}{3S^{3}}
\end{equation*}
and $S$ is the best constant corresponding to the embedding $D^{1,2}(\R^3)\hookrightarrow L^{2^*}(\R^3)$.

\begin{lemma}\label{lemma9.1}
Assume that  $2<p<\frac{10}{3}$ holds. For each  $a>0$, the function $\rho_a(t)$ has a unique global maximum and the maximum satisfies
\begin{displaymath}
\max_{t\in(0,+\infty)}\rho_a(t)\left\{
\begin{array}{ll}
>0, & \textrm{if $a<\tilde{a}_0$},\\
=0,& \textrm{if $a=\tilde{a}_0$},\\
<0,& \textrm{if $a>\tilde{a}_0$},
\end{array} \right.\\
\end{displaymath}
where
\begin{equation*}
\tilde{a}_0:=\left(\frac{1}{2K_0}\right)^{\frac{3}{2}}>0
\end{equation*}
and
\begin{equation*}
K_0:=\left[\frac{A(10-3p)}{8B}\right]^{\frac{3p-10}{18-3p}}\left[\frac{A(18-3p)}{8}\right]>0.
\end{equation*}
\end{lemma}
\begin{proof}
By \x{9-1}, we can get
\begin{equation*}
\frac{\text{d}\rho_a(t)}{\text{d}t}=Aa^{\frac{6-p}{4}} \frac{10-3p}{4}t^{\frac{3p-14}{4}}-2Bt.
\end{equation*}
It follows from $\frac{\text{d}\rho_a(t)}{\text{d}t}=0$ that there exists
\begin{equation*}
t^*_a:=\left(\frac{A(10-3p)}{8B}\right)^{\frac{4}{18-3p}}a^{\frac{1}{3}}
\end{equation*}
such that $t^*_a$ is a unique global maximum since $\rho_a(t)\rightarrow-\infty$ as $t\rightarrow0^+$ and $\rho_a(t)\rightarrow-\infty$ as $t\rightarrow+\infty$. Furthermore, its global maximum value is given by
 \begin{equation*}
\aligned
\max_{t\in(0,+\infty)}\rho_a(t)
&=\rho_a(t^*_a)=\frac{1}{2}-K_0 a^{\frac{2}{3}},
 \endaligned
 \end{equation*}
where
\begin{equation*}
K_0=:\left[\frac{A(10-3p)}{8B}\right]^{\frac{3p-10}{18-3p}}\left[\frac{A(18-3p)}{8}\right]>0.
\end{equation*}
Thus there exists $\tilde{a}_0=\left(\frac{1}{2K_0}\right)^{\frac{3}{2}}>0$ such that the conclusion holds.
\end{proof}

\begin{lemma}\label{lemma9.2}
Assume that  $2<p<\frac{10}{3}$ holds. Then for any $a>0$ and  $v\in\mathcal{S}_a$, there holds:
\begin{equation*}
\Psi(v)\geq\|\nabla v\|^2_2\rho_a\left(\|\nabla v\|^2_2\right).
\end{equation*}
\end{lemma}
\begin{proof}
By \x{f2}, Lemma \ref{lemma1.0}, Lemma \ref{lemma2.1}-(8) and Sobolev inequality, we have
 \begin{equation*}
\aligned
\Psi(v)&=\frac{1}{2}\|\nabla v\|^2_2-\frac{1}{p}\|f(v)\|^p_p-\frac{1}{12}\|f(v)\|^{12}_{12}\\
&\geq\frac{1}{2}\|\nabla v\|^2_2-\frac{C^p_{p,3}a^{\frac{6-p}{4}}}{p}\|\nabla f(v)\|^{\frac{3(p-2)}{2}}_2-\frac{2}{3S^3}\|\nabla v\|^{6}_{2}\\
&\geq\|\nabla v\|^2_2\left[\frac{1}{2}-\frac{C^p_{p,3}a^{\frac{6-p}{4}}}{p}\|\nabla  v\|^{\frac{3p-10}{2}}_2-\frac{2}{3S^3}\|\nabla v\|^{4}_{2}\right]\\
&=\|\nabla v\|^2_2\rho_a\left(\|\nabla v\|^2_2\right)\ \ \text{for each }\ \ v\in \mathcal{S}_a.
 \endaligned
 \end{equation*}
 The proof is completed.
\end{proof}

Let $t^*_0:=t^*_{\tilde{a}_0}>0$. Now, set
\begin{equation*}
\Lambda_{\ell}:=\left\{ v\in H^1(\R^3): \|\nabla v\|^2_2<\ell\right\}
\end{equation*}
and
\begin{equation*}
m_0(a):=\inf_{v\in \mathcal{S}_a\cap\Lambda_{t^*_0}}\Psi(v).
\end{equation*}

\begin{lemma}\label{lemma9.3}
Assume that  $2<p<\frac{10}{3}$ holds. Then for any $a\in (0,\tilde{a}_0)$, the following conclusion hold
\begin{equation*}
m_0(a)=\inf_{v\in \mathcal{S}_a\cap\Lambda_{t^*_0}}\Psi(v)<0<\inf_{v\in \partial(\mathcal{S}_a\cap\Lambda_{t^*_0})}\Psi(v).
\end{equation*}
\end{lemma}
\begin{proof}
On the one hand, let $v\in \mathcal{S}_a$ be fixed. Then
\begin{equation*}
v_s(x):=f^{-1}(s^{3/2}f(v(sx)))\in  \mathcal{S}_a,\ \ \forall\, s>0.
\end{equation*}
and
\begin{equation*}
\Psi(v_s)=\frac{s^2}{2}\int_{\R^3}\frac{1+2s^3f^2(v)}{1+2f^2(v)}|\nabla v|^2\text{d}x-\frac{s^{\frac{3(p-2)}{2}}}{p}\int_{\R^3}|f(v)|^p\text{d}x-\frac{s^{15}}{12}\int_{\R^3}|f(v)|^{12}\text{d}x.
\end{equation*}
Thus there exists $0<s_0<1$ small enough such that
\begin{equation*}
\|\nabla v_{s_0}\|^2_2=s^2_0 \int_{\R^3}\frac{1+2s^3_0f^2(v)}{1+2f^2(v)}|\nabla v|^2\text{d}x\leq s^2_0 \|\nabla v\|^2_2<t^*_0.
\end{equation*}
and $\Psi(v_{s_0})<0$.
Hence it follows that  $v_{s_0}\in \mathcal{S}_a\cap\Lambda_{t^*_0}$ and $\Psi(v_{s_0})<0$ for $s_0>0$ small enough. So $m_0(a)\leq\Psi(v_{s_0})<0$. By this fact and Lemmas \ref{lemma9.1}-\ref{lemma9.2}, we can get the conclusion.
\end{proof}

\begin{lemma}\label{lemma99.3}
Assume that  $2<p<\frac{10}{3}$ holds and $a\in (0,\tilde{a}_0)$. Then for all $b\in(0,a)$, $m_0(a)\leq m_0(b)+m_0(a-b)$; and if $m_0(b)$ or $m_0(a-b)$ can be achieved, then the inequality is strict.
\end{lemma}
\begin{proof}
Let $b\in(0,a)$ be fixed. For any $\theta\in \left[1,\frac{a}{b}\right]$, it follows from  \x{9-1} and Lemma \ref{lemma9.1} that
\begin{equation}\label{9-3}
\aligned
\rho_b\left(\frac{b\theta}{a}t^*_0\right)&=\frac{1}{2}-\frac{C^p_{p,3}b^{\frac{6-p}{4}}}{p}\left(\frac{b\theta}{a}t^*_0\right)^{\frac{3p-10}{4}}
-\frac{2}{3S^{3}}\left(\frac{b\theta}{a}t^*_0\right)^{2}\\
&\geq\frac{1}{2}-\frac{C^p_{p,N}a^{\frac{6-p}{4}}}{p}\left(t^*_0\right)^{\frac{3p-10}{4}}
-\frac{2}{3S^{3}}\left(t^*_0\right)^{\frac{2^*-2}{2}}\\
&=\rho_a(t^*_0)=\rho_a(t^*_{\tilde{a}_0})>\rho_{\tilde{a}_0}(t^*_{\tilde{a}_0})=0.
\endaligned
\end{equation}
This shows that $\rho_b(t)>0$ for any $t\in\left[\frac{b}{a}t^*_0,t^*_0\right]$.
Since $m_0(b)<0$, it follows that there exists a sequence $\{v_n\}\subset\mathcal{S}_b\cap\Lambda_{t^*_0}$  such that
\begin{equation*}
\|\nabla v_n\|^2_2\rho_b\left(\|\nabla v_n\|^2_2\right)\leq\Psi(v_n)<0 \ \ \ \text{for $n$ large enough},
\end{equation*}
which together \x{9-3}, we have
\begin{equation*}
\|\nabla v_n\|^2_2\leq\frac{b}{a}t^*_0.
\end{equation*}
For $\theta\in \left(1,\frac{a}{b}\right]$, let
\begin{equation*}
\ddot{v}_n(x):=v_n\left(\theta^{-\frac{1}{3}}x\right).
\end{equation*}
By a simple calculation, we get that
\begin{equation*}
\|\nabla\ddot{v}_n\|^2_2=\theta^{\frac{1}{3}}\|\nabla v_n\|^2_2\leq\theta^{\frac{1}{3}}\frac{b}{a}t^*_0<t^*_0.
\end{equation*}
Similar to \x{section8}, we can infer that
\begin{equation*}
m_0(\theta b)\leq\theta m_0(b),\ \ \ \forall\, \theta\in \left(1,\frac{a}{b}\right].
\end{equation*}
Thus
\begin{equation*}
\aligned
m_0(a)&\leq \frac{a-b}{a}m_0\left(\frac{a}{a-b}(a-b)\right)+\frac{b}{a}m_0\left(\frac{a}{b}b\right)\\
&\leq m_0\left(a-b\right)+m_0\left(b\right)
\endaligned
\end{equation*}
and the inequality is strict if $m_0\left(a-b\right)$ or $m_0\left(b\right)$ can be achieved.
\end{proof}

\begin{lemma}\label{lemma9.4}
Assume that  $2<p<\frac{10}{3}$ holds. Then the function $a\mapsto m_0(a)$ is continuous on $(0,\tilde{a}_0)$.
\end{lemma}
\begin{proof}
The proof is similar to the proof of Lemma \ref{lemma3.7}-$(iii)$. So we omit the proof.
\end{proof}

To achieved the compactness, we need to prove a splitting lemma of minimizing sequence. At first, we begin by correcting the nonlinear term.
Let
\begin{equation}\label{r11}
\aligned
\tilde{h}(v)&=|f(v)|^{p-2}f(v)f'(v)
+|f(v)|^{10}f(v)f'(v)-4
|v|^{4}v,
\endaligned
\end{equation}
and
\begin{equation}\label{r22}
\tilde{H}(v)=\frac{1}{p}|f(v)|^{p}
+\frac{1}{12}|f(v)|^{12}
-\frac{2}{3}|v|^{6}.
\end{equation}
According to the  above,   we write Eq. \x{f1}  in the following equivalence  form
\begin{equation*}
-\Delta v=\tilde{h}(v)+ 4|v|^{4}v+\lambda f(v)f'(v)
\end{equation*}
with prescribed mass
\begin{equation*}
\int_{\R^3}|f(v)|^2\text{d}x=a,
\end{equation*}
and the functional is also given by
\begin{equation}\label{r1}
\Psi(v)=\frac{1}{2}\int_{\R^3}|\nabla v|^2\text{d}x
-\int_{\R^3}\tilde{H}(v)\text{d}x
-\frac{2}{3}\int_{\R^3}|v|^{6}\text{d}x.
\end{equation}

\begin{lemma}\label{lemma3.35}For any $2<p<\frac{10}{3}$, the following statement holds:
\begin{equation*}
\aligned
&\lim_{|v|\rightarrow0}\frac{\tilde{H}(v)}{v^2}=0,\ \  \lim_{|v|\rightarrow+\infty}\frac{\tilde{H}(v)}{|v|^{6}}
=0, \\
&\lim_{|v|\rightarrow0}\frac{\tilde{h}(v)}{|v|}=0,\ \  \lim_{|v|\rightarrow+\infty}\frac{\tilde{h}(v)}{|v|^{5}}=0.
 \endaligned
\end{equation*}
\end{lemma}
\begin{proof}
To prove our results, we only need to show the first equality. In fact,
\begin{equation*}
\lim_{|v|\rightarrow0}\frac{\tilde{H}(v)}{v^2}=\lim_{|v|\rightarrow0}\left(\frac{1}{p}\frac{|f(v)|^{p-2}f^2(v)}{v^2}+\frac{1}{12}\frac{|f(v)|^{10}f^2(v)}{v^2}
-\frac{2}{3}|v|^{4}\right)=0
\end{equation*}
and
\begin{equation*}
\lim_{|v|\rightarrow+\infty}\frac{\tilde{H}(v)}{|v|^{6}}=\lim_{|v|\rightarrow+\infty}\left(\frac{1}{p}\frac{|f(v)|^{p-2}f^2(v)}{|v|^{6}}+\frac{1}{12}\frac{|f(v)|^{12}}{|v|^{6}}
-\frac{2}{3} \right)=0.
\end{equation*}
By a similar argument, we can get the second conclusion.
\end{proof}

\begin{lemma}\label{lemma3.8}
Suppose that $v_n\rightharpoonup v_0$ in $H^1(\R^3)$, then up to a subsequence,
\begin{equation}\label{ss10}
\Psi(v_n)=\Psi(v_0)+\Psi(v_n-v_0)+o_n(1).
\end{equation}
\end{lemma}
\begin{proof}
On the one hand,
\begin{equation*}
\|\nabla v_n-\nabla v_0\|^2_2=\|\nabla v_n\|^2_2-\|\nabla v_0\|^2_2+o_n(1)
\end{equation*}
and by Brezis-Lieb Lemma in \cite{MW},
\begin{equation*}
\|v_n-v_0\|^{6}_{6}=\|v_n\|^{6}_{6}-\|v_0\|^{6}_{6}+o_n(1).
\end{equation*}
On the other hand, we only need to illustrate that
\begin{equation*}
\int_{\R^3}\left(\tilde{H}(v_n-v_0)-\tilde{H}(v_n)+\tilde{H}(v_0)\right)\text{d}x\rightarrow0,\,\, \text{as}\,\, n\rightarrow+\infty.
\end{equation*}
In fact, by Lemma \ref{lemma3.35}, for any $\varepsilon>0$, there exists $C_\varepsilon>0$ such that
\begin{equation}\label{r16}
|\tilde{h}(s)|\leq 2\varepsilon |s| +6C_\varepsilon|s|^{5}\ ,\, |\tilde{H}(s)|\leq\varepsilon |s|^2 +C_\varepsilon|s|^{6}\ \ \forall\,s\in\R.
\end{equation}
Take $w_n:=v_n-v_0$, we have $w_n\rightharpoonup0$ in $H^1(\R^3)$ and $w_n(x)\rightarrow0$ a.e. on $\R^3$. It follows from \x{r16} that
\begin{equation*}
\aligned
\left|\tilde{H}(w_n+v_0)-\tilde{H}(w_n)\right|&\leq\int_0^1\left|\tilde{h}(w_n+\zeta v_0)v_0\right|\text{d}\zeta\\
&\leq\int_0^1\left[2\varepsilon|w_n+\zeta v_0||v_0|+6C_\varepsilon\left(|w_n+\zeta v_0|^{5}|v_0|\right)\right]\text{d}\zeta\\
&\leq C\left(2\varepsilon|w_n||v_0|+2\varepsilon|v_0|^2+62^{5}C_\varepsilon|w_n|^{5}|v_0|+C_\varepsilon|v_0|^{6}\right).
\endaligned
\end{equation*}
By Young's inequality, we have
\begin{equation}\label{r17}
\aligned
&\left|\tilde{H}(w_n+v_0)-\tilde{H}(w_n)\right|\leq C\left(\varepsilon|w_n|^2+C_\varepsilon|v_0|^2+\varepsilon|w_n|^{6}+C_\varepsilon|v_0|^{6}\right).
\endaligned
\end{equation}
Combining \x{r16} and \x{r17}, one has
\begin{equation*}
\aligned
&\left|\tilde{H}(w_n+v_0)-\tilde{H}(w_n)-\tilde{H}(v_0)\right|\leq C\left(\varepsilon|w_n|^2+C_\varepsilon|v_0|^2+\varepsilon|w_n|^{6}+C_\varepsilon|v_0|^{6}\right).
\endaligned
\end{equation*}
Let
\begin{equation*}
\aligned
\digamma_n(x)=\max\bigg\{\left|\tilde{H}(w_n+v_0)-\tilde{H}(w_n)-\tilde{H}(v_0)\right|-C\varepsilon\left(|w_n|^2+|w_n|^{6}\right),0\bigg\}.
\endaligned
\end{equation*}
Then
\begin{equation*}
0\leq\digamma_n(x)\leq \left(C_\varepsilon|v_0|^2+C_\varepsilon|v_0|^{6}\right)\in L^1(\R^3).
\end{equation*}
Thus by Lebesgue's dominated convergence theorem, we have
\begin{equation*}
 \int_{\R^3}\digamma_n(x)\text{d}x\rightarrow0, \ \ \text{as}\ \  n\rightarrow\infty.
\end{equation*}
Moreover, by the definition of $\digamma_n(x)$, we get
\begin{equation*}
\left|\tilde{H}(w_n+v_0)-\tilde{H}(w_n)-\tilde{H}(v_0)\right|\leq C\varepsilon\left(|w_n|^2+|w_n|^{6}\right)+\digamma_n(x),
\end{equation*}
which shows that
\begin{equation*}
\left|\int_{\R^N}\left(\tilde{H}(w_n+v_0)-\tilde{H}(w_n)-\tilde{H}(v_0)\right)\text{d}x\right|\rightarrow0,\ \ \text{as}\ \  n\rightarrow\infty.
\end{equation*}
The proof of this lemma is completed.
\end{proof}

\par\noindent
{\bf Proof of Theorem \ref{theoremeX}.}\,\, Let $\{v_n\}\subset\mathcal{S}_a\cap \Lambda_{t^*_0} $ be a minimizing sequence of $m_0(a)$.  By Lemma \ref{lemma9.3}, we have
\begin{equation}\label{9-6}
\|\nabla v_n\|^2_2<t^*_0,\ \ \ \|f(v_n)\|^2_2=a,\ \ \Psi(v_n)=m_0(a)+o_n(1)<0.
\end{equation}
To complete the proof of this theorem, we split the proof into several steps.

Let
\begin{equation*}
\bar{\delta}:=\limsup_{n\rightarrow\infty}\sup_{y\in\R^3}\int_{B_1(y)}|v_n|^2\text{d}x.
\end{equation*}
Suppose that $\bar{\delta}=0$. Then by Lions' concentration compactness principle, we have that $v_n\rightarrow0$ in $L^r(\R^3)$ for all $r\in(2,6)$. Thus for any $2<p<\frac{10}{3}$ and Lemma \ref{lemma2.3}, we have
\begin{equation}\label{9-7}
\int_{\R^3}|f(v_n)|^p\text{d}x\rightarrow0\ \ \text{as}\ \ n\rightarrow\infty.
\end{equation}
From \x{9-6}, \x{9-7} and Sobolev inequality, we get
 \begin{equation}\label{9-8}
 \aligned
0>m_0(a)&=\frac{1}{2}\|\nabla v_n\|^2_2-\frac{1}{p}\|f(v_n)\|^p_p-\frac{1}{12}\|f(v_n)\|^{12}_{12}+o_n(1)\\
&\geq\|\nabla v\|^2_2\left(\frac{1}{2}-\frac{2}{3S^{3}}\|\nabla v\|^{4}_{2}\right)+o_n(1)\\
&\geq\|\nabla v\|^2_2\left[\frac{1}{2}-\frac{2}{3S^{3}}(t^*_0)^{2}\right]+o_n(1)\\
&=\|\nabla v\|^2_2\left[\rho_a(t^*_0)+\frac{C^p_{p,3}a^{\frac{6-p}{4}}}{p}(t^*_0)^{\frac{3p-10}{4}}\right]+o_n(1)\\
&\geq o_n(1),
\endaligned
 \end{equation}
which is a contradiction. Thus $\delta>0$.

Up to a subsequence, there exists a sequence $\{y_n\}\subset\R^3$ such that
\begin{equation*}
\int_{B_1(y_n)}|v_n|^2\text{d}x>\frac{\delta}{2}.
\end{equation*}
Let $\hat{\upsilon}_n(x):=v_n(x+y_n)$. Then there exists $\hat{\upsilon}\in H^1(\R^3)\setminus\{0\}$ such that, passing to subsequence,
\begin{equation*}
\hat{\upsilon}_n\rightharpoonup\hat{\upsilon}\ \text{in $H^1(\R^3)$,}\ \ \hat{\upsilon}_n\rightarrow\hat{\upsilon}\ \text{in $L^r_{loc}(\R^3)$ for all $r\in [2,2^*)$},
\ \hat{\upsilon}_n\rightarrow\hat{\upsilon}\ \text{a.e. on $\R^3$.}
\end{equation*}
From \x{9-6}, we deduce that
\begin{equation}\label{9-10}
\|\nabla \hat{\upsilon}_n\|^2_2<t^*_0,\ \ \ 0<\|f(\hat{\upsilon})\|^2_2\leq\|f(\hat{\upsilon}_n)\|^2_2=a,\ \ \Psi(\hat{\upsilon}_n)=m_0(a)+o_n(1)<0.
\end{equation}

Set $\tilde{w}_n:=\hat{\upsilon}_n-\hat{\upsilon}$. By \x{9-10} and Lemma \ref{lemma3.8}, we know that
 \begin{equation}\label{9-11}
\Psi(\hat{\upsilon}_n)=\Psi(\hat{\upsilon})+\Psi(\tilde{w}_n)+o_n(1),
 \end{equation}
 \begin{equation}\label{9-12}
\|\nabla\hat{\upsilon}_n\|^2_2=\|\nabla\hat{\upsilon}\|^2_2+\|\nabla \tilde{w}_n \|^2_2+o_n(1),
 \end{equation}
 and
\begin{equation}\label{9-13}
\|f(\tilde{w}_n)\|^2_2=\|f(\hat{\upsilon}_n)\|^2_2-\|f(\hat{\upsilon})\|^2_2+o_n(1)=a-\|f(\hat{\upsilon})\|^2_2+o_n(1).
 \end{equation}
Next, we claim that $\|f(\tilde{w}_n)\|^2_2\rightarrow0$ as $n\rightarrow\infty$. In fact, let $\bar{c}:=\|f(\hat{\upsilon})\|^2_2\leq a$. If $\bar{c}=a$, then the claim holds. Suppose that $\bar{c}<a$. In view of \x{9-12}, \x{9-13} and \x{9-10}, for $n$ large enough, one has
\begin{equation}\label{9-14}
\beta_n:=\|f(\tilde{w}_n)\|^2_2\leq a,\ \  \ \|\nabla \tilde{w}_n \|^2_2\leq\|\nabla\hat{\upsilon}_n\|^2_2<t^*_0.
\end{equation}
By \x{9-14}, we have
\begin{equation}\label{9-15}
\tilde{w}_n\in \mathcal{S}_{\beta_n}\cap\Lambda_{t^*_0}, \ \ \ \Psi(\tilde{w}_n)\geq m_0(\beta_n):=\inf_{v\in\mathcal{S}_{\beta_n}\cap\Lambda_{t^*_0}}\Psi(v).
\end{equation}
From \x{9-10}, \x{9-11} and \x{9-15}, we have
\begin{equation}\label{9-16}
\aligned
m_0(a)+o_n(1)&=\Psi(\hat{\upsilon}_n)=\Psi(\hat{\upsilon})+\Psi(\tilde{w}_n)+o_n(1)\\
&\geq\Psi(\hat{\upsilon})+m_0(\beta_n)+o_n(1).
\endaligned
\end{equation}
By virtue of \x{9-16}, Lemma \ref{lemma9.4} and \x{9-13}, it follows from the continuity of $m_0(a)$ on $a\in(0,\tilde{a}_0)$ that
\begin{equation}\label{9-17}
m_0(a)\geq\Psi(\hat{\upsilon})+m_0(a-\bar{c}).
\end{equation}
Moreover, by \x{9-14} and weak lower semi-continuity of norm, we infer that $\hat{\upsilon}\in \mathcal{S}_{\bar{c}}\cap \overline{\Lambda_{t^*_0}}$. Thus $\Psi(\hat{\upsilon})\geq m_0(\bar{c})$. If $\Psi(\hat{\upsilon})> m_0(\bar{c})$, then it follows from \x{9-17} and Lemma \ref{lemma99.3} that
\begin{equation*}
m_0(a)>m_0(\bar{c})+m_0(a-\bar{c})\geq m_0(a),
\end{equation*}
which is a contradiction. Hence $m_0(\bar{c})$ is achieved at $\hat{\upsilon}$. By Lemma \ref{lemma99.3} with strict inequality, it follows from \x{9-17} that
\begin{equation*}
m_0(a)\geq m_0(\bar{c})+m_0(a-\bar{c})>m_0(a).
\end{equation*}
This is a contraction. So $\|f(\hat{\upsilon})\|^2_2=a$ and $\hat{\upsilon}\in \mathcal{S}_{a}\cap \overline{\Lambda_{t^*_0}}$. Thus the claim holds. From the claim, we infer that $\|f(\tilde{w}_n)\|^p_p\rightarrow0$ as $n\rightarrow\infty$ for all $2<p<2+\frac{4}{N}$. Since $\Psi(\hat{\upsilon})\geq m_0(a)$, it follows from \x{9-10} and \x{9-11} that
\begin{equation}\label{9-18}
\aligned
m_0(a)+o_n(1)&=\Psi(\hat{\upsilon}_n)=\Psi(\hat{\upsilon})+\Psi(\tilde{w}_n)+o_n(1)\\
&\geq  m_0(a)+\frac{1}{2}\|\nabla \tilde{w}_n\|^2_2-\frac{1}{p}\|f(\tilde{w}_n)\|^p_p-\frac{1}{12}\|f(\tilde{w}_n)\|^{12}_{12}\\
&\geq m_0(a)+\|\nabla \tilde{w}_n\|^2_2\left(\frac{1}{2}-\frac{2}{3S^3}\|\nabla \tilde{w}_n\|^{4}_{2}\right)+o_n(1)\\
&\geq m_0(a)+\|\nabla \tilde{w}_n\|^2_2\left[\frac{1}{2}-\frac{2}{3S^3}(t^*_0)^{2}\right]+o_n(1)\\
&= m_0(a)+\|\nabla \tilde{w}_n\|^2_2\left[\rho_a(t^*_0)+\frac{C^p_{p,3}a^{\frac{6-p}{4}}}{p}(t^*_0)^{\frac{3p-10}{4}}\right]+o_n(1),
\endaligned
\end{equation}
which implies that $\|\nabla \tilde{w}_n\|^2_2\rightarrow0$ as $n\rightarrow\infty$, which, together with $\|f(\tilde{w}_n)\|^2_2\rightarrow0$ as $n\rightarrow\infty$, implies that
\begin{equation*}
\int_{\R^3}(|\nabla \tilde{w}_n|^2+f^{2}(\tilde{w}_n))\text{d}x\rightarrow0\ \ \text{as $n\rightarrow\infty$}.
\end{equation*}
By \x{henzhong}, we have that $\tilde{w}_n\rightarrow0$ in $H^1(\R^3)$, that is, $\hat{\upsilon}_n\rightarrow\hat{\upsilon}$ in $H^1(\R^3)$. Hence
\begin{equation*}
\|f(\hat{\upsilon})\|^2_2=a,\ \ \ \ \|\nabla\hat{\upsilon}\|^2_2\leq t^*_0,\ \ \ \ m_0(\bar{c})=\Psi(\hat{\upsilon}).
\end{equation*}
From Lemma \ref{lemma9.3}, we deduce that $\|\nabla\hat{\upsilon}\|^2_2<t^*_0$. By Corollary 2.4 in \cite{Chen3}, we infer that $\Psi|'_{\mathcal{S}_a}=0$, which, together with Lagrange's multiplier rule, implies that the proof is completed. $\hfill\Box$

\section{Ground state normalized solutions with  critical exponential growth and $N=2$}
In this section, we want to prove the existence of radial ground state normalized solutions with critical exponential growth in $\R^2$. For this we study $\Psi(v)$ on radial space on $H^1_r(\R^2)$.  Since this space can embed  compactly in $L^r(\R^2)$ for all $r\in(2,+\infty)$, it can help us to recover the compactness. From this and Palais' symmetric principle in \cite{MW}, the critical point in $H^1_r(\R^2)$ is also the critical point in $H^1(\R^2)$.

Now, we need to establish the following two inequalities.
For $\zeta>\zeta_0$, it follows from $(\mathcal{H}_1)$ and $(H_1)$ that for any $\varepsilon>0$, there exists $C_\varepsilon>0$ such that
\begin{equation}\label{A1}
h(t)t\leq\varepsilon|t|^6+C_\varepsilon|t|^6\left[\exp{(\zeta t^{4})}-1\right]
\end{equation}
and
\begin{equation}\label{A2}
H(t)\leq\frac{\varepsilon}{4}|t|^6+\bar{C}_\varepsilon|t|^6\left[\exp{(\zeta t^{4})}-1\right].
\end{equation}

To our aim, we first introduce the following lemmas, which are proved via a similar argument as Lemmas \ref{lemma3.1}-\ref{lemma3.7}.

\begin{lemma}\label{lemma6.1}
Let
\begin{equation*}
\aligned
\bar{A}(t,v)&=\frac{1}{2}\bigg[\left(1-\frac{t^2(1+2t^2f^2(v))}{1+2f^2(v)}\right)|\nabla v|^2-\frac{1-t^{4}}{4}\left( 2
+\frac{4f^2(v)}{1+2f^2(v)}\right)|\nabla v|^2\bigg].
\endaligned
\end{equation*}
Then for any  $t>0$ and $v\in H^1_r(\R^2)\backslash\{0\}$, $\bar{A}(t,v)\geq 0$.
\end{lemma}

\begin{lemma}\label{lemma6.2}
Assume that  $(h_3)$ holds and
\begin{equation*}
\aligned
\bar{B}(t,v)&:=-H(f(v))+t^{-2}H(tf(v))+\frac{1-t^4}{4}[h(f(v))f(v)-2H(f(v))].
\endaligned
\end{equation*}
Then $\bar{B}(t,v)\geq0$, for any  $t>0$ and $v\in H^1_r(\R^2)\backslash\{0\}$.
\end{lemma}
\begin{remark}\label{remark6.2}
In Lemma \ref{lemma6.2}, by letting $t\rightarrow0$, we can get that $$0<6H(f(v))\leq h(f(v))f(v)$$ for all $v\in H^1_r(\R^2)$.
\end{remark}

\begin{lemma}\label{lemma6.4}
Assume that $(h_3)$ holds. Then
\begin{equation}\label{sss5}
\aligned
\Psi(v)&=\Psi(v_t)+\frac{1-t^{4}}{4}\mathcal{G}(v)+\int_{\R^2}\bar{A}(t,v)\text{d}x+\int_{\R^2}\bar{B}(t,v)\text{d}x,
\endaligned
\end{equation}
for any  $t>0$ and $v\in H^1_r(\R^2)$.
\end{lemma}

 \begin{lemma}\label{lemma6.3}
Assume that $(h_3)$ holds. Then
\begin{equation}\label{s4}
\Psi(v)=\max_{t>0}\Psi(v_t), \ \ \ \forall\,v\in \mathcal{N}_a.
\end{equation}
\end{lemma}

\begin{lemma}\label{lemma6.04}
 Assume that $(H_1)$ and $(h_3)$ hold.  Then
 for any $v\in \mathcal{S}_a$, there exists a unique $t_v>0$ such that  $v_{t_v}\in\mathcal{N}_a$.
 \end{lemma}

\begin{lemma}\label{lemma6.5}
Assume that $(h_3)$ holds. Then
\begin{equation}\label{shsss4}
\sigma^*(a):=\inf_{v\in\mathcal{N}_a}\Psi(v)=\inf_{v\in\mathcal{S}_a}\max_{t>0}\Psi(v_t)
\end{equation}
\end{lemma}
 \begin{proof}
On the one hand, for any $v\in\mathcal{N}_a$, by Lemma \ref{lemma6.3}, we have
\begin{equation*}
\Psi(v)=\max_{t>0}\Psi(v_t),
\end{equation*}
which implies that
\begin{equation*}
\inf_{v\in\mathcal{N}_a}\Psi(v)\geq\inf_{v\in\mathcal{S}_a}\max_{t>0}\Psi(v_t).
\end{equation*}
On the other hand, for any $v\in\mathcal{S}_a$, by Lemma \ref{lemma6.4}, there exists a unique $t_v>0$ such that  $v_{t_v}\in\mathcal{N}_a$. This shows that for any $v\in\mathcal{S}_a$,
\begin{equation*}
\max_{t>0}\Psi(v_t)\geq\Psi(v_{t_v})\geq\inf_{v\in\mathcal{N}_a}\Psi(v)\Rightarrow\inf_{v\in\mathcal{N}_a}\Psi(v)\leq\inf_{v\in\mathcal{S}_a}\max_{t>0}\Psi(v_t).
\end{equation*}
Based on the above two facts, the proof is completed.
 \end{proof}

\begin{lemma}\label{lemma6.6} Assume that $(\mathcal{H}_1)$ and $(H_1)$-$(H_4)$ hold. Then

(i) there exists a constant $\varrho>0$ such that $\|\nabla v\|_2^2\geq\varrho$  for any $v\in \mathcal{N}_a$;

(ii) $\sigma^*(a):=\inf\limits_{v\in\mathcal{N}_a}\Psi(v)>0$.
 \end{lemma}
\begin{proof}
$(i)$  Suppose, by contradiction, there exists $\{v_n\}\subset\mathcal{N}_a$ such that
\begin{equation}\label{6-1}
0\leq \|\nabla f(v_n)\|_2^2\leq\|\nabla v_n\|_2^2\rightarrow0, \ \ \text{ as $n\rightarrow\infty$}.
\end{equation}
 Since $v_n\in \mathcal{N}_a$, it follows from $\mathcal{G}(v_n)=0$, \x{A1}, Lemma \ref{lemma2.1}-(8), Lemma \ref{lemma1.1} and \x{6-1} that
 \begin{equation}\label{6-2}
  \aligned
\|\nabla f(v_n)\|_2^2&\leq\|\nabla v_n\|_2^2\\
&\leq\int_{\R^2}[h(f(v_n))f(v_n)-2H(f(v_n))]\text{d}x\\
&\leq\int_{\R^2} h(f(v_n))f(v_n) \text{d}x\\
&\leq\varepsilon\int_{\R^2}|f(v_n)|^6 \text{d}x+C_\varepsilon\int_{\R^2}|f(v_n)|^6\left[\exp{(\zeta |f(v_n)|^{4})}-1\right] \text{d}x\\
&\leq\varepsilon\int_{\R^2}|f(v_n)|^6 \text{d}x+C_\varepsilon\left(\int_{\R^2}|f(v_n)|^{12}\text{d}x\right)^{1/2}\left(\int_{\R^2}\left[\exp{\left(4\zeta |v_n|^{2}\right)}-1\right]\text{d}x\right)^{1/2}\\
&\leq\varepsilon\int_{\R^2}|f(v_n)|^6 \text{d}x+CC_\varepsilon\left(\int_{\R^2}|f(v_n)|^{12}\text{d}x\right)^{1/2}.
\endaligned
 \end{equation}
 By   Lemma \ref{lemma1.0},  for any $v\in \mathcal{S}_a$ and $2< p<+\infty$, we have
 \begin{equation}\label{6-3}
 \aligned
\|f(v)\|^p_p&\leq C^p_{p,2}\|f(v)\|_2^{2}\|\nabla f(v)\|_2^{p-2}.
\endaligned
 \end{equation}
From \x{6-1}, \x{6-2} and \x{6-3}, we have
\begin{equation*}
\|\nabla f(v_n)\|_2^2\leq\varepsilon C^6_{6,2}a\|\nabla f(v_n)\|_2^{4}+CC_\varepsilon (C^{12}_{12,2}a\|\nabla f(v_n)\|_2^{10})^{1/2},
\end{equation*}
 which shows that there exists a positive constant $\hat{\varrho}>0$ such that $\|\nabla f(v_n)\|_2^2\geq \hat{\varrho}$. This contradicts to \x{6-1}.

$(ii)$ It follows from $(H_3)$ that
\begin{equation*}
 \aligned
\Psi(v)&=\Psi(v)-\frac{1}{6}\mathcal{G}(v)\\
 &=\frac{1}{6}\|\nabla   v \|_2^2+\frac{1}{6}\|\nabla  f(v)\|_2^2+\frac{1}{6}\int_{\R^2}\left[h(f(v))f(v)-8H(f(v))\right]\text{d}x\\
 &\geq\frac{1}{6} \varrho>0,\ \ \forall\,v\in\mathcal{N}_a.
 \endaligned
\end{equation*}
From this fact, we know that $\sigma^*(a)>0$.
 \end{proof}

\begin{lemma}\label{lemma6.7}
 Assume that $(H_1)$-$(H_4)$ and $(h_3)$ hold.  Then
 the function $a\mapsto \sigma(a)$ is continuous at each $a>0$.
 \end{lemma}
 \begin{proof}
The proof is similar to Lemma \ref{lemma3.10}. So we omit it.
\end{proof}

In the following, we shall estimate the level of minimax level with critical growth.
To this aim, for any $p>6$, we consider the following auxiliary functional $\mathcal{I} :\mathcal{S}_a\rightarrow\R$:
\begin{equation*}
  \mathcal{I} (v)= \frac{1}{2}\int_{\R^2}|\nabla v|^2\text{d}x -\frac{1}{p}\int_{\R^2}|f(v)|^p\text{d}x.
\end{equation*}

By Theorem \ref{theorem1}, we know that there exist
\begin{equation*}
(\hat{v}_a,\hat{\lambda}_a)\in \mathcal{S}_a\times(-\infty,0)
\end{equation*}
 and $m^*(a)>0$ such that
\begin{equation*}
\mathcal{I}'(\hat{v}_a)-\hat{\lambda}_af(\hat{v}_a)f'(\hat{v}_a)=0
\end{equation*}
and
\begin{equation*}
m^*(a)=\mathcal{I}(\hat{v}_a).
\end{equation*}
Similar to Lemma \ref{lemmas2.1}, we know that
\begin{equation}\label{6-16}
p\left(\frac{1}{2}\int_{\R^2}|\nabla \hat{v}_a|^2\text{d}x-m^*(a)\right)=\int_{\R^2}|f(\hat{v}_a)|^p\text{d}x
\end{equation}
and
\begin{equation}\label{6-17}
\int_{\R^2}|\nabla \hat{v}_a|^2\text{d}x+\int_{\R^2}\frac{2f^2(\hat{v}_a)}{1+2f^2(\hat{v}_a)}|\nabla \hat{v}_a|^2\text{d}x=\frac{p-2}{p}\int_{\R^2}|f(\hat{v}_a)|^p\text{d}x.
\end{equation}
From \x{6-16} and \x{6-17}, thus
 \begin{equation*}
 \frac{p-6}{4}\int_{\R^2}|f(\hat{v}_a)|^p\text{d}x+\frac{p}{4}\int_{\R^2} |\nabla f(\hat{v}_a)|^2 \text{d}x=pm^*(a),
 \end{equation*}
 which shows  that
  \begin{equation}\label{6-18}
\int_{\R^2}|f(\hat{v}_a)|^p\text{d}x\leq \frac{4p}{p-6}m^*(a).
 \end{equation}

In order to get the compactness, we need to get a more precise information about the mini-max level  as follows:
\begin{lemma}\label{lemma6.8}
Assume that $(\mathcal{H}_1)$  and $(H_2)$ hold. Then
\begin{equation*}
\sigma^*(a)<\frac{\pi}{\zeta_0}.
\end{equation*}
\end{lemma}
\begin{proof}
By $(H_2)$, we have
\begin{equation*}
 \aligned
\Psi((\hat{v}_a)_t)&=\frac{t^2}{2}\int_{\R^2}\frac{1+2t^2f^2(\hat{v}_a)}{1+2f^2(\hat{v}_a)}|\nabla \hat{v}_a|^2\text{d}x- t^{-2}\int_{\R^2}H(tf(\hat{v}_a))\text{d}x\\
&\leq\frac{t^2}{2}\int_{\R^2}\frac{1 }{1+2f^2(\hat{v}_a)}|\nabla \hat{v}_a|^2\text{d}x+\frac{t^4}{2}\int_{\R^2}\frac{2 f^2(\hat{v}_a)}{1+2f^2(\hat{v}_a)}|\nabla \hat{v}_a|^2\text{d}x-\xi t^{p-2}\int_{\R^2}|f(\hat{v}_a)|^p\text{d}x.
 \endaligned
\end{equation*}
Let
\begin{equation*}
\Upsilon(t)=\frac{t^2}{2}\int_{\R^2}\frac{1 }{1+2f^2(\hat{v}_a)}|\nabla \hat{v}_a|^2\text{d}x+\frac{t^4}{2}\int_{\R^2}\frac{2 f^2(\hat{v}_a)}{1+2f^2(\hat{v}_a)}|\nabla \hat{v}_a|^2\text{d}x-\xi t^{p-2}\int_{\R^2}|f(\hat{v}_a)|^p\text{d}x.
\end{equation*}
Then $\Upsilon(t)\rightarrow0$ as $t\rightarrow0^+$ and $\Upsilon(t)\rightarrow-\infty$ as $t\rightarrow+\infty$, since $p>6$.
If
\begin{equation*}
\xi>\frac{\int_{\R^2}|\nabla \hat{v}_a|^2\text{d}x}{2\int_{\R^2}|f(\hat{v}_a)|^p\text{d}x},
\end{equation*}
then $\Upsilon(1)<0$ and the
maximum value of $\Upsilon(t)$ can be obtained on $[0,1]$. Thus by \x{6-17} and \x{6-18}, one has
\begin{equation*}
 \aligned
\max_{t\geq0}\Upsilon(t)&\leq\max_{t\in[0,1]}\left[\frac{t^2}{2}\int_{\R^2} |\nabla \hat{v}_a|^2\text{d}x -\xi t^{p-2}\int_{\R^2}|f(\hat{v}_a)|^p\text{d}x \right]\\
&\leq\max_{t\in[0,1]}\left[\frac{(p-2)t^2}{2p}  -\xi t^{p-2} \right]\int_{\R^2}|f(\hat{v}_a)|^p\text{d}x\\
&\leq\frac{4p}{p-6}m^*(a)\max_{t\in[0,1]}\left[\frac{(p-2)t^2}{2p}  -\xi t^{p-2} \right]\\
&:=\frac{4p}{p-6}m^*(a)\max_{t\in[0,1]}\Lambda(t).
  \endaligned
\end{equation*}
By a simple computation, we get
\begin{equation*}
\Lambda_{\max}=\Lambda\left[\left(\frac{1}{p\xi}\right)^{\frac{1}{p-4}}\right]=\frac{p-4}{2p}\left(\frac{1}{p\xi}\right)^{\frac{2}{p-4}}.
\end{equation*}
This shows that
\begin{equation*}
\max_{t\geq0}\Upsilon(t)\leq\frac{2(p-4)}{p-6}\left(\frac{1}{p\xi}\right)^{\frac{2}{p-4}}m^*(a)\rightarrow0\ \ \text{as $\xi\rightarrow+\infty$}.
\end{equation*}
Thus if
\begin{equation*}
 \aligned
\xi>\xi^*=\max\left\{ \frac{1}{p}\left[\frac{2\xi_0(p-4)m^*(a)}{\pi(p-6)}\right]^{\frac{p-4}{2}},\frac{\int_{\R^2}|\nabla \hat{v}_a|^2\text{d}x}{2\int_{\R^2}|f(\hat{v}_a)|^p\text{d}x}\right\},
 \endaligned
\end{equation*}
then by \x{shsss4}, we get
\begin{equation*}
\sigma^*(a)<\frac{\pi}{\zeta_0}.
\end{equation*}
The proof is completed.
\end{proof}

\begin{lemma}\label{lemma6.9}

Assume that  $(H_1)$-$(H_4)$ and $(h_3)$ hold. Then

 (i) the function $a\mapsto \sigma^*(a)$ is nonincreasing on $(0,+\infty)$;

(ii) $\sigma^*(a)>\sigma^*(\tilde{a})$ for  all $\tilde{a}>a$, if $\sigma^*(a)$ is achieved;

(iii) $\sigma^*(a)$ is achieved.
\end{lemma}
\begin{proof}
The proof of (i) and (ii) is similar to Lemma \ref{lemma3.10}.
We only need to prove (iii).  From Lemma \ref{lemma6.6} and Lemma \ref{lemma6.04}, we have $\mathcal{N}_a\neq\emptyset$ and $\sigma^*(a)>0$.
 Assume that $\{v_n\}\subset \mathcal{N}_a$ such that $\Psi(v_n)\rightarrow\sigma^*(a)$. Similar to the proof of Lemma \ref{lemma6.6},  it follows from $\mathcal{G}(v_n)=0$ and $\Psi(v_n)\rightarrow\sigma^*(a)$ that
 $\{\|\nabla v_n\|_2^2\}$ is bounded. By Lemma \ref{lemma2.3}(iv), up to a subsequence, there exists $v_0\in H^1(\R^N)$ such that $v_n\rightharpoonup v_0$ in $H^1_r(\R^2)$, $v_n\rightarrow v_0$ in $L^r(\R^N)$ for all $r\in(2,+\infty)$, and then $v_n\rightarrow v_0$ a.e. on $\R^N$ and so $f(v_n)\rightarrow f(v_0)$ in $L^{\tilde{r}}(\R^N)$ for all $\tilde{r}\in(2,+\infty)$.  Now,  we show that $v_0\neq0$. Suppose that $v_0=0$. By $(H_3)$ and $(H_4)$, we deduce that, for any $\tilde{\tau}>0$, there exists $R_{\tilde{\tau}}>0$ such that
 \begin{equation}\label{6-20}
0<\tilde{\tau}H(t)\leq th(t),\ \ \forall\,|t|\geq R_{\tilde{\tau}}.
 \end{equation}
By Lemma \ref{lemma6.6}, $\Psi(v_n)\rightarrow\sigma^*(a)$, $(H_3)$, \x{6-20} with $\tilde{\tau}=10$ and $(H_3)$,  we infer that
\begin{equation}\label{6-21}
 \aligned
\sigma^*(a)+o_n(1)&=\Psi(v_n)-\frac{1}{6}\mathcal{G}(v_n)\\
 &=\frac{1}{6}\|\nabla   v_n \|_2^2+\frac{1}{6}\|\nabla  f(v_n)\|_2^2+\frac{1}{6}\int_{\R^2}\left[h(f(v_n))f(v_n)-8H(f(v_n))\right]\text{d}x\\
 &= \frac{1}{6}\|\nabla   v_n \|_2^2+\frac{1}{6}\|\nabla  f(v_n)\|_2^2+\frac{1}{6}\int_{|f(v_n)|\leq R_{10}}\left[h(f(v_n))f(v_n)-8H(f(v_n))\right]\text{d}x\\
 &\quad\quad\quad\quad\quad\quad+\frac{1}{6}\int_{|f(v_n)|\geq R_{10}}\left[h(f(v_n))f(v_n)-8H(f(v_n))\right]\text{d}x\\
 &\geq\frac{1}{6}\|\nabla   v_n \|_2^2+\frac{1}{6}\|\nabla  f(v_n)\|_2^2+\frac{1}{30}\int_{|f(v_n)|\geq R_{10}}h(f(v_n))f(v_n)\text{d}x.
 \endaligned
\end{equation}
From $(H_1)$, there exists $C>0$ such that
\begin{equation}\label{6-22}
0\leq\frac{1}{6}\int_{|f(v_n)|\leq R_{10}}\left[h(f(v_n))f(v_n)-8H(f(v_n))\right]\text{d}x\leq C\int_{\R^2}|f(v_n)|^6\text{d}x\leq C.
\end{equation}
It follows from \x{6-21} and \x{6-22} that there is a constant $C>0$ such that
\begin{equation}\label{6-23}
\int_{\R^2}h(f(v_n))f(v_n)\text{d}x\leq C.
\end{equation}
For any $\varepsilon\in (0,L_0)$, if we choose $B_\varepsilon>\max\left\{\frac{L_0 C}{\varepsilon},t_0\right\}$, then it follows from $(H_4)$, \x{6-22} and \x{6-23} that
\begin{equation}\label{6-24}
 \aligned
\int_{|f(v_n)|\geq B_\varepsilon}H(f(v_n))\text{d}x&\leq L_0\int_{|f(v_n)|\geq B_\varepsilon}|h(f(v_n))|\text{d}x\\
&\leq\frac{L_0}{B_\varepsilon}\int_{|f(v_n)|\geq B_\varepsilon}h(f(v_n))f(v_n)\text{d}x\\
&<\varepsilon.
 \endaligned
\end{equation}
By $(\mathcal{H}_1)$ and  $(H_1)$, there exists $C_\varepsilon>0$ such that
\begin{equation}\label{6-25}
\int_{|f(v_n)|\leq B_\varepsilon}H(f(v_n))\text{d}x\leq C_\varepsilon\int_{\R^2}|f(v_n)|^6\text{d}x=o_n(1).
\end{equation}
In view of \x{6-24} and \x{6-25}, we have
\begin{equation}\label{6-26}
\int_{\R^2}H(f(v_n))\text{d}x=o_n(1)
\end{equation}
By using $\Psi(v_n)\rightarrow\sigma^*(a)$ again, we have
\begin{equation}\label{6-27}
\|\nabla v_n\|^2_2=2\Psi(v_n)+2\int_{\R^2}H(f(v_n))\text{d}x=2\sigma^*(a)+o_n(1).
\end{equation}
Similar to \x{6-25}, we have
\begin{equation}\label{6-29}
\int_{|f(v_n)|\leq1}h(f(v_n))f(v_n)\text{d}x\leq C \int_{\R^2}|f(v_n)|^6\text{d}x=o_n(1).
\end{equation}
By Lemma \ref{lemma6.8} and \x{6-27}, for any $a>0$, there exists $\hat{\epsilon}>0$ such that
\begin{equation}\label{6-28}
\|\nabla v_n\|^2_2\leq\frac{2\pi(1-3\hat{\epsilon})}{\zeta_0}<\frac{2\pi}{\zeta_0}\ \ \text{for large $n$}.
\end{equation}
Moreover, for above $\hat{\epsilon}>0$, we can choose $q\in(1,2)$ such that
\begin{equation}\label{6-30}
\frac{(1+\hat{\epsilon})(1-3\hat{\epsilon})q}{1-\hat{\epsilon}}<1,
\end{equation}
and by $(\mathcal{H}_1)$, one has
\begin{equation}\label{6-31}
|h(f(v_n))|^q\leq C\left[\exp(\zeta_0(1+\hat{\epsilon})q|f(v_n)|^4)-1\right],\ \ \ \text{for $|f(v_n)|\geq1$}.
\end{equation}
From H\"{o}lder inequality, Lemma \ref{lemma2.1}-(8) and  \x{6-28}-\x{6-31}, we get
\begin{equation}\label{6-32}
 \aligned
\left|\int_{|f(v_n)|\geq 1}h(f(v_n))f(v_n)\text{d}x\right|&\leq\left(\int_{|f(v_n)|\geq 1}|h(f(v_n))|^q\text{d}x\right)^\frac{1}{q}\|f(v_n)\|_{\frac{q}{q-1}}\\
&\leq C\left(\int_{|f(v_n)|\geq 1}\left[\exp(2\zeta_0(1+\hat{\epsilon})q|v_n|^2)-1\right]\text{d}x\right)^\frac{1}{q}\|f(v_n)\|_{\frac{q}{q-1}}\\
&\leq C\|f(v_n)\|_{\frac{q}{q-1}}=o_n(1).
 \endaligned
\end{equation}
Then it follows from \x{6-29}, \x{6-32} and \x{6-26}  that
 \begin{equation*}
\int_{\R^2}[h(f(v_n))f(v_n)-2H(f(v_n))]\text{d}x\rightarrow0\ \ \text{as}\ \ n\rightarrow\infty,
 \end{equation*}
which, together with $\mathcal{G}(v_n)=0$, imply that
\begin{equation*}
\aligned
\int_{\R^2}|\nabla v_n|^2\text{d}x+\int_{\R^2}\frac{2f^2(v_n)}{1+2f^2(v_n)}|\nabla v_n|^2\text{d}x\rightarrow0\ \ \text{as}\ \ n\rightarrow\infty.
\endaligned
\end{equation*}
Based on these facts, we infer that $\Psi(v_n)\rightarrow0$. This contradicts $\sigma^*(a)>0$. So $v_0\neq0$.
Now, we  claim
$\mathcal{G}(v_0)\leq0$. In  fact, since $\mathcal{G}(v_n)=0$, by Fatou's lemma, we get this claim.

Since $v_0\neq0$ and $a_0:=\|f(v_0)\|_2^2\leq\liminf\limits_{n\rightarrow\infty}\|f(v_n)\|_2^2=a$, by Lemma \ref{lemma6.04}, there exists $\bar{t}_0>0$ such that $(v_0)_{\bar{t}_0}\in\mathcal{N}_{a_0}$. In view of Remark \ref{remark6.2} and Lemma \ref{lemma6.4},  we have
\begin{equation*}
 \aligned
\sigma^*(a)&=\liminf_{n\rightarrow\infty}\left[\Psi(v_n) -\frac{1}{4}\mathcal{G}(v_n)\right]\\
&=\liminf_{n\rightarrow\infty}\bigg\{\frac{1}{4}\|\nabla  f(v_n)\|_2^2+\frac{1}{4}\int_{\R^2}\left[h(f(v_n))f(v_n)-6H(f(v_n))\right]\text{d}x\bigg\}\\
&\geq\frac{1}{4}\|\nabla  f(v_0)\|_2^2+\frac{1}{4}\int_{\R^2}\left[h(f(v_0))f(v_0)-6H(f(v_0))\right]\text{d}x\\
&=\Psi(v_0)-\frac{1}{4}\mathcal{G}(v_0)\\
&\geq\Psi((v_0)_{\bar{t}_0})-\frac{\bar{t}_0^{4}}{4}\mathcal{G}(v_0)\\
&\geq\sigma^*(a_0)\geq\sigma^*(a),
  \endaligned
\end{equation*}
since $a_0\leq a$. Thus $\Psi(\bar{v}_0)=\sigma^*(a)=\sigma^*(a_0)$, $\mathcal{G}(v_0)=0$, and $a_0=\|f(v_0)\|_2^2$. In addition, if $a_0<a$, then it follows from $(ii)$ that $\sigma^*(a)<\sigma^*(a_0)$, which is a contradiction.
This shows that $a_0=a$. Thus $\sigma^*(a)$ is achieved.
\end{proof}

\par\noindent
{\bf Proof of Theorem \ref{theorem2}.} By Lemma \ref{lemma6.9} and a similar argument as Lemma \ref{lemma3.11}, we can complete the proof.   $\hfill\Box$

\vskip4mm
\subsection*{Acknowledgments}
The research of J. Chen and J. Sun was supported by National Natural Science Foundation of China (Grant Nos. 12361024 and 11901276),
and partly by supported by Jiangxi Provincial Natural Science Foundation (Nos. 20252BAC250127, 20242BAB23002, 20242BAB25001, 20232ACB211004 and 20232BAB201001).
The research of V.D. R\u{a}dulescu was supported by a grant of the Romanian Ministry of Research,
Innovation and Digitalization (MCID), project ``Nonlinear Differential Systems in Applied
Sciences", within PNRR-III-C9-2022-I8/22.
The research of J. Zhang was supported by the National Natural Science Foundation of
China (12271152), the Key project of Scientific Research Project of Department of
Education of Hunan Province (24A0430). This research turned into supported by the AGH University of Krakow under grant no. 16.16.420.054, funded by the Polish Ministry of Science and Higher Education.



\begin{thebibliography}{99}

\bibitem{MA}
{\small M. Agueh, Sharp Gagliaedo-Nirenberg inequalities via $p$-Laplacian type equations, NoDEA Nonlinear Differential Equation Appl. 15 (2008) 457-472.}


\bibitem{BN}
{\small F. G. Bass, N. N. Nasanov, Nonlinear electromagnetic-spin waves, Phys. Rep. 189 (1990) 165-223.}

\bibitem{BL}
{\small H. Berestycki, P. Lions, Nonlinear scalar field equations, II: existence of ground state, infinitely many solutions, Arch. Rational Mech. Anal. 82 (1983) 347-375.}


\bibitem{CL1}
{\small M. Colin, L. Jeanjean, Solutions for a quasilinear Schr\"{o}dinger equation: a dual approach, Nonlinear Anal. 56 (2004)
213-226.}

\bibitem{CJS}
{\small M. Colin, L. Jeanjean, M. Squassina, Stability and instability results for standing waves of quasi-linear Schr\"{o}dinger equations, Nonlinearity, 23 (2010) 1353-1385.}

\bibitem{SC1}
{\small S. Cuccagna, On instability of excited states of the nonlinear Schr\"{o}dinger equation, Phys. D 238 (2009) 38-54.}

\bibitem{CTC}
{\small J. H. Chen, X. H. Tang, B. T. Cheng, Existence of ground state solutions for a class of quasilinear Schr\"{o}dinger equations with general critical nonlinearity, Comun. Pure Appl. Anal. 18 (2019) 493-517.}

\bibitem{CTC1}
{\small J. H. Chen, X. H. Tang, B. T. Cheng, Existence of ground state solutions for quasilinear Schr\"{o}dinger equations with super-quadratic condition, Appl. Math. Lett. 79 (2018) 27-33.}

\bibitem{CTC2}
{\small
S. T. Chen, V. D. R\v{a}dulescu, X. H. Tang, B. L. Zhang, Ground state solutions for quasilinear Schr\"{o}dinger equations with variable potential and superlinear reaction,
 Rev. Mat. Iberoam. 36 (2020) 1549-1570.}



\bibitem{Chen1}
{\small
S. T. Chen, X. H. Tang, Normalized solutions for nonautonomous Schr\"{o}dinger equations on a suitable manifold, J. Geom. Anal. 30 (2020) 1637-1660.}


\bibitem{Chen2}
{\small
S. T. Chen, V. D. R\v{a}dulescu, X. H. Tang, S. Yuan, Normalized solutions for  Schr\"{o}dinger equations with critical expoential growth in $\R^2$, SIAM J. Math. Anal. 55 (2023) 7704-7740.}


\bibitem{Chen3}
{\small
S. T. Chen,  X. H. Tang, Normalized solutions for  Kirchhoff equations with Sobolev critical expoent  growth and mixed nonlinearities,  Math. Ann. 391 (2024) 2783-2836.}

\bibitem{Chen4}
{\small
S. T. Chen,  X. H. Tang,  Berestycki-Lions conditions on ground state solutions for a nonlinear Schr\"{o}dinger equation with variable potentials,  Adv. Nonlinear Anal. 9 (2020) 496-515.}


\bibitem{CAO}
{\small D. M. Cao, Nontrivial solution of semilinear elliptic equation with critical exponent in $\R^2$, Comm. Partial Differ. Equ. 17 (1992)
407-435.}


\bibitem{DMS2}
{\small J. M. do \'{O}, A.  Moameni,
Solitary waves for quasilinear Schr\"{o}dinger equations arising in plasma physics, Adv. Nonlinear Stud. 9 (2009)  479-497.}

\bibitem{JS}
{\small J. M. do \'{O}, U. Severo, Solitary waves for a class of quasilinear Schr\"{o}dinger equations in dimension two, Calc. Var.
Partial Differ. Equ. 38 (2010) 275-315.}

\bibitem{DMS}
{\small J. M. do \'{O}, U. Miyagaki, S. Soares, Soliton solutions for quasilinear Schr\"{o}dinger
equations with critical growth, J. Differ. Equ. 248 (2010) 722-744.}

\bibitem{DMS1}
{\small J. M. do \'{O}, O. Miyagaki, S. Soares, Soliton solutions for quasilinear Schr\"{o}dinger equations: The critical exponential case, Nonlinear Anal. 67 (2007) 3357-3372.}

\bibitem{JS1}
{\small J. M. do \'{O}, A.  Moameni, U. Severo, Semi-classical states for quasilinear Schr\"{o}dinger equations arising in plasma physics, Commun. Contemp. Math. 11 (2009) 547-583.}


\bibitem{JS0}
{\small J. M. do \'{O}, $N$-Laplacian equations in $\R^N$ with critical growth, Abstr. Appl. Anal. 2 (1997) 301-315.}


\bibitem{FMR}
{\small D. G. Figueiredo, O. H. Miyagaki, B. Ruf, Elliptic equations in $\R^2$ with nonlinearities in the critical growth range, Calc. Var. Partial Differ. Equ. 3 (1995) 139-153.}

\bibitem{LLW5}
{\small R. Feola, B. Gr\'{e}bert,  F. Iandoli, Long time solutions for quasilinear Hamiltonian perturbations of Schr\"{o}dinger and Klein–Gordon equations on tori, Anal. PDE. 16(2023) 1133-1203.}


\bibitem{FS}
{\small X. Fang, A. Szulkin, Multiple solutions for a quasilinear Schr\"{o}dinger equation, J. Differ. Equ. 254 (2013) 2015-2032.}

\bibitem{HQZ}
{\small X. He, A. Qian, W. Zou, Existence and concentration of positive solutions for quasilinear Schr\"{o}dinger equations with critical growth, Nonlinearity 26 (2013) 3137-3168.}


\bibitem{JSLa}
{\small
L. Jeanjean, J. Jendrej, T. Le, N. Visciglia, Orbital stability of ground states for a Sobolev critical Schr\"{o}dinger equations, J. Math. Pures Appl. 9 (2022) 158-179. }


\bibitem{JSL}
{\small
L. Jeanjean, S. Lu, A mass supercritical problem revisited, Calc. Var. Partial Differ. Equ. 59 (2023) 174. }

\bibitem{JLW}
{\small
L. Jeanjean, T Luo, Z. Q. Wang, Multiple normalized solutions for quasi-linear Schr\"{o}dinger equations, J. Differ. Equ. 259 (2015) 3894-3928.}

\bibitem{LJ}
{\small L. Jeanjean, T. Luo, Sharp nonexistence results of prescribed $L^2$-norm solutions for some class of Schr\"{o}dinger-Poisson and quasi-linear equations, Z. Angew. Math. Phys. 64 (2013) 937-954.}


\bibitem{Jean}
{\small
L. Jeanjean, Existence of solutions with prescribed norm for semilinear elliptic equations, Nonlinear Anal. 28 (1997) 1633-1659.}


\bibitem{JZZ}
{\small
L. Jeanjean, J. Zhang, X. Zhong, Existence and limiting profile of energy ground states quasi-linear Schr\"{o}dinger equations: mass super-critical case, arXiv:2501.03845v1.}



\bibitem{OK}
{\small O. Kavian  Introduction \`{a} la Th\`{e}orie des Points Critiques et Applications aux Problèmes Elliptiques.
Springer, Paris (1993)}


\bibitem{SK}
{\small S. Kurihara, Large-amplitude quasi-solitons in superfluid films, J. Phys. Soc. Japan 50 (1981) 3262-3267.}


\bibitem{LW}
{\small J. Q. Liu, Z. Q. Wang, Soliton solutions for quasilinear Schr\"{o}dinger equations, I,
Proc. Amer. Math. Soc. 131 (2003) 441-448.}


\bibitem{LWW}
{\small J. Q. Liu, Y. Wang, Z. Q. Wang, Soliton Solutions for quasilinear Schr\"{o}dinger equations, II, J. Differ. Equ. 187 (2003)
473-493.}


\bibitem{LWW1}
{\small J. Q. Liu, Y. Wang, Z. Q. Wang, Solutions for quasilinear Schr\"{o}dinger equations via the Nehari method,
Comm. Partial Differ. Equ. 29 (2004) 879-901.}


\bibitem{LLW1}
{\small X. Q. Liu, J. Q. Liu, Z. Q. Wang, Ground states for quasilinear Schr\"{o}dinger
equations with critical growth, Calc. Var. Partial Differ. Equ. 46 (2013)
641-669.}

\bibitem{LLW2}
{\small X. Q. Liu, J. Q. Liu, Z. Q. Wang, Quasilinear elliptic equations with critical growth via perturbation method,
Proc. Amer. Math. Soc. 254 (2013) 102-124.}


\bibitem{LLW6}
{\small X. Q. Liu, J. Q. Liu, Z. Q. Wang,  Localized nodal solutions for quasilinear Schr\"{o}dinger equations, J. Differ. Equ. 267 (2019) 7411-7461.}

\bibitem{LLW3}
{\small J. Q. Liu, X. Q. Liu, Z. Q. Wang, Multiple sign-changing solutions for quasilinear elliptic equations via perturbation method, Comm. Partial Differ. Equ.  39 (2014) 2216-2239.}

\bibitem{LLW4}
{\small
S. Liu, J. Zhou, Standing waves for quasilinear Schr\"{o}dinger equations with indefinite potentials, J. Differ. Equ. 265 (2018) 3970-3987.}

\bibitem{LZ}
{\small H. Li, W. Zou,  Quasilinear Schr\"{o}dinger equations: ground state and infinitely many normalized solutions, Pacific  J.  Math. (2023) 99-138.}



\bibitem{AMI1}
{\small A. Moameni, On a clss of periodic quasilinear  Schr\"{o}dinger equations involving critical growth in $\R^2$, J. Math. Anal. Appl. 334 (2007) 775-786.}



\bibitem{AM2}
{\small A. Moameni, On the existence of standing wave solutions to quasilinear Schr\"{o}dinger equations,
Nonlinearity 19 (2006) 937-957.}


\bibitem{MF}
{\small V. G. Makhankov, V. K. Fedyanin, Nonlinear effects in quasi-one-dimensional models and condensed matter theory, Phys. Rep. 104 (1984) 1-86.}


\bibitem{SV}
{\small E. A. Silva, G. F. Vieira, Quasilinear asymptotically periodic Schr\"{o}dinger equations with critical growth, Calc. Var. Partial Differ. Equ.  39 (2010) 1-33.}

\bibitem{WZ}
{\small Y. Wang, W. Zou, Bound states to critical quasilinear
Schr\"{o}dinger equations, Nonlinear Differ. Equ. Appl. 19 (2012)  19-47.}

\bibitem{WW}
{\small
K. Wu, X. Wu.  Radial solutions for quasilinear Schr\"{o}dinger equations without 4-superlinear condition, Appl. Math. Lett. 76 (2018) 53-59.}

\bibitem{XW}
{\small X. Wu, Multiple solutions for quasilinear Schr\"{o}dinger equations
with a parameter, J. Differ. Equ.256 (2014) 2619-2632.}

\bibitem{MW}
{\small M. Willem, Minimax theorems. Birkh\"{a}user, Berlin (1996).}

\bibitem{YY}
{\small H. Ye, Y. Yu, The existence of normalized solutions for $L^2$-critical quasilinear Schr\"{o}dinger equations, J. Math. Anal. Appl.  497 (2021)  124839.}

\bibitem{MIW}
{\small M. I. Weinstein, Nonlinear Schr\"{o}dinger equations and sharp interpolation estimates,
 Commun. Math. Phys. 87(1983)  567-576.}

\bibitem{ZLW}
{\small L. Zhang, Y. Li, Z. Q. Wang, Multiple normalized solutions for a quasilinear Schr\"{o}dinger equations via dual approach, Topol. Math. Nonlinear Anal.  61 (2023) 465-489.}

\bibitem{ZTAW}
{\small J. Zhang, X. H. Tang, W. Zhang, Infinitely many solutions of quasilinear Schr\"{o}dinger equation with sign-changing potential, J. Math. Anal.  Appl. 420 (2014) 1762-1775.}

\end{thebibliography}
\end{document}